\documentclass[12pt, reqno]{amsart}

\usepackage{array,amssymb,amsxtra,latexsym,graphicx,psfrag,epsfig,ifthen,mathtools,mathrsfs}
\usepackage[all]{xy}
\SelectTips {cm}{}
\usepackage{bbm,bm}
\usepackage{ulem}
\usepackage{rotating}
\usepackage{float}
\usepackage{hyperref}
\usepackage{cleveref}

\usepackage[backend=biber]{biblatex}
\AtEveryBibitem{\normalem}
\AtBeginBibliography{\sloppy}\usepackage{microtype}

\usepackage{rotating}
\usepackage{tikz,graphicx}
\usetikzlibrary{math} 
\usepackage{subcaption}
\usepackage{tikz-3dplot}
\usetikzlibrary{matrix,shapes,arrows,fit,calc,math,shadings}
\usetikzlibrary{decorations.markings,fpu}
\usetikzlibrary{arrows.meta}

\usepackage{todonotes}

\newcommand{\defn}[1]{{\color{blue!50!black}\textbf{\emph{#1}}}}

\usepackage{xcolor}

\definecolor{blue}{rgb}{0, 0.445, 0.695}
\definecolor{bluishgreen}{rgb}{0, 0.626, 0.456}
\definecolor{red}{rgb}{0.896, 0.395, 0}
\definecolor{purple}{rgb}{0.783, 0.464, 0.640}
\definecolor{skyblue}{rgb}{0.359, 0.752, 0.973}
\definecolor{orange}{rgb}{0.999, 0.706, 0.0}
\definecolor{yellow}{rgb}{0.937, 0.890, 0.258}

\definecolor{olive}{RGB}{116,141,19}
\definecolor{green}{RGB}{108,208,48}
\definecolor{teal}{RGB}{47,77,62}
\definecolor{turquoise}{RGB}{86,235,211}
\definecolor{lightblue}{RGB}{150,178,153}
\definecolor{blue2}{RGB}{25,50,191}
\definecolor{indigo}{RGB}{142,128,251}
\definecolor{indigo2}{RGB}{114,32,246}
\definecolor{lightpurple}{RGB}{243,197,250}
\definecolor{purple2}{RGB}{105,66,131}
\definecolor{magenta}{RGB}{206,43,188}
\definecolor{brown}{RGB}{110,57,13} 
\definecolor{darkblue}{rgb}{0.0, 0.0, 0.7}

\definecolor{mygreen}{rgb}{0,.4,0}
\definecolor{myblue}{rgb}{0,0,.5}
\definecolor{mymagenta}{cmyk}{0,.6,0,0}

\newcommand{\set}[1]{\left\lbrace#1\right\rbrace}
\newcommand{\Top}{\operatorname{Top}}
\newcommand{\Bot}{\operatorname{Bot}}
\newcommand{\h}{\operatorname{h}}
\newcommand{\asc}{\operatorname{asc}}
\newcommand{\desc}{\operatorname{desc}}
\newcommand{\Asc}{\operatorname{Asc}}
\newcommand{\Desc}{\operatorname{Desc}}

\usepackage[newitem,newenum,neverdecrease]{paralist}
\makeatletter
\if@plflushright
  
\else
  
\fi
\renewcommand{\@asparaenum@}{%
  \expandafter\list\csname label\@enumctr\endcsname{%
    \usecounter{\@enumctr}%
    \labelwidth\z@
    \labelsep.5em
    \leftmargin\z@
    \parsep\parskip
    \itemsep\z@
    \topsep\z@
    \partopsep\parskip
    \itemindent\parindent
    \advance\itemindent\labelsep
    \def\makelabel##1{\upshape ##1}}}
\makeatother
\usepackage{psfrag}
\usepackage{longtable}
\usepackage{shuffle}

\theoremstyle{plain}
\newtheorem{theorem}{Theorem}[section]
\newtheorem{proposition}[theorem]{Proposition}
\newtheorem{lemma}[theorem]{Lemma}
\newtheorem{corollary}[theorem]{Corollary}

\newtheorem*{theorem*}{Theorem}

\theoremstyle{definition}
\newtheorem{definition}[theorem]{Definition}

\newtheorem{example}[theorem]{Example}
\newtheorem{remark}[theorem]{Remark}

\numberwithin{equation}{section}

\newcommand\bbR{\mathbb{R}}

\newcommand\calC{\mathcal{C}}

\newcommand\calF{\mathcal{F}}

\newcommand\scrI{\mathscr{I}}
\newcommand\scrL{\mathscr{L}}

\newcommand\scrO{\mathscr{O}}

\newcommand\In{\mathrm{In}}
\newcommand\Out{\mathrm{Out}}
\newcommand\precccwrot{\prec_{\mathrm{rot}}^{\mathrm{ccw}}}
\newcommand\leqccwrot{\leq_{\mathrm{rot}}^{\mathrm{ccw}}}
\newcommand\cw{\mathrm{cw}}
\newcommand\ccw{\mathrm{ccw}}

\newcommand{\oru}[1]{\mathrm{oru}(#1)}

\begin{document}
\title{Framingtopes}

\author[Fernandez de soto]{Sergio Alejandro Fernandez de Soto Guerrero}
\address[S. Fernandez de soto]{Institute of Discrete Mathematics, Technische Universität Graz, Austria, Graz, Austria}
\email{sergio.fernandez@tugraz.at}

\author[Ceballos]{Cesar Ceballos}
\address[C. Ceballos]{Technische Universit\"at Graz}
\email{cesar.ceballos@tugraz.at}

\author[von Bell]{Matias von Bell}
\address[M. von Bell]{Indiana University Southeast}
\email{mvonbell@iu.edu}

\date{\today} 

\begin{abstract}


Framing lattices arise from the dual graphs of framed (or DKK) triangulations of flow polytopes and provide a common framework encompassing classical lattices such as the Boolean, Tamari, and weak-order lattices, as well as $\tau$-tilting posets of certain gentle algebras. In this paper, we introduce the \emph{framingtope}, a polytopal complex that provides a geometric counterpart to a framing lattice: its edge graph is the Hasse diagram of the framing lattice. We prove that the framingtope admits three equivalent descriptions, in terms of interior faces of the framed triangulation, sets of pairwise coherent routes covering the graph, and pure intervals of the framing lattice. We further construct a tropical realization of the framingtope as the bounded-cell complex of an arrangement of tropical hypersurfaces associated with an admissible height function. This construction yields explicit vertex coordinates for broad classes of framed graphs, including plane framed graphs and multioruga graphs. In the multioruga case, these coordinates give tropical realizations of weak orders on multipermutations and, in the ordinary oruga case, recover the classical permutahedron.
\end{abstract}
\maketitle
\tableofcontents

\section{Introduction}

Flow polytopes are fundamental objects in combinatorial optimization and discrete geometry, providing a rich interplay between graph theory and polyhedral geometry. Given a directed acyclic graph $G$ and an integer vector of net-flows, the flow polytope~$\mathcal{F}_G$ consists of all nonnegative flows on $G$ satisfying the prescribed net-flow conditions. Among the many geometric and combinatorial structures associated with flow polytopes, their triangulations have attracted particular attention. Danilov, Karzanov, and Koshevoy~\cite{dkktriangulation} introduced a distinguished family of regular unimodular triangulations, known as \emph{DKK triangulations} or \emph{framed triangulations}. These triangulations are determined by a \emph{framing}, namely, a choice of ordering of the incoming and outgoing edges at each vertex of $G$.

A striking feature of framed triangulations is that their dual graphs give rise to a remarkably broad family of lattice structures. Classical examples include the Boolean lattice, the Tamari lattice, and the weak order on permutations. More precisely, the Hasse diagram of each of these lattices occurs as the dual graph of a framed triangulation of a suitable flow polytope. Their geometric realizations, however, have traditionally been studied separately. The Tamari lattice is realized by the associahedron, the weak order on permutations by the permutahedron, and analogous constructions arise for other families of lattices. These examples suggest that the lattice structures arising from framed triangulations should admit a common geometric framework.

A unifying lattice-theoretic framework for these examples was introduced by von Bell and Ceballos \cite{vonBell_framing_2024}. For any directed graph $G$ equipped with a framing $F$, they defined a lattice $\mathscr{L}_{G,F}$, called the \emph{framing lattice}, whose Hasse diagram is the dual graph of the corresponding framed triangulation. Framing lattices are polygonal, semidistributive, and congruence uniform, and encompass many classical and remarkable lattices, including Boolean, Cambrian, Tamari-type, weak-order-type lattices, and $\tau$-tilting posets of certain gentle algebras~\cite{vonBell_framing_2024,BGMY23,BBBHPSY24,d2023realizing}. Thus, framing lattices are not merely a common generalization of several classical combinatorial lattices, but also provide a framework connecting phenomena from polyhedral geometry, lattice theory, and representation theory. This raises a natural geometric question: if framing lattices provide a common combinatorial framework for these examples, what is the corresponding common geometric object?

For the classical examples, the answer is a polytope: the permutahedron realizes the weak order on permutations, while the associahedron realizes the Tamari lattice. In general, however, a single convex polytope is not sufficient. The Hasse diagram of a framing lattice need not be the edge graph of a single convex polytope. Instead, it is naturally realized as the $1$-skeleton of a \emph{polytopal complex}. This leads to the main object of this paper, which may be viewed as the geometric counterpart of a framing lattice: the \emph{framingtope}.

The framingtope is obtained by dualizing all interior faces of a framed triangulation. Its vertices correspond to maximal simplices, and its edges record their adjacencies. Since the dual graph of the triangulation is the Hasse diagram of the framing lattice~$\mathscr{L}_{G,F}$, the edge graph of the framingtope $\mathscr{P}_{G,F}$ is precisely this Hasse diagram. Thus, the framingtope plays the role for framing lattices that the permutahedron and associahedron respectively play for the weak order and Tamari lattices, with the distinction that the framingtope is generally a polytopal complex rather than a polytope. See \Cref{framinglattice_framintope} for an example.

\begin{figure}[ht]
\centering
\includegraphics[]{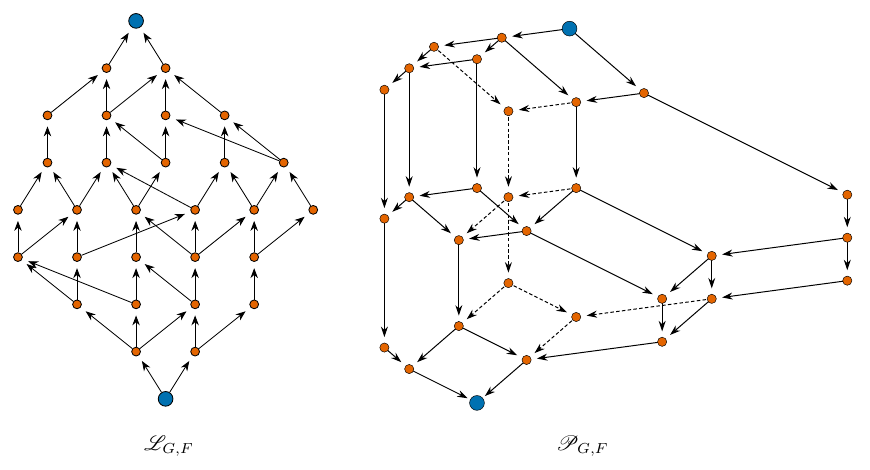}
\caption{A framing lattice and its framingtope.}
\label{framinglattice_framintope}
\end{figure}

The framingtope admits three equivalent descriptions, each emphasizing a different aspect of the construction: its polyhedral origin, its intrinsic route combinatorics, and its lattice-theoretic structure.

\begin{theorem*}[Theorem~\ref{thm.three_equivalent_definitions}]
Let $(G,F)$ be a framed graph. The following three posets are isomorphic and define the combinatorial framingtope $\mathscr{P}_{G,F}$:
\begin{enumerate}
\item the complex of interior faces of the framed triangulation, ordered by reverse inclusion;
\item the collection of sets of pairwise coherent routes covering $G$, ordered by reverse inclusion;
\item the collection of pure intervals of the framing lattice, ordered by inclusion.
\end{enumerate}
\end{theorem*}

The first description connects the framingtope directly to the geometry of the framed triangulation. The second provides an intrinsic description in terms of the route structure of the graph, without reference to the ambient flow polytope. The third reveals the lattice-theoretic content of the construction: the elements (or faces) of the framingtope are naturally indexed by pure intervals of the framing lattice. Each such interval can be realized by a polytope, and these polytopes fit together along their common faces to form the framingtope as a polytopal complex. Thus, the three descriptions provide complementary ways of understanding the same object, linking the geometry of flow polytopes, the combinatorics of routes, and the structure of framing lattices.

The polyhedral construction above interacts naturally with tropical geometry, which provides an explicit geometric realization of the framingtope. The key ingredient is the regularity of framed triangulations. A regular triangulation can be encoded by a height function, and the tropical Cayley trick transforms it into an arrangement of tropical hypersurfaces. Tropical duality then identifies the bounded cells of this arrangement with the interior faces of the corresponding regular triangulation. Consequently, the framingtope can be realized as the polytopal complex of bounded cells in a suitable arrangement of tropical hypersurfaces.

\begin{theorem*}[cf.~\Cref{thm_framingtope_geometric_realization}]
Let $(G,F)$ be a framed graph. The framingtope $\mathscr{P}_{G,F}$ admits a tropical realization as the polytopal complex of bounded cells of an arrangement of tropical hypersurfaces.
\end{theorem*}

This tropical realization is particularly useful, both because it provides a uniform construction for a broad class of examples and because it leads to explicit vertex coordinates for framingtopes. We focus in particular on \emph{plane framed graphs} and on a family of framed graphs that we call \emph{ABBE graphs}. For these graphs, we introduce a simple height function that differs from the one typically used to construct DKK triangulations. We then show that the resulting tropical framingtopes admit explicit coordinate formulas. These formulas are expressed directly in terms of the routes and the planar structure of the graph, providing concrete geometric realizations of the corresponding framingtopes.

As a particularly rich family of examples, we study \emph{multioruga graphs}. Their framing lattices are weak orders on multipermutations, and our construction gives explicit coordinates for the corresponding framingtopes. In the special case of the ordinary oruga graph, the framing lattice is the weak order on permutations, and the resulting framingtope gives a tropical realization of the permutahedron. More generally, multioruga graphs yield tropical realizations of weak orders on multipermutations, extending the classical permutahedral picture to a broader family of lattice structures.

An alternative realization of framing lattices in terms of cubical coordinates was announced in the extended abstract~\cite{berggren_canonical_2026}. However, the authors focus only on the Hasse diagram and do not consider the full dual complex of interior faces of the framed triangulation.

\textbf{Acknowledgements}. The authors were partially supported by the Austrian Science Fund FWF, grants 10.55776/P33278 and 10.55776/I5788.

\section{Framing Lattices}\label{section: framinglat}
Framing lattices were introduced by von Bell and Ceballos in~\cite{vonBell_framing_2024} as a class of posets whose Hasse diagrams are dual to framed triangulations of flow polytopes. In this section, we recall the concepts of flow polytopes, framed triangulations, and framing lattices.

\subsection{Flow polytopes}
Let us start by defining a special kind of polytope associated to a flow graph. A \defn{flow graph} $G$ is a directed acyclic graph on vertex set $V(G) = [n]$ and edge multiset $E(G)$ such that all edges are directed from smaller to larger vertices, and $G$ has a unique source $s=1$ and sink $t=n$. If $e=(i,j)$, we call $h(e)=j$ the \defn{head} of $e$ and $t(e)=i$ the \defn{tail}. A path from the source to the sink is called a \defn{route}.

A \defn{unit flow} on $G$ is a tuple $(x_e)_{e\in E(G)} \in \bbR^{|E(G)|}_{\geq0}$ satisfying
$$\sum_{e \in \Out(j)} x_e - \sum_{e \in \In(j)}  x_e = u_j,$$
where $u_1 = 1$, $u_n = -1$, and $u_j = 0$ for $1 < j < n$; for vertex $j$ of $G$, we denote by $\In(j)$ and $\Out(j)$ the incoming edges and the outgoing edges respectively and $x_e$ the flow associated to the edge $e$. The \defn{flow polytope} of $G$ is the set~$\calF_G$ of unit flows on $G$. An alternative description of the flow polytope $\calF_G$ can be given in terms of its vertices, which can be characterized as the unit flows on $G$ taking value one on the edges of a route and value zero on the remaining edges. Thus,~$\calF_G$ is the convex hull of the indicator vectors of the routes of $G$. The dimension of a flow polytope $\calF_{G}$ is known to be $\dim(\calF_G)=|E(G)| - |V(G)| +1$.

\begin{example}[Running example]\label{exam: oruga3}
    Consider the flow graph $G$ illustrated in~\Cref{fig_oruga3}. It has four vertices $\{1,2,3,4\}$ and six edges
    \[
    e_1=(1,2),\ e_2=(1,2),\ e_3=(2,3),\ e_4=(2,3),\ e_5=(3,4),\ e_6=(3,4)
    \]
    which are oriented from the smaller to the larger vertex. Note that we omit edge orientations in our drawings (depicting them as undirected edges rather than arrows) because the direction is implicitly assumed to be from left to right. See \Cref{fig_oruga3}.

   \begin{figure}[ht]
    \centering
    \includegraphics{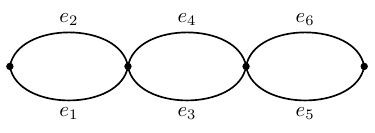}
        \caption{The Oruga graph $G=\oru{3}$.}
        \label{fig_oruga3}
\end{figure}

This graph has eight different routes, whose corresponding indicator vectors are:

    \begin{center}
    \begin{tabular}{cc}
      $\set{e_1,e_3,e_5} \to (1,0,1,0,1,0)$,   & $\set{e_1,e_3,e_6}\to(1,0,1,0,0,1)$ \\
       $\set{e_1,e_4,e_5}\to(1,0,0,1,1,0)$,  & $\set{e_1,e_4,e_6}\to(1,0,0,1,0,1)$  \\
        $\set{e_2,e_3,e_5}\to(0,1,1,0,1,0)$, & $\set{e_2,e_3,e_6}\to(0,1,1,0,0,1)$  \\
       $\set{e_2,e_4,e_5}\to(0,1,0,1,1,0)$,  &  $\set{e_2,e_4,e_6}\to(0,1,0,1,0,1)$ 
    \end{tabular}       
    \end{center}

    The flow polytope $\calF_G\subseteq \mathbb{R}^6$ associated to $G$ is the convex hull of these eight indicator vectors. Although the ambient space is 6-dimensional, the result is a 3-dimensional cube, as illustrated in~\Cref{fig_cube_oruga3}.

\begin{figure}[ht]
    \centering
    \includegraphics{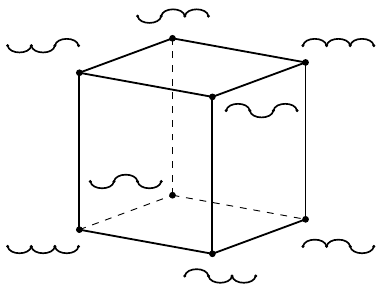}
    \caption{The flow polytope $\calF_G$ of the oruga graph $G=\oru{3}$.}
    \label{fig_cube_oruga3}
\end{figure}

\end{example}
   
\begin{example}[The oruga graph and the cube]  \label{exam_oruga_n}  
    Example~\ref{exam: oruga3} can be generalized as follows. The \defn{oruga graph} $G_n=\oru{n}$ is the oriented graph on the vertex set~$[n+1]$, containing two edges between $i$ and $i+1$ for $i \in [n]$. These two edges are oriented from the smaller vertex to the larger one, and are labeled by $e_{2i-1}$ and $e_{2i}$. Example~\ref{exam: oruga3} corresponds to the case $n=3$. The flow polytope $\calF_{G_n}$ can be shown to be an $n$-dimensional cube in~$\mathbb{R}^{2n}$. Some examples of the oruga graph and its corresponding flow polytopes for small $n$ are illustrated in~\Cref{fig_oruga_graph}. The name ``oruga" was given in \cite{d2023realizing}, meaning caterpillar in Spanish. We will explore this example and a generalization of it, called the multi-oruga graph, in greater detail in~\Cref{sec_multipermutations}.

    \begin{figure}[ht]
    \centering
    \includegraphics{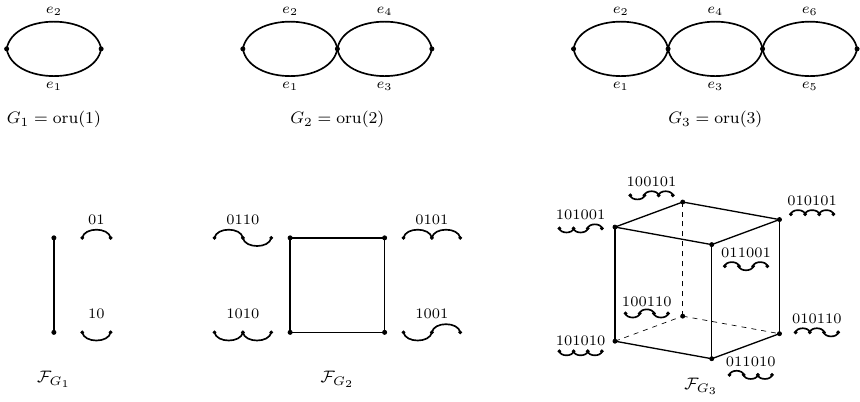}

    \caption{Some examples of the oruga graph and their flow polytopes.}
    \label{fig_oruga_graph}
\end{figure}

\end{example}

\subsection{Framed triangulations}
Given a flow graph $G$, Danilov, Karzanov, and Koshevoy~\cite{dkktriangulation} introduced a family of triangulations of the flow polytope $\calF_G$ which are commonly referred to as \emph{DKK triangulations} or \emph{framed triangulations}. Their main ingredient is the concept of a framing, which we now recall.

For a vertex $v$ of $G$, we denote by $\In(v)$ and $\Out(v)$ the (possibly empty) sets of incoming and outgoing edges at $v$, respectively. A \defn{framing} at the vertex $v$ is a pair of linear orders $(\leq_{\In(v)}, \leq_{\Out(v)})$ on the incoming and outgoing edges at $v$. A \defn{framed graph} $(G,F)$ is a flow graph $G$ with a framing $F$ at every vertex.

\begin{example}[Running example continued]\label{exam: oruga3_framings}
    Let $G=\oru{3}$ be the flow graph in our running Example~\ref{exam: oruga3}. Consider the two different framings $F_1$ and $F_2$ of $G$ illustrated in~\Cref{fig_running_framings}, where the labels indicate the order of the incoming and outgoing edges at every vertex.

\begin{figure}[ht]
    \centering
        \includegraphics{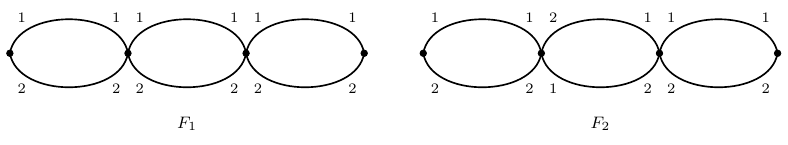}
        \caption{Two different framings $F_1$ and $F_2$ of $G=\oru{3}$.}
        \label{fig_running_framings}
\end{figure}

\begin{figure}[ht]
    \centering
    \includegraphics{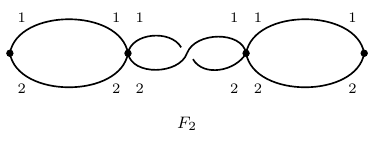}
        \caption{The framing $F_2$ drawn such that the framing is increasing from top to bottom.}
        \label{fig_running_framings_F2}
\end{figure}

\end{example}

For a path $P$ containing a vertex $v$, let $Pv$ (resp. $vP$) denote the maximal subpath of~$P$ ending (resp. beginning) at $v$. Furthermore, let $\scrI(v)$ (resp. $\scrO(v)$) denote the set of paths in~$G$ ending (resp. beginning) at $v$. We define the \defn{relations $\leq_{\scrI(v)}$} and \defn{$\leq_{\scrO(v)}$} on $\scrI(v)$ and~$\scrO(v)$ as follows.

Given paths $Pv, Qv\in\scrI(v)$, let $w\leq v$ be the smallest vertex after which $Pv$ and $Qv$ coincide. If $w$ is the first vertex of $Pv$ or $Qv$, we say that $Pv=_{\scrI(v)} Qv$. Otherwise, let $e_P$ be the edge of $P$ entering $w$ and let $e_Q$ be the edge of $Q$ entering $w$. Then $Pv <_{\scrI(v)} Qv$ if and only if~$e_P <_{\In(w)} e_Q$.

Similarly, for $vP, vQ\in\scrO(v)$, let $w'\geq v$ be the largest vertex before which $vP$ and $vQ$ coincide. If $w'$ is the largest vertex of $vP$ or $vQ$, then $vP =_{\scrO(v)} vQ$. Otherwise, let $e_P'$ be the edge of $P$ leaving $w'$ and let~$e_Q'$ be the edge of $Q$ leaving $w'$. Then $vP <_{\scrO(v)} vQ$ if and only if $e_P' <_{\Out(w')} e_Q'$.

We say that a vertex $v$ of a path $P$ is an \defn{inner vertex} if $v$ is not the first or last vertex of the path. If $v$ is an inner vertex of paths $P$ and $Q$, we say that $P$ and~$Q$ are \defn{incoherent at $v$} if $Pv <_{\scrI(v)} Qv$ and $vQ <_{\scrO(v)} vP$, or if $Qv <_{\scrI(v)} Pv$ and $vP <_{\scrO(v)} vQ$. We say that $P$ and $Q$ are \defn{coherent at $v$} otherwise. Paths $P$ and $Q$ are then said to be \defn{coherent} if they are coherent at each common inner vertex, and they are \defn{incoherent} otherwise. When the graph $G$ is drawn such that the framing of the incoming and outgoing edges increases from top to bottom, the notion of coherence becomes very intuitive. In this case, two paths are coherent if they do not ``cross''.

If a route $R$ is coherent with all the other routes, we say that $R$ is an \defn{exceptional} route. A set of pairwise coherent routes is called a \defn{clique}, and a \defn{maximal clique} is a clique that is maximal under inclusion. We denote by \defn{$\calC$} the \defn{collection of maximal cliques}, and by \defn{$\Delta_C$} the convex hull $\Delta_C$ of the indicator vectors of the routes in~$C$.

\begin{example}[Running example continued]
    Let $G=\oru{3}$ and $F=F_1$ be the first framing in~\Cref{exam: oruga3_framings}; two incoherent routes at vertex $v=2$ are $\set{e_2,e_3,e_6}$ and $\set{e_1,e_4,e_6}$, which are shown as the black and blue routes in~\Cref{fig_Two_incoherent_routes}.

\begin{figure}[ht]
    \centering
    \includegraphics{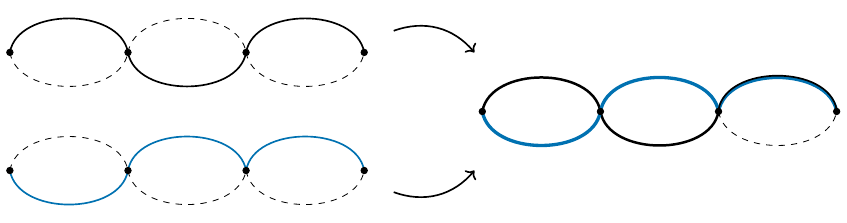}
    \caption{Two incoherent routes for the framing $F_1$ of $G=\oru{3}$.}
    \label{fig_Two_incoherent_routes}
\end{figure}

 With the same framing,~\Cref{fig_max_clique_example} shows an example of a maximal clique $C$ consisting of the routes $\set{e_1,e_3,e_5}$, $\set{e_1,e_4,e_5}$, $\set{e_1,e_4,e_6}$, and $\set{e_2,e_4,e_6}$. The figure also shows the convex hull $\Delta_C$ of their indicator vectors.

\begin{figure}[ht]
    \centering
    \includegraphics{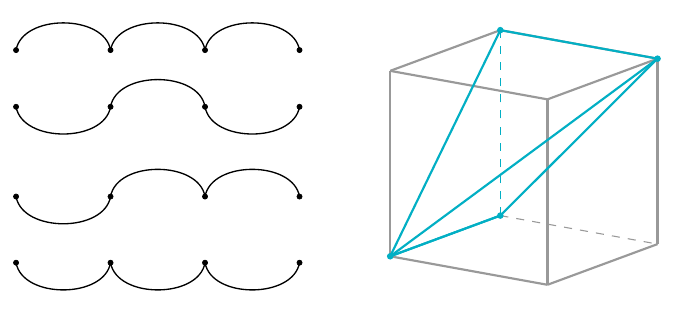}
    \caption{A maximal clique $C$ for the framing $F_1$ of $G=\oru{3}$, and the convex hull $\Delta_C$ of its indicator vectors.}
    \label{fig_max_clique_example}
\end{figure}

\end{example}

Danilov, Karzanov, and Koshevoy showed in~\cite{dkktriangulation} that $|C| = \dim (\mathcal{F}_G) + 1$ for every maximal clique $C$, and that $\Delta_C$ forms a simplex inside the flow polytope $\calF_G$. Putting all these simplices together gives a triangulation of $\calF_G$.

\begin{proposition}[{Danilov, Karzanov and Koshevoy~\cite{dkktriangulation}}]
Let $(G,F)$ be a framed graph. The set $\{\Delta_C \mid C \in \calC\}$ is the set of the top-dimensional simplices in a regular unimodular triangulation of the flow polytope $\calF_G$.
\end{proposition}

The triangulation in this proposition is called the \defn{framed triangulation} of $\calF_G$, and is denoted by \defn{$\Delta_{G,F}$}. Thus maximal cliques of \((G,F)\) are in bijection with facets of $\Delta_{G,F}$. Note that the maximal cliques are determined by the framing. So modifying the framing produces a change in the triangulation.

\begin{example}[Running example continued]\label{ex_running}
    Consider the framings $F_1$ and $F_2$ of $G=\oru{3}$ as in~\Cref{exam: oruga3_framings}. Their corresponding framed triangulations are shown in~\Cref{fig_cubes_Triang_oruga3}.

\begin{figure}[ht]
    \centering
    \includegraphics{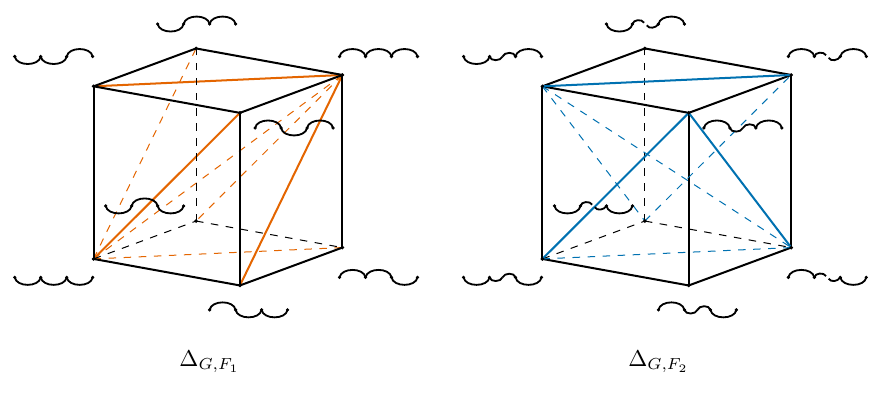}
    \caption{Two framed triangulations of $\calF_G$, for $G=\oru{3}$.}
    \label{fig_cubes_Triang_oruga3}
\end{figure}

\end{example}

\subsection{Framing lattices}

The framing lattice of a framed graph $(G,F)$ was introduced in~\cite{vonBell_framing_2024} as a certain poset whose Hasse diagram is the dual graph of the framed triangulation $\Delta_{G,F}$. In this section, we briefly recall the definition.

Let $C\neq C'$ be two maximal cliques in $(G,F)$ such that $C'=(C\setminus R)\cup R'$. Their corresponding simplices $\Delta_C$ and $\Delta_{C'}$ are adjacent in the framed triangulation because they share the codimension $1$ face $\Delta_{C\cap C'}$. The routes $R$ and $R'$ must be incoherent at some vertex $v$. We say that \defn{$R$ is clockwise from $R'$ at vertex $v$} when $Rv \leq_{\scrI(v)} R'v$ and $vR' \leq_{\scrO(v)} vR$, and denote this by $R <_v^{\cw} R'$. In this case we say that $R'$ is obtained from $R$ by a \defn{ccw rotation at $v$}, and also that $C'$ is obtained from $C$ by a \defn{ccw rotation}. Define the cover relation $C \precccwrot C'$ if $C'$ is obtainable from $C$ by a ccw rotation. An example is shown in~\Cref{fig_ccwrot}. The \defn{framing poset} $\scrL_{G,F} = (\calC, \leqccwrot)$ is the poset of ccw rotations of maximal cliques induced by the transitive closure of the cover relation~$\precccwrot$. We often just write $\leq$ and $\prec$ for simplicity, when the order is clear from the context.

\begin{figure}[ht]
    \centering
    \includegraphics[scale=0.7]{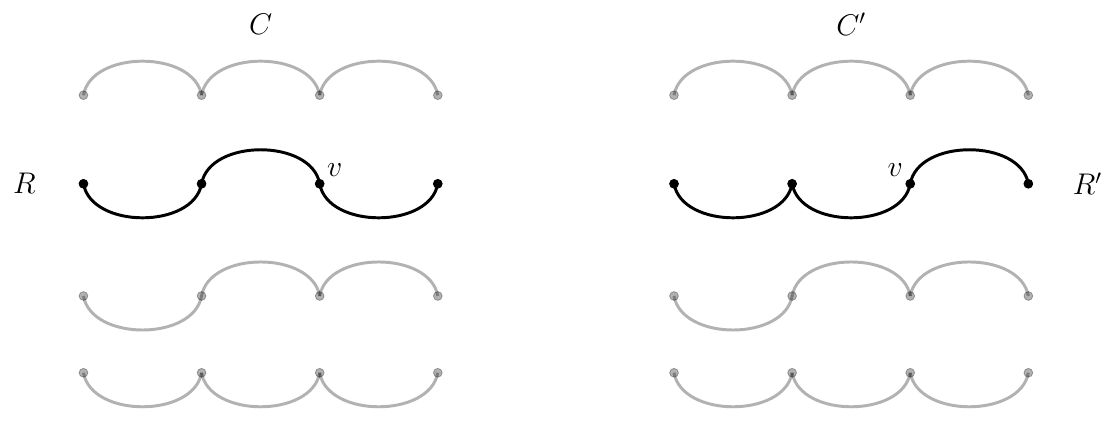}
    \caption{Two maximal cliques $C\precccwrot C'$ with $C'=(C\setminus R)\cup R'$ and~$R <_v^{\cw} R'$.}
    \label{fig_ccwrot}
\end{figure}

\begin{theorem}[{von Bell and Ceballos~\cite{vonBell_framing_2024}}]
    The framing poset $\scrL_{G,F}$ is:
    \begin{enumerate}
        \item A polygonal lattice whose polygons consist of squares, pentagons, or hexagons.
        \item A semidistributive lattice.
        \item An $HH$-lattice, and hence congruence uniform.
    \end{enumerate}
    Furthermore, its Hasse diagram is dual to the framed triangulation $\Delta_{G,F}$ of $\calF_G$.
\end{theorem}

\begin{example}[Running example continued]
    Consider the two framings $F_1$ and $F_2$ of~$G=\oru{3}$ as in~\Cref{exam: oruga3_framings}. Their corresponding framing lattices are shown in~\Cref{fig_framing_examples}.
    
    \begin{figure}[ht]
    
    \centering
    \includegraphics[scale=0.95]{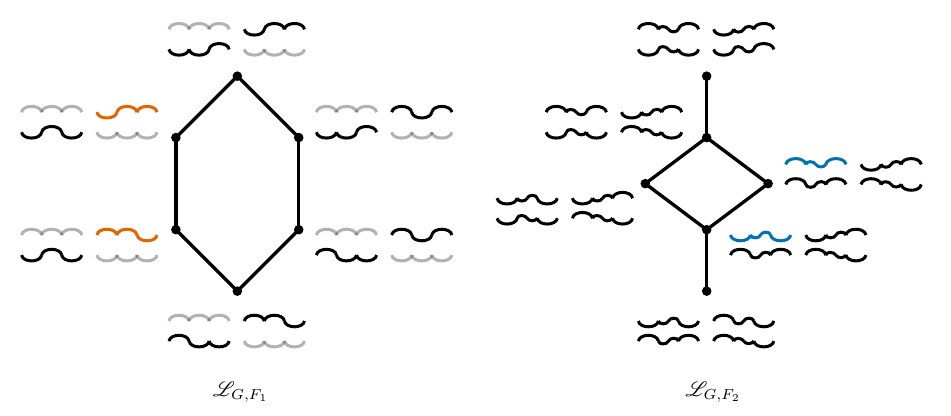}
    \caption{Two framing lattices for $G=\oru{3}$.}
    \label{fig_framing_examples}
\end{figure}

\end{example}

\begin{example}[The weak order on permutations~\cite{vonBell_framing_2024}]
\label{ex_weak_order}
Let $G_n=\oru{n}$ be the oruga graph from~\Cref{exam_oruga_n}, and let $F$ be the framing that orders the incoming and outgoing edges of $G_n$ from top to bottom. The maximal cliques of $(G,F)$ are in correspondence with permutations of $[n]$ as follows.

Given a permutation $[i_1,\dots,i_n]$ of $[n]$, construct a maximal clique consisting of $n+1$ routes $R_0,\dots,R_n$, where $R_k$ is the route with top edges outgoing at the vertices $i_1,\dots, i_k$ and bottom edges outgoing at the vertices $i_{k+1},\dots, i_n$. The resulting set of routes is a maximal clique, and every maximal clique is obtained this way. In other words, the maximal simplices of the framed triangulation of $\calF_{G_n}$ induced by the framing $F$ are in correspondence with permutations of $[n]$. Moreover, two maximal simplices are adjacent when the corresponding permutations can be obtained from each other by swapping two consecutive numbers. As a consequence, the framing lattice $\scrL_{G,F}$ is the classical weak order of permutations of $[n]$.
\end{example}
\section{Framingtopes: three equivalent definitions}
The goal of this section is to gain a deeper understanding of the underlying geometric structure of framing lattices. The Hasse diagram of the framing lattice is the edge graph of a higher-dimensional polytopal complex that we call the framingtope. Before presenting this geometric polytopal complex, we provide three different combinatorial definitions of framingtopes. Our main objective is to show that these definitions all yield the same object. The following theorem summarizes this result. The definitions in~\cref{def_framingtope_one,def_framingtope_two,def_framingtope_three} and their equivalence is presented in~\Cref{sec_framiingtope_interior,sec_framiingtope_covering,sec_framiingtope_pure}, while the geometric description in~\Cref{thm_framingtope_geometric_realization} is presented separately in~\Cref{sec_tropical_framingtope}, following some preliminaries on tropical geometry.

\begin{theorem}
    \label{thm.three_equivalent_definitions}
    Let $(G,F)$ be a framed graph, $\Delta_{G,F}$ be the framed triangulation of the flow polytope $\mathcal{F}_G$, and $\mathscr{L}_{G,F}$ be the corresponding framing lattice. The following three posets are isomorphic:
    \begin{enumerate}
        \item The complex of interior faces of the framed triangulation ordered by reverse inclusion. \label{def_framingtope_one}
        \item The collection of sets of pairwise coherent routes covering $G$ ordered by reverse inclusion. \label{def_framingtope_two}
        \item The collection of pure intervals of the framing lattice ordered by inclusion. \label{def_framingtope_three}
        
    \end{enumerate}
    We call any of these three posets the \defn{framingtope} $\mathscr{P}_{G,F}$.
\end{theorem}

The framingtope $\mathscr{P}_{G,F}$ can be geometrically realized as a polytopal complex of bounded cells of an arrangement of tropical hypersurfaces. This arrangement depends on an admissible height function $\h$, and we call the corresponding polytopal complex the \defn{tropical framingtope}~$\mathscr{P}_{G,F}^{\h}$.

\begin{theorem}\label{thm_framingtope_geometric_realization}
    The tropical framingtope~$\mathscr{P}_{G,F}^{\h}$ realizes the combinatorial framingtope~$\mathscr{P}_{G,F}$
\end{theorem}

Its edge graph provides a geometric realization of the Hasse diagram of the corresponding framing lattice.

\begin{corollary}
    The Hasse diagram of the framing lattice $\mathscr{L}_{G,F}$ is the edge graph of the tropical framingtope $\mathscr{P}_{G,F}^{\h}$.
\end{corollary}

\subsection{Framingtopes in terms of interior faces}
\label{sec_framiingtope_interior}

We start by introducing the combinatorial definition of the framingtope in terms of interior faces of the framed triangulation.

\begin{definition}[Framingtopes: First Definition]\label{def_framingtope_asDual}
    Let $(G,F)$ be a framed graph and $\Delta_{G,F}$ its framed triangulation. A face in $\Delta_{G,F}$ is called a \defn{boundary face} if it is contained in the boundary of the flow polytope $\mathcal{F}_G$. All other faces in~$\Delta_{G,F}$ are called \defn{interior faces}. The \defn{framingtope} $\mathscr{P}_{G,F}$ is the complex of interior faces of the framed triangulation~$\Delta_{G,F}$ of $\mathcal{F}_G$, ordered by reverse inclusion. Here and throughout, by the dual complex of the interior faces we mean the poset of interior faces ordered by reverse inclusion.
\end{definition}

\begin{example}[Running example]
\label{Example:Dual_Cube}
Let $G$ be the oruga graph $\oru{3}$ with the framing~$F$ where the edge $e_i$ is smaller than $e_{i+1}$ in both the incoming and outgoing orders for $i=1,3,5$.

\Cref{fig_framingtope_def1} shows all the interior faces of the framed triangulation: six interior tetrahedra, six interior triangles, and one interior edge. All other faces lie in the boundary of the cube. As we can see in the figure, the complex of interior faces ordered by reverse inclusion (its dual complex) can be identified with the complex of faces of a hexagon.

\begin{figure}[ht]
    \centering
    \includegraphics[scale=0.98]{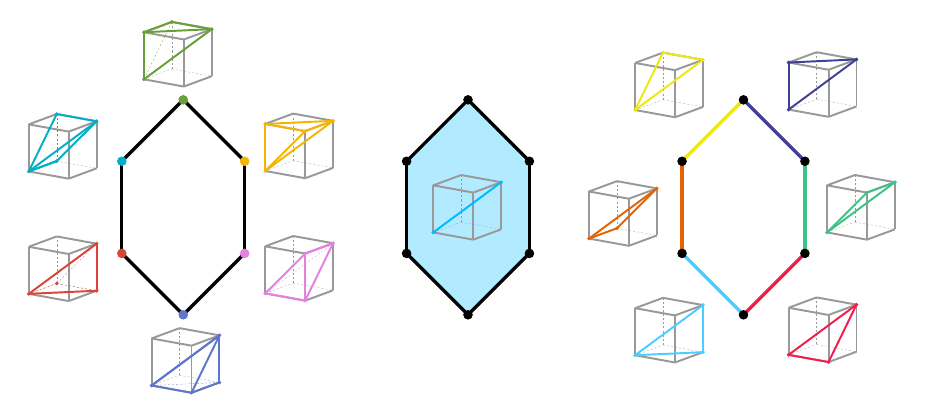}
    \caption{Framingtope $\mathscr{P}_{G,F}$ as the complex of interior faces of a framed triangulation ordered by reverse inclusion (First Definition).}
    \label{fig_framingtope_def1}
\end{figure}

\end{example}

\subsection{Framingtopes in terms of covering routes} 
\label{sec_framiingtope_covering}

As introduced above, the framingtope is the dual of the complex of interior faces of the framed triangulation. Our next step is to give a combinatorial characterization of such interior faces in terms of sets of coherent routes covering the graph $G$.

\begin{definition}
    We say that a set of routes $S$ \defn{covers} $G$ if for every edge $e\in E(G)$ there is a route $R\in S$ containing $e$.
\end{definition}

In order to show our characterization we need to understand the boundary of a flow polytope. In~\cite[Theorem~3.2]{hille_quivers_2003}, Hille gave a complete characterization of the faces of flow polytopes with arbitrary net flow. The following is a specialization of their result for the case of flow polytopes of unit flows considered in this paper. See also~\cite{dugan_fvectors_2024} for a similar formulation.

\begin{proposition}[{\cite[Theorem~3.2]{hille_quivers_2003}}]\label{prop_faces_hille}
    Let $\mathcal{F}_G$ be the flow polytope of unit flows of $G$. The faces of $\mathcal{F}_G$ are of the form $\mathcal{F}_{H}$ for subgraphs $H\subseteq G$ induced by a set of routes of~$G$. The dimension of the face $\mathcal{F}_{H}$ is $|E(H)|-|V(H)|+1$.
\end{proposition}

    The vertices of $\mathcal{F}_G$ correspond to routes of $G$, the empty face corresponds to the empty subgraph, and the full polytope corresponds to $H=G$.
    \begin{example}
        Let us take the graph $G$ as in \Cref{fig_faces_hille}. Its flow polytope is a 3-dimensional pyramid whose faces correspond to the subgraphs obtained by overlapping the routes giving the vertices of that face. For \Cref{fig_exterior_faces} let $\oru{3}$ be the graph with the framing $F_1$ from \Cref{fig_cubes_Triang_oruga3}. These figures show explicit examples of \Cref{prop_faces_hille}.
    \end{example}

\begin{figure}[ht]
    \centering
    \includegraphics[scale=0.8]{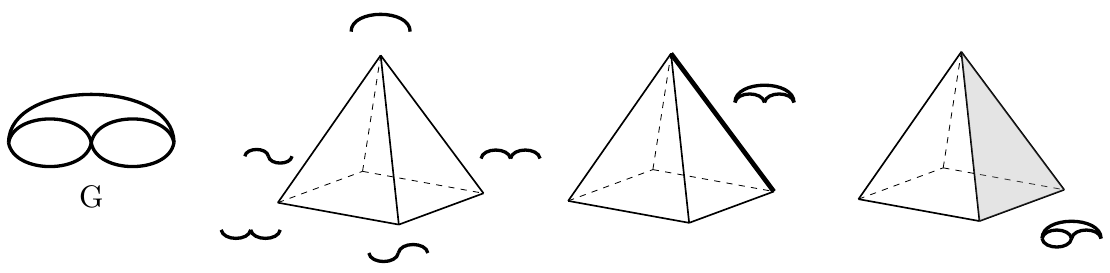}
    \caption{Faces of a flow polytope in terms of subgraphs induced by sets of routes.}
    \label{fig_faces_hille}
\end{figure}    

\begin{figure}[ht]
    \centering
    \includegraphics[scale=0.8]{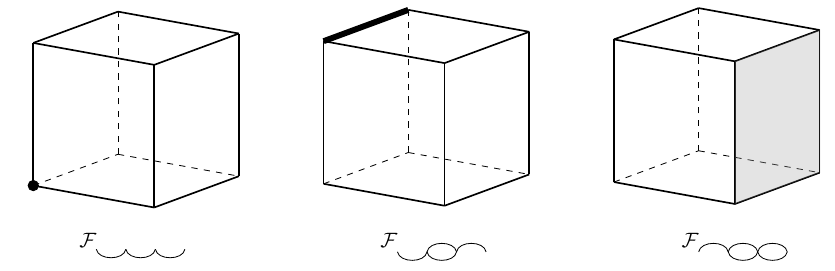}
    \caption{Some faces of the flow polytope of the graph $\oru{3}$.}
    \label{fig_exterior_faces}
\end{figure}

\begin{theorem}\label{thm_interior_equal_covering}    
    Let $(G,F)$ be a framed graph. Then $S$ is an interior face of the framed triangulation of $\mathcal{F}_G$ if and only if $S$ is a set of pairwise coherent routes that covers $G$.
\end{theorem}

\begin{proof}
Let $S= \{R_1,\ldots, R_k\}$ be a set of coherent routes. This set corresponds to a face~$\widetilde{S}=\text{conv}\set{\textbf{V}(R_1),\ldots,\textbf{V}(R_k)}$, which is the convex hull of the indicator vectors of the routes in $S$. We use the notation $\widetilde{S}$ to avoid confusion.

We need to show that $\widetilde{S}$ is an interior face of the framed triangulation if and only if~$S$ covers $G$. Equivalently, $\widetilde{S}$ is in the boundary of $\mathcal{F}_G$ if and only if $S$ does not cover~$G$.

$(\Rightarrow)$ Assume that $\widetilde{S}$ is in the boundary of $\mathcal{F}_G$. By~\Cref{prop_faces_hille}, we have that $\widetilde{S}\subseteq \mathcal{F}_H$ for some proper subgraph $H\subseteq G$ induced by a set of routes of $G$. Thus, there is an edge $e\in G\setminus H$ such that none of the routes $R_i\in S$ contain $e$ as an edge. Therefore,~$S$ does not cover $G$.

$(\Leftarrow)$ Assume that $S$ does not cover $G$, and let $e$ be an uncovered edge. Consider the subgraph $H\subseteq G$ induced by the routes in $S$. There exists at least one route in~$G$ containing $e$ (since $G$ has a unique source and sink, $e$ can be extended to a route), and its indicator vector is a vertex of $\mathcal{F}_G$ that is not in $\mathcal{F}_H$. Therefore, $\mathcal{F}_H$ is a boundary face of $\mathcal{F}_G$. Since $\widetilde{S}\subseteq \mathcal{F}_H$, then $\widetilde{S}$ is in the boundary of $\mathcal{F}_G$.
\end{proof}

This gives a correspondence between the interior faces of the framed triangulation and the collection of pairwise coherent routes covering the graph. This provides the following alternative definition of framingtopes.

\begin{definition}[Framingtopes: Second Definition]\label{def_framingtope_coveringRoutes}
    Let $(G,F)$ be a framed graph. The \defn{framingtope} $\mathscr{P}_{G,F}$ is the collection of sets $S$ of pairwise coherent routes covering~$G$, ordered by reverse inclusion.
\end{definition}

\begin{corollary}
    The first and second definitions of framingtopes (\Cref{def_framingtope_asDual,def_framingtope_coveringRoutes}) are equivalent.
\end{corollary}

\begin{proof}
    This follows directly from the characterization of interior faces of the framed triangulation in~\Cref{thm_interior_equal_covering}.
\end{proof}

\begin{example}[\Cref{Example:Dual_Cube} continued]
    Let $(G,F)$ be as in our running~\Cref{Example:Dual_Cube}. Figure~\ref{fig_framingtope_second_definition} shows the framingtope~$\mathscr{P}_{G,F}$ obtained as the dual of the complex of interior faces of the framed triangulation indexed by sets of coherent routes covering $G$. The maximal cliques (consisting of four routes) correspond to the six vertices, the cliques of size 3 correspond to the six edges, and the clique consisting of the two exceptional routes is the 2-dimensional face of the framingtope. The other cliques not shown in the figure do not cover the graph $G$ and correspond to boundary faces of the triangulation.

\begin{figure}[ht]
    \centering
    \includegraphics[]{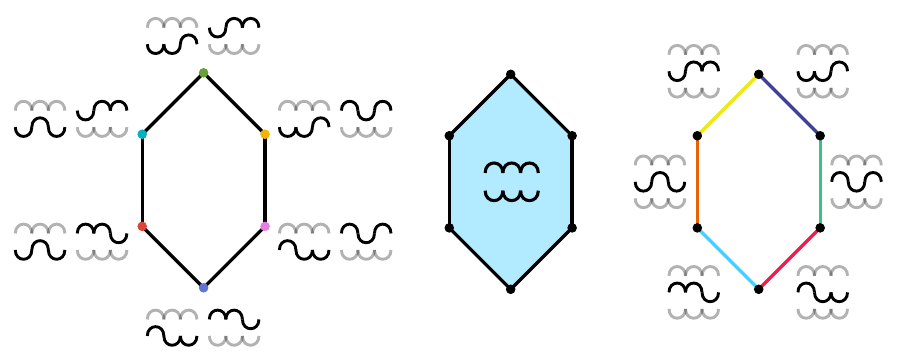}
    \caption{Framingtope $\mathscr{P}_{G,F}$ as the complex of pairwise coherent routes covering $G$, ordered by reverse inclusion (Second Definition).}
    \label{fig_framingtope_second_definition}
\end{figure}

\end{example}

\subsection{Framingtopes in terms of pure intervals}
\label{sec_framiingtope_pure}
With the description of interior faces in terms of covering sets of coherent routes established, we now present a characterization of the framingtope purely in terms of the framing lattice. To this end, we introduce the notion of pure intervals, which can be thought of as the ``small polytopal intervals'' of the lattice.

\begin{definition}[Descents and ascents]
    For a maximal clique $C$ in the framing lattice~$\scrL_{G,F}$, the \defn{descent set} and \defn{ascent set} of $C$ are defined as
    \begin{align*}
        \Desc(C) &= \{R\in C : R \text{ can be rotated downwards in } C \} \\
        \Asc(C) &= \{R\in C : R \text{ can be rotated upwards in } C \} 
    \end{align*}
    The routes in $\Desc(C)$ are called the \defn{descents of $C$}, and the routes in $\Asc(C)$ are called the \defn{ascents of $C$}. In other words, an element \(R \in C\) is called a descent (respectively, an ascent) if there exists a route \(R'\) such that \(R' <^{\textbf{cw}}_{v} R\) (respectively, \(R <^{\textbf{cw}}_{v} R'\)) for some \(v \in R \cap R'\), and the set \(C \setminus \{R\} \cup \{R'\}\) is again a maximal clique. In this case, we say that \(R'\) is the \defn{decreasing flip} (respectively, \defn{increasing flip}) of \(R\).
    
    We also define the \defn{descent} and \defn{ascent} statistics of $C$ as
    \begin{align*}
        \desc(C) &=|\Desc(C)|= \text{number of descents of } C \\
        \asc(C) &=|\Asc(C)|= \text{number of ascents of } C. 
    \end{align*}
\end{definition}

\begin{definition}[Pure intervals]
    Let $A$ be a subset of ascents of a maximal clique $C$. We define the maximal clique $C+A$ as the join
    \[C+A:=\bigvee_{a\in A}(C+a)\] 
    where $C+a:=C\setminus\set{a}\cup\set{a'}$ is the maximal clique covering $C$ corresponding to the increasing flip $a'$ of $a$. The interval $[C,C+A]$ is called a \defn{pure interval} of the framing lattice $\mathscr{L}_{G,F}$.
\end{definition}

\begin{example}[\Cref{Example:Dual_Cube} continued]
    Let $(G,F)$ be as in our running~\Cref{Example:Dual_Cube}. Figure~\ref{fig_framingtope_pureIntervals} shows all pure intervals $[C, C + A]$ of the framing lattice $\mathscr{L}_{G,F}$. The exceptional routes are shown in gray, while the ascent routes in $A$ are highlighted in colors other than black or gray. There are six pure intervals with an empty subset of ascents, $|A| = 0$; these correspond to the vertices labeled by maximal cliques on the left of the figure. There are six pure intervals with $|A| = 1$, which label the six edges of the hexagon on the right of the figure. Finally, there is one pure interval with $|A| = 2$, corresponding to the hexagon in the center.

    In particular, we observe that the collection of pure intervals, ordered by containment, provides another interpretation of the framingtope $\mathscr{P}_{G,F}$ in this example. We will show that this holds in general.
\end{example}

\begin{figure}[ht]
    \centering
    \includegraphics[scale=0.95]{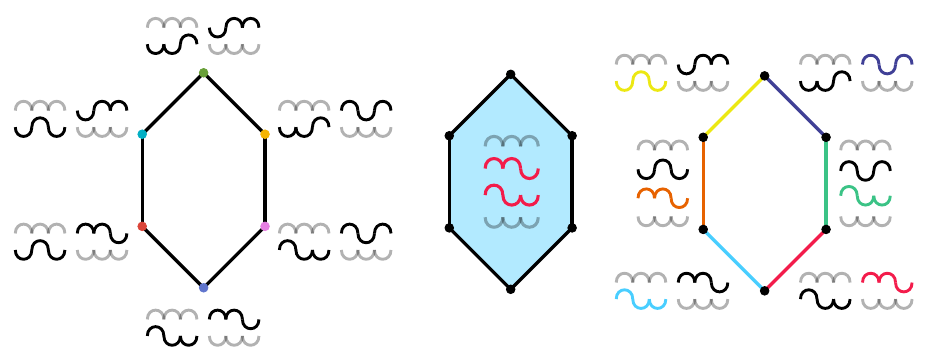}
    \caption{Framingtope $\mathscr{P}_{G,F}$ as the collection of pure intervals ordered by containment (Third Definition).}
    \label{fig_framingtope_pureIntervals}
\end{figure}

\begin{definition}[Framingtopes: Third Definition]\label{def_framingtope_pureIntervals}
    Let $(G,F)$ be a framed graph. The \defn{framingtope} $\mathscr{P}_{G,F}$ is the collection of pure intervals of the framing lattice $\mathscr{L}_{G,F}$, ordered by containment.
\end{definition}

One of our main goals is to show that this definition yields an alternative, equivalent interpretation of the complex of interior faces of the framed triangulation, ordered by reverse inclusion:

\begin{corollary}\label{cor_def_framingtopes_one_three}
    The first and third definitions of framingtopes (\Cref{def_framingtope_asDual,def_framingtope_pureIntervals}) are equivalent.
\end{corollary}

The proof of this result is more involved and we need to develop several useful tools. In Section~\ref{sec_enumerative}, we present various enumerative properties of framed triangulations and, in particular, show that the number of interior faces is equal to the number of pairs $(C, A)$, where $C$ is a maximal clique and $A \subseteq C$ is a subset of ascents (\Cref{cor_enumeration_interior_faces}). The full correspondence between interior faces and pure intervals leading to the proof of~\Cref{cor_def_framingtopes_one_three} is then established in~\Cref{sec_interior_faces_pure_intervals}.

\begin{remark}
We remark that pure intervals could equivalently be defined in terms of descents rather than ascents. This follows from~\Cref{cor_pureintervals_equivalence} in~\Cref{sec_interior_faces_pure_intervals}.
\end{remark}
\section{Enumerative properties}\label{sec_enumerative}

\subsection{\texorpdfstring{$f$- and $h$-vectors}{f- and h-vectors} of framed triangulations}

Let $\Delta$ be a finite $d$-dimensional simplicial complex. The \defn{f-vector} of $\Delta$ is the sequence \( f(\Delta) = (f_{-1}, f_0, f_1, \dots, f_d), \) where $f_i$ denotes the number of $i$-dimensional faces of $\Delta$, and $f_{-1} = 1$ for the empty face. The \defn{h-vector} $(h_0, h_1, \dots, h_{d+1})$ is defined by the polynomial identity
\begin{equation}\label{eq1_f_h}
\sum_{i=0}^{d+1} h_i t^i
=
\sum_{i=0}^{d+1} f_{i-1}\, t^i (1 - t)^{d+1-i}.  
\end{equation}
Or equivalently,
\begin{equation}\label{eq2_f_h}
\sum_{i=0}^{d+1} h_i t^{d+1-i}
=
\sum_{i=0}^{d+1} f_{i-1}\, (t - 1)^{d+1-i}.    
\end{equation}

The polynomial
\[
h_\Delta(t) = \sum_{i=0}^{d+1} h_i t^i
\]
is called the \defn{h-polynomial} of $\Delta$.

\begin{example}\label{ex_f_h_vector_triangulated_cube}
    Let $\Delta_1=\Delta_{G,F_1}$ and $\Delta_2=\Delta_{G,F_2}$ be the two framed triangulations of the 3-dimensional cube in our running example,~\Cref{ex_running}. Both triangulations have $8$ vertices, $19$ edges, $18$ triangles, and $6$ tetrahedra. Their $f$-vectors coincide and are equal to
    \[
    f=(1,8,19,18,6).\]

We can compute the $h$-vector using the defining relations with $d=3$. The result is
\[
h = (1, 4, 1, 0, 0).
\]
Equivalently, the $h$-vector can be computed using a Pascal-like triangle method (Stanley's trick) for visualizing and calculating the coefficients recursively in a triangular array:
Place the $f$-vector values on the rightmost edge of the triangle, and fill the leftmost edge with $1$'s. For every other position, compute the value by taking the value diagonally above it to the right and subtracting the value diagonally above it to the left. The $h$-vector appears along the bottom edge of the triangle as illustrated in \Cref{fig_pascal}.

\begin{figure}[ht]
    \centering
    \includegraphics[]{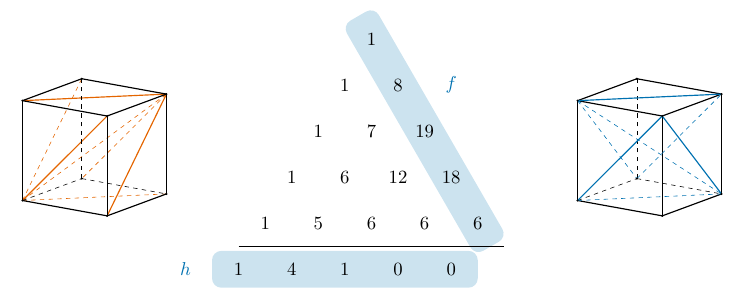}
    \caption{Pascal-like triangle of two framed triangulations, with the same $f$- and $h$-vector.}
    \label{fig_pascal}

\end{figure}

The $h$-polynomial for both triangulations is
\[
h_{\Delta_i}(t) = 1 + 4t + t^2.
\]
\end{example}

\medskip

The $h$-vector contains all the information about the face numbers of $\Delta$ and it is very often simpler to compute. In particular, when $\Delta$ is a unimodular triangulation of a lattice polytope, it is well known that the $h$-polynomial will coincide with the $h^*$-polynomial from Ehrhart theory, as we will recall in~\Cref{sec_Ehrhart}. As a consequence, one can deduce that the $f$-vector and $h$-vector of a framed triangulation is independent of the framing (\Cref{cor_fvector_hvector_framingIndependent}).

Before doing that, we present an explicit way of computing the $h$-vector of framed triangulations in terms of framing lattices (\Cref{thm_h_vector}), using shellings. As a straightforward consequence one obtains a direct combinatorial description of the $h^*$-polynomials of flow polytopes and their Ehrhart series (\Cref{cor_hStar_polynomial}).

\subsection{Shellings and framing Narayana numbers}


Let $\Delta$ be a pure $d$-dimensional simplicial complex. A \defn{shelling} of $\Delta$ is an ordering $C_1, \ldots, C_m$ of its maximal simplices (facets) such that for each $1< j \leq m$, if we set $ K_{j} = \bigcup_{i=1}^{j} C_i, $ then the intersection of~$C_j$ with the previous facets
\[
C_{j} \cap K_{j-1}
\]
is a pure $(d-1)$-dimensional subcomplex of $K_{j-1}$. In other words, for every $i<j$ there exists some $\ell<j$ such that the intersection $C_i\cap C_j$ is contained in $C_\ell \cap C_j$, and such that $C_\ell \cap C_j$ is a facet of $C_j$. We refer to~\cite{ziegler_lectures_1995} for references on shellings.

The following result provides a family of natural shellings of a framed triangulation, indexed by the linear extensions of the framing lattice.

\begin{proposition}\label{prop_shelling}
Any ordering $\mathscr{O}=(C_1,\ldots, C_m)$ of the maximal cliques that is a linear extension of $\mathscr{L}_{G,F}$ or of its opposite lattice is a shelling order of the framed triangulation $\Delta_{G,F}$ of $\mathcal{F}_G$.
\end{proposition}

\begin{proof}
    We treat the case where $\mathscr{O}=(C_1,\ldots, C_m)$ is a linear extension of the framing lattice $\mathscr{L}_{G,F}$; a linear extension of the opposite lattice is analogous. We need to verify the shelling condition: for every $i<j$, there exists $\ell<j$ such that
    \[
        C_i\cap C_j \subseteq C_\ell\cap C_j = C_j\setminus \set{R}
    \]
    for some route $R\in C_j$.

    Let $C_{k} = C_i \wedge C_j$, and set $S = C_i \cap C_j$. Since $C_{k} < C_j$, we have $k < j$. By the description of the meet operation in~\cite[Proposition~1.3.7 and Theorem~1.3.8]{vonBell_framing_2024},
    \[
        S = C_i \cap C_j \subseteq C_i \wedge C_j = C_{k},
    \]
    and hence $C_i \cap C_j \subseteq C_{k} \cap C_j$.

    Consider a saturated chain
    \[
        C_{k} = D_1 \prec D_2 \prec \cdots \prec D_r \prec D_{r+1} = C_j,
    \]
    and let $D_r = C_\ell$. Then $\ell < j$. By construction, $C_\ell$ is obtained from $C_j$ by rotating a route $R \in C_j$ downward (clockwise rotation), and thus
    
    \[
        C_\ell \cap C_j = C_j \setminus \set{R}.
    \]

    Moreover, there exists a route $R' \in C_\ell$ such that $C_j \setminus \set{R} = C_\ell \setminus \set{R'}$ and $R' <^{\mathrm{cw}}_v R$ for some $v \in R' \cap R$. In particular, this implies that $R \notin C_k$; otherwise, $C_k$ would not be smaller than or equal to $C_\ell$ (see~\cite[Theorem~1.2.15]{vonBell_framing_2024}).

    Therefore,
    \[
        C_i \cap C_j \subseteq C_{k} \cap C_j \subseteq C_\ell \cap C_j = C_j \setminus \set{R},
    \]
    as required.
\end{proof}

Thus the framing lattice does not merely encode the adjacency graph of the triangulation; its order structure also determines natural shelling orders of the triangulation.

Shellable simplicial complexes have the advantage that their $h$-vector admits a direct combinatorial interpretation in terms of the restriction faces of a shelling. For each facet $C_j$ in a shelling, its \defn{restriction face} $R(C_j)$ is the set of all vertices $v \in C_j$ such that $C_j\setminus v$ is contained in one of the earlier facets:
\[
R(C_j)=\{
v\in C_j:\ C_j\setminus v \subseteq C_i \text{ for some } i<j 
\}.
\]
Equivalently~\cite[Section~8.3]{ziegler_lectures_1995}, the restriction face $R(C_j)$ is the unique minimal face of~$C_j$ that is not contained in $K_{j-1} = \bigcup_{i=1}^{j-1} C_i$ or, in other words, the smallest face of~$C_j$ that is ``new'' when $C_j$ is added to the complex.


The \defn{h-vector} $(h_0, h_1, \dots, h_d)$ of $\Delta$ can then be computed directly from the shelling: for each $k=0, \dots, d$,
\[
h_k = \#\{ j \mid |R(C_j)| = k \},
\]
where $|R(C_j)|$ denotes the number of vertices in the restriction face $R(C_j)$. In light of~\Cref{prop_shelling} we get the following result.

\begin{theorem}\label{thm_h_vector}
    Let $(G,F)$ be a framed graph and $d=\dim(\calF_G)$ be the dimension of the flow polytope $\calF_G$. The $h$-vector $h=(h_0,h_1,\dots,h_{d+1})$ of the framed triangulation $\Delta_{G,F}$ of $\calF_G$ can be described combinatorially as follows:
    \begin{enumerate}
        \item $h_k$ equals the number of elements in the framing lattice $\scrL_{G,F}$ that have exactly~$k$ descents (equivalently, $k$ down covers).
        \item $h_k$ equals the number of elements in the framing lattice $\scrL_{G,F}$ that have exactly~$k$ ascents (equivalently, $k$ up covers).
    \end{enumerate}
    In particular, the last coefficient $h_{d+1} = 0$.
    
    The numbers $h_k$ are referred to as the \defn{framing Narayana numbers}.
\end{theorem}

\begin{proof}
Fix a linear extension $\mathscr O=(C_1,\ldots,C_m)$ of the framing lattice $\scrL_{G,F}$. By~\Cref{prop_shelling}, this is a shelling of $\Delta_{G,F}$. We claim that the restriction face of $C_j$ is precisely the set of its descents: 
\[
R(C_j)=\Desc(C_j).
\]

If $R\in \Desc(C_j)$ is a descent of $C_j$ then it can be rotated downwards to produce a maximal clique $C_i$, with $C_j\setminus R = C_i\setminus R'\subseteq C_i$ for some route $R'$. Since $\mathscr O$ is a linear extension then $i<j$ and so $R\in R(C_j)$. On the other hand, if $R \in R(C_j)$ then there exist a maximal clique $C_i$ with $i<j$ such that $C_j\setminus R \subseteq C_i$. Therefore, $C_i$ and $C_j$ share a codimension 1 face $C_i\cap C_j=C_j\setminus R$, and are adjacent facets in the framed triangulation. Thus, they are related by a rotation. More precisely, the maximal clique~$C_i$ is obtained by rotating the route $R$ in $C_j$. Since $i<j$ then this rotation must be a downward rotation and so $R\in \Desc(C_j)$ as wanted.

As a consequence,
\[
h_k=\#\set{j:|R(C_j)|= |\Desc(C_j)|=k}=\#\set{C\in\scrL_{G,F}: C\text{ has $k$ descents}},
\]
which proves~(1).

For~(2), consider instead a linear extension of the opposite lattice $\scrL_{G,F}^{\mathrm{op}}$. Again by~\Cref{prop_shelling}, this gives a shelling of $\Delta_{G,F}$. The descents in the opposite lattice are exactly the ascents in $\scrL_{G,F}$, so the same argument as in part (1) yields
\[
h_k=\#\set{C\in\scrL_{G,F}: C\text{ has $k$ ascents}}.
\]

Finally, the last coefficient $h_{d+1}$ can be computed in terms of the Euler characteristic of the triangulation via the relation
    \[
    h_{d+1} = (-1)^d \bigl(\chi(\Delta_{G,F}) - 1\bigr).
    \] 
    The Euler characteristic of a triangulated ball is 1, so $h_{d+1}=(-1)^d (1 - 1)=0$.
\end{proof}

\begin{remark}
The method of obtaining a shelling order by taking a linear extension of a poset structure on the underlying simplicial complex appears in several instances in the literature \cite{BGMY23, BBBHPSY24, BS24, CeballosPAdrolSarmiento2019}. For example, it was used to compute the $h$-vector of the $(I,J)$-Tamari complex in \cite{CeballosPAdrolSarmiento2019}, whose coefficients are the $(I,J)$- and $\nu$-Narayana numbers. In the context of flow polytopes, the method was used on the planar framed triangulation of the $\nu$-caracol flow polytope in \cite{BGMY23} to show that its $h^*$-polynomial is the~$\nu$-Narayana polynomial.
\end{remark}

\begin{corollary}\label{cor_hPolynomial_desc_asc}
    The $h$-polynomial $h(t)=\sum h_i t^i$ of the framed triangulation $\Delta_{G,F}$ is
    \begin{enumerate}
        \item $h(t)= \sum t^{\desc(C)}$,
        \item $h(t)= \sum t^{\asc(C)}$,
    \end{enumerate}
    where the sum runs over all maximal cliques of the framing lattice $\scrL_{G,F}$.
\end{corollary}
\begin{proof}
    This follows directly from~\Cref{thm_h_vector}.
\end{proof}

\begin{remark}
    In the case of the oruga graph $G=\oru{n}$ with its natural plane framing, the $h$-polynomial is the Eulerian polynomial, which enumerates permutations by their number of descents. As we will explain below, however, $h(t)$ is independent of the choice of framing (\Cref{cor_fvector_hvector_framingIndependent}).
\end{remark}

\subsection{Ehrhart theory of flow polytopes}\label{sec_Ehrhart}
Given a lattice polytope \( P \subset \mathbb{R}^d \), its Ehrhart series is
\[
\sum_{m \ge 0} \lvert mP \cap \mathbb{Z}^d \rvert \, t^m
\;=\;
\frac{h^*(P,t)}{(1-t)^{d+1}},
\]
where
\[
h^*(P,t) = h_0^* + h_1^* t + \cdots + h_s^* t^s
\]
is the \(h^*\)-polynomial of \(P\).

For a unimodular triangulation \(\mathcal{T}\) of \(P\), denote by
\[
h(\mathcal{T}, t) = h_0 + h_1 t + \cdots + h_d t^d
\]
the \(h\)-polynomial of \(\mathcal{T}\).

A classical result in Ehrhart theory~\cite[Theorem 10.3]{beck_computing_2015} asserts that if a lattice polytope~$P$ admits a unimodular triangulation $\mathcal{T}$, then
\[
h^*(P,t) = h(\mathcal{T}, t).
\]
Thus any two unimodular triangulations of the same lattice polytope have the same \(h\)-vector, and hence the same \(f\)-vector. Since all framed triangulations of flow polytopes are unimodular~\cite{dkktriangulation}, the following corollary is immediate.

\begin{corollary}\label{cor_fvector_hvector_framingIndependent}
The \(f\)-vector and \(h\)-vector of a framed triangulation are independent of the choice of framing.
\end{corollary}

In contrast to~\Cref{thm_h_vector}, we obtain the following combinatorial interpretation of the \(h^*\)-polynomial of a flow polytope.
\begin{corollary}\label{cor_hStar_polynomial}
    Let $\calF_G$ be the flow polytope of a flow graph $G$ and $d = \dim(\calF_G)$. For any framing $F$ of $G$, the Ehrhart series of $\calF_G$ can be written as
    \[
    \operatorname{Ehr}_{\calF_G}(t) = \frac{h^*(t)}{(1-t)^{d+1}},
    \qquad
    h^*(t) = h_0^* + h_1^* t + \dots + h_d^* t^d,
    \]
    where
    \begin{enumerate}
        \item $h_k^*$ equals the number of elements in the framing lattice $\scrL_{G,F}$ with exactly $k$ down covers.
        \item $h_k^*$ equals the number of elements in the framing lattice $\scrL_{G,F}$ with exactly $k$ up covers.
    \end{enumerate}
\end{corollary}

\begin{proof}
    This follows directly from the $h$-vector computation for framed triangulations in~\Cref{thm_h_vector}, and the fact that the $h^*$-polynomial of the flow polytope coincides with the $h$-polynomial of any framed triangulation.
\end{proof}

\begin{remark}
    Special cases of Corollary~\ref{cor_hStar_polynomial} appear in the literature. The case for a family of graphs called full, that admit an ample framing, was given in \cite[Lemma~6.10]{BBBHPSY24}. This was further extended to a family of graphs called rooted framed graphs in~\cite[Proposition~4.12]{BS24}.
\end{remark}

\begin{remark}
    Although the Ehrhart theory of flow polytopes is well studied, our results give a simple combinatorial interpretation of the $h^*$-coefficients. For flow polytopes with arbitrary net-flow $\mathbf{a}$, the Ehrhart polynomial was first computed by Baldoni and Vergne~\cite{baldoni2008} using residue computations, with a purely combinatorial description later given by M\'esz\'aros and Morales \cite{MeszarosMorales}. More precisely,
    the Ehrhart polynomial $\operatorname{ehr}(m)=|mP\cap \mathbb{Z}^d|$ is given by the Kostant partition function $K(\Phi_G^+, m\mathbf{a})$, which counts the number of ways to write the vector $m\mathbf{a}$ as a nonnegative integer linear combination of the positive roots $\Phi_G^+$ in the type $A$ root system associated to the edges of the graph~$G$. In the case of unit flow polytopes, the net flow vector is $\mathbf{a}=\mathbf{e_1}-\mathbf{e_n}$, where $\mathbf{e}_1$ and $\mathbf{e}_n$ are the standard basis vectors associated to the source and sink respectively.
\end{remark}

\subsection{Counting interior faces} 
We now turn our attention to the problem of counting the number of interior faces of framed triangulations of flow polytopes. This will be achieved using a variant of the Dehn-Somerville relations~\cite{CeballosMuhle2022} and the evaluation of the $h$-polynomial of the triangulation at $t+1$.

\begin{proposition}
    Let $(G,F)$ be a framed graph and $d=\dim(\calF_G)$ be the dimension of the flow polytope $\calF_G$. The $h$-polynomial $h(t)= \sum_{i=0}^{d+1} h_i t^i$ of the framed triangulation~$\Delta_{G,F}$ satisfies
    \begin{equation}\label{eq_h_evaluation_interior_faces}
        h(t+1) =
        \sum_{i=0}^{d} f_{d-i}^{int} t^{i},
    \end{equation}
    where $f_i^{int}$ is the number of $i$-dimensional interior faces of $\Delta_{G,F}$.
\end{proposition}

\begin{proof}
    Applying the variant of the Dehn-Sommerville relations in~\cite[Corollary~3.2]{CeballosMuhle2022} to the case of a triangulated topological ball, which is the case for a framed triangulation of a flow polytope, we obtain the following relation.
    \begin{equation}
        t^{d+1}h\left( \frac{t+1}{t}\right) =
        \sum_{i=1}^{d+1} f_{i-1}^{int} t^i.
    \end{equation}
    Evaluating this relation at $\frac{1}{t}$, and then multiplying the result by $t^{d+1}$ yields
    \begin{equation}
        h(t+1) =
        \sum_{i=1}^{d+1} f_{i-1}^{int} t^{d+1-i}
        =
        \sum_{i=0}^{d} f_{i}^{int} t^{d-i}.
    \end{equation}
    This relation is equivalent to the desired equation.
\end{proof}

\begin{example}[Example~\ref{ex_f_h_vector_triangulated_cube} continued]
In our example of the two framed triangulations of the 3-dimensional cube associated with the oruga graph $\oru{3}$, we have $d = 3$, and the $h$-polynomial is
\[h(t)=1+4t+t^2.\]
Its evaluation at $t+1$ is
\begin{align*}
h(t+1) &= 6+6t+t^2\\
&= 
f_3^{int}+
f_2^{int}t+
f_1^{int}t^2+
f_0^{int}t^3 \,.
\end{align*}
Therefore
\[
f_3^{int}=6, \quad
f_2^{int}=6, \quad
f_1^{int}=1, \quad 
f_0^{int}=0 \, .
\]
Geometrically, this implies that both triangulations have 6 interior tetrahedra, 6 interior triangles, 1 interior edge, and no interior vertices. The reader may verify this in~\Cref{fig_cubes_Triang_oruga3}.
\end{example}

\begin{corollary}\label{cor_enumeration_interior_faces}
    The number $f_{d-i}^{int}$ of codimension $i$ interior faces of the framed triangulation $\Delta_{G,F}$ is equal to:
    \begin{enumerate}
        \item the number of pairs $(C,D)$ such that $C$ is a maximal clique and $D\subseteq \Desc(C)$ is a subset of descents of $C$, with $|D|=i$.
        \item the number of pairs $(C,A)$ such that $C$ is a maximal clique and $A\subseteq \Asc(C)$ is a subset of ascents of $C$, with $|A|=i$.
    \end{enumerate}
\end{corollary}
\begin{proof}
    By~\Cref{cor_hPolynomial_desc_asc} (1) we have
    \begin{align*}
    h(t+1) = \sum_C (t+1)^{|\Desc(C)|}     = \sum_{\substack{(C,D)\\ D\subseteq \Desc(C)}} t^{|D|},
    \end{align*}
    where the sum runs over all pairs $(C,D)$ such that $C$ is a maximal clique and $D\subseteq \Desc(C)$ is a subset of descents of $C$. Item~(1) then follows from Equation~\eqref{eq_h_evaluation_interior_faces}. Item~(2) follows similarly from~\Cref{cor_hPolynomial_desc_asc} (2) and Equation~\eqref{eq_h_evaluation_interior_faces}.
\end{proof}

\subsection{Bijections between interior faces, descent pairs, and ascent pairs}

We define a \defn{descent pair} of the framing lattice $\scrL_{G,F}$ as a pair $(C,D)$, where $C$ is a maximal clique and $D \subseteq \Desc(C)$ is a subset of descents of $C$. The \defn{size} of $(C,D)$ is defined to be $|D|$. Similarly, an \defn{ascent pair} of $\scrL_{G,F}$ is a pair $(C,A)$, where $C$ is a maximal clique and $A \subseteq \Asc(C)$ is a subset of ascents of $C$. The \defn{size} of $(C,A)$ is defined to be $|A|$.

As shown in~\Cref{cor_enumeration_interior_faces}, the interior faces of the framed triangulation have the same cardinality as the descent pairs and the ascent pairs of the framing lattice.

\[
\begin{array}{@{}ccc@{}}
\text{Number of interior faces} 
& = \text{Number of descent pairs}
& = \text{Number of ascent pairs} \\
\text{$S\in\Delta_{G,F}$ of codimension $i$}
& \text{$(C,D)$ of $\scrL_{G,F}$ of size i}
& \text{$(C,A)$ of $\scrL_{G,F}$ of size i}
\end{array}
\]

\bigskip
\noindent
The purpose of this section is to present explicit bijections between these three sets. For this we will need the following lemma.

\begin{lemma}[{\cite[Corollary~1.2.20 and Corollary~1.2.22]{vonBell_framing_2024}}]
\label{lem_Cmin_Cmax}
    Let $S$ be a set of pairwise coherent routes in a framed graph $(G,F)$. The following hold:
    \begin{enumerate}
        \item there is a unique maximal clique $C_{\min}(S)$ that is smaller in the order $\leq_{\mathrm{rot}}^{\ccw}$ than all the maximal cliques containing $S$.
        \item there is a unique maximal clique $C_{\max}(S)$ that is bigger in the order $\leq_{\mathrm{rot}}^{\ccw}$ than all the maximal cliques containing $S$.
    \end{enumerate}
    The set of maximal cliques containing $S$ is the interval $[C_{\min}(S), C_{\max}(S)]$ of the framing lattice $\scrL_{G,F}$.
\end{lemma}

\begin{lemma}\label{lem_maps_well_defined}
    Let $S$ be an interior face of the framed triangulation $\Delta_{G,F}$. Then,
    \begin{enumerate}
        \item $D=C_{\max}(S)\smallsetminus S$ is a subset of descents of $C_{\max}(S)$.
        \item $A=C_{\min}(S)\smallsetminus S$ is a subset of ascents of $C_{\min}(S)$.
    \end{enumerate}
\end{lemma}
\begin{proof}
    We prove item (1), the proof of item (2) is analogous.
    
     Let $R \in D = C_{\max}(S) \smallsetminus S$. If $R$ admits no rotation (neither upward nor downward) within $C_{\max}(S)$, then $C_{\max}(S) \smallsetminus R$ is a boundary face. Consequently, $S \subseteq C_{\max}(S) \smallsetminus R$ would also be a boundary face, which is a contradiction. Rotating $R$ in $C_{\max}(S)$ yields another maximal clique containing $S$, so by maximality the rotation must be downward. Thus, $R$ is a descent of $C_{\max}(S)$.
\end{proof}

\begin{proposition}\label{prop_bijection_interior_pairs}
    The following maps are bijections:
    \begin{enumerate}
        \item The map from the set of interior faces of the framed triangulation~$\Delta_{G,F}$ to the set of descent pairs of $\scrL_{G,F}$ given by
        \[
            S \;\longmapsto\; (C,D)
            \quad\text{where}\quad
            \begin{cases}
                C = C_{\max}(S),\\
                D = C \smallsetminus S.
            \end{cases}
        \]
        
        \item The map from the set of interior faces of the framed triangulation~$\Delta_{G,F}$ to the set of ascent pairs of $\scrL_{G,F}$ given by
        \[
            S \;\longmapsto\; (C,A)
            \quad\text{where}\quad
            \begin{cases}
                C = C_{\min}(S),\\
                A = C \smallsetminus S.
            \end{cases}
        \]
    \end{enumerate}
    Moreover, $S$ has codimension $i$ if and only if the corresponding pairs $(C,D)$ and $(C,A)$ have size $i$.
\end{proposition}
\begin{proof}
    The two maps are well defined by~\Cref{lem_maps_well_defined}. They are injective, since $S$ is uniquely determined by its image: in the first case, $S = C \smallsetminus D$, and in the second, $S = C \smallsetminus A$.

    Moreover, by~\Cref{cor_enumeration_interior_faces}, the number of interior faces equals both the number of descent pairs and the number of ascent pairs. Hence, injectivity together with equinumeration implies that the two maps are bijections.

    Finally, $S$ has codimension $i$ when $|C\smallsetminus S|=|D|=|A|=i$.
\end{proof}

As a consequence, we get the following non-trivial result. We were not able to find a simple proof without the equinumerous argument above.

\begin{corollary}\label{cor_interior_faces_as_clique_minus_ascents}
    Let $C\in \scrL_{G,F}$ be a maximal clique. The following hold:
    \begin{enumerate}
        \item If $D\subseteq \Desc(C)$ then $S=C\smallsetminus D$ is an interior face of $\Delta_{G,F}$.
        \item If $A\subseteq \Asc(C)$ then $S=C\smallsetminus A$ is an interior face of $\Delta_{G,F}$.
    \end{enumerate}
\end{corollary}
\begin{proof}
    By the bijection in~\Cref{prop_bijection_interior_pairs} (1), there exist an interior face $S$ that maps to the descent pair $(C,D)$. Since $S=C\smallsetminus D$, then item (1) follows. Item (2) follows similarly from ~\Cref{prop_bijection_interior_pairs} (2).
\end{proof}
\section{Interior faces and pure intervals}\label{sec_interior_faces_pure_intervals}

A key step in proving that pure intervals give rise to a combinatorial model for framingtopes (\Cref{cor_def_framingtopes_one_three}) is to obtain a more refined understanding of pure intervals. One of the main goals of this section is to establish the following characterization.

\begin{theorem}\label{thm_pure_min_max}
The following hold:
\begin{enumerate}
    \item The map
\[
S \longmapsto I_S=[C_{\min}(S), C_{\max}(S)]
\]
is a bijection from the set of interior faces of the framed triangulation $\Delta_{G,F}$ to the set of pure intervals of the framing lattice $\mathscr{L}_{G,F}$.
\item This bijection satisfies
\[
S\subseteq S' \Leftrightarrow I_{S'}\subseteq I_S.
\]
\end{enumerate}

\end{theorem}

This characterization will be established using a result of independent interest, the Join-$C_{\max}$ Theorem~\ref{thm_join_max}, which provides a method for computing the join of a collection~$C_1,\dots,C_n$ of maximal cliques covering a common element $C$; see~\Cref{fig_Join_Cmax_Lemma} for an illustration. The analog result for the meet is stated in~\Cref{thm_meet_min}.

\subsection{The Join-\texorpdfstring{$C_{\max}$}{Cmax} Theorem}

Recall that a graph is called \defn{$n$-regular} if all of its vertices have degree $n$.

\begin{lemma}\label{lem_nRegular}
    If $S$ is an interior face of codimension $n$ of~$\Delta_{G,F}$, then the Hasse diagram of $[C_{\min}(S),C_{\max}(S)]$ is an $n$-regular oriented graph with a unique source $C_{\min}(S)$ and a unique sink $C_{\max}(S)$.
\end{lemma}

\begin{proof}
    By~\Cref{lem_Cmin_Cmax}, the interval $[C_{\min}(S),C_{\max}(S)]$ consists precisely of the maximal cliques containing~$S$. Hence, its dual graph agrees with the dual graph of the link of~$S$ in~$\Delta_{G,F}$. Since $S$ is an interior face of codimension~$n$ in a triangulated ball, its link is a simplicial sphere whose facets each have~$n$ vertices. It follows that the corresponding dual graph is $n$-regular. Moreover, the Hasse diagram of $[C_{\min}(S),C_{\max}(S)]$ defines an orientation of this dual graph, with the unique minimal element $C_{\min}(S)$ giving the unique source and the unique maximal element $C_{\max}(S)$ giving the unique sink.
\end{proof}

\begin{figure}[ht]
    \centering
    \includegraphics[]{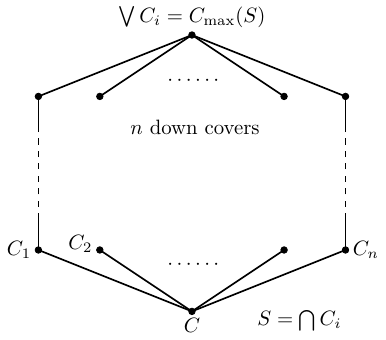}
    \caption{Illustration of the Join-$C_{\max}$ Theorem.}\label{fig_Join_Cmax_Lemma}

\end{figure}

    

\begin{lemma}\label{lem_join_ndowncovers}
    If $C_1,\dots , C_n$ are different maximal cliques covering $C\in \scrL_{G,F}$, then:
    \begin{enumerate}
        \item Their intersection $S=\bigcap C_i$ is an interior face and satisfies $C=C_{\min}(S)$.
        \item Their join $\bigvee C_i\leq C_{\max}(S)$ has $n$ down covers within the interval $[C,\bigvee C_i]$.
    \end{enumerate}
\end{lemma}

\begin{proof}
    For each $i \in [n]$, let $R_i\in C_i$ and $R_i^C\in C$ denote the two routes satisfying
    \[
    C_i\smallsetminus\{R_i\}=C\smallsetminus \{R_i^C\}.
    \] 
    In particular, the set $A=\{R_1^C,\dots,R_n^C\}$ is a subset of ascents of $C$, and
    \[
    S=\bigcap C_i = C\smallsetminus A.
    \]
    By~\Cref{cor_interior_faces_as_clique_minus_ascents}(2), we deduce that $S=C\smallsetminus A$ is an interior face. Furthermore, every route in $C\smallsetminus S=A$ can be rotated upward, implying that $C$ is a minimal element among the maximal cliques containing~$S$. Since such a minimal element is unique, it follows that~$C=C_{\min}(S)$. This finishes the proof of part (1).

    For the second part, note that $S\subseteq C_i$ and so $C_i\leq C_{\max}(S)$, for all $i$. As a consequence
    \[
    \bigvee C_i\leq C_{\max}(S).
    \]
    Let $D_1,\dots,D_m$ be the down covers of $\bigvee C_i$ in the interval $I$,
    \[
    I=[C,\bigvee C_i]\subseteq [C,C_{\max}(S)]=[C_{\min}(S),C_{\max}(S)].
    \] 
    We aim to show that $m=n$. The inequality $m\leq n$ follows directly from~\Cref{lem_nRegular}.

    Now, for each down cover $D$ of $\bigvee C_i$ there exists at least one $C_i$ such that $C_i \nleq D$, otherwise $\bigvee C_i\leq D$ which would be a contradiction. Let $j_1,\dots,j_m\in [n]$ be such that
    \[
    C_{j_k}\nleq D_k.
    \]
    We claim that
    \[
    \bigvee_{k=1}^m C_{j_k} = \bigvee_{i=1}^n C_i.
    \]
    The inequality $\leq$ is clear from the fact $\{j_1,\dots, j_m\}\subseteq [n]$. On the other hand, $\bigvee_{k=1}^m C_{j_k}$ cannot be less than or equal to any of $D_1,\dots, D_m$, so equality must hold.

    In order to prove $m\geq n$, we use the \emph{crosscut-simplicial property} established by McConville in~\cite[Theorem~6]{McConville_crosscut_2017} (see also~\cite[Theorem~2]{barnard_canonicaljoincomplex_2019}). This property states that, for any interval $[x,y]$ in a meet-semidistributive lattice, the join of any proper subset of the atoms of $[x,y]$ is strictly smaller than~$y$. Applying this to the interval $I$ in $\scrL_{G,F}$, which is meet-semidistributive by~\cite{vonBell_framing_2024}, we conclude that the set $\{j_1,\dots,j_m\}$ must contain at least $n$ distinct atoms. Hence, $m\geq n$.
\end{proof}

\begin{theorem}[The Join-$C_{\max}$ Theorem]\label{thm_join_max}
    If $C_1,\dots , C_n$ are different maximal cliques covering $C\in \scrL_{G,F}$ and $S=\bigcap C_i$, then
    \[
    C_{\max}(S) = \bigvee C_i.
    \]
\end{theorem}

\begin{proof}
    Since $S\subseteq C_i$, we have $C_i\leq C_{\max}(S)$ for all~$i$. Therefore,
    \[
    \bigvee C_i \leq C_{\max}(S).
    \]
    
    On the other hand, by~\Cref{lem_nRegular}, the Hasse diagram of $[C_{\min}(S),C_{\max}(S)]$ is an oriented $n$-regular graph with unique sink~$C_{\max}(S)$. By~\Cref{lem_join_ndowncovers}(2), the join $\bigvee C_i$ has $n$ down-covers within the interval $[C,\bigvee C_i]$. Hence, $\bigvee C_i$ must coincide with the unique sink of the interval, and therefore
    \[
    \bigvee C_i = C_{\max}(S).
    \]
\end{proof}

\subsection{The Meet-\texorpdfstring{$C_{\min}$}{Cmin}  Theorem}

Using similar arguments as in the previous section for the reverse lattice, we get analog statements for the meet operation. We state the results for completeness.

\begin{lemma}\label{lem_meet_nupcovers}
    If $C_1,\dots , C_n$ are different maximal cliques covered by $C\in \scrL_{G,F}$, then:
    \begin{enumerate}
        \item Their intersection $S=\bigcap C_i$ is an interior face and satisfies $C=C_{\max}(S)$.
        \item Their meet $\bigwedge C_i\geq C_{\min}(S)$ has $n$ up covers within the interval $[\bigwedge C_i,C]$.
    \end{enumerate}
\end{lemma}

\begin{theorem}[The Meet-$C_{\min}$ Theorem]\label{thm_meet_min}
    If $C_1,\dots , C_n$ are different maximal cliques covered by $C\in \scrL_{G,F}$ and $S=\bigcap C_i$, then
    \[
    C_{\min}(S) = \bigwedge C_i.
    \]
\end{theorem}

\subsection{Pure intervals}\label{subsec_pure_intervals}
In this section, we present three equivalent characterizations of pure intervals:
\begin{itemize}
\item in terms of ascents,
\item in terms of descents, and
\item in terms of interior faces.
\end{itemize}

The ascent characterization is the original definition: pure intervals are intervals of the form $[C,C+A]$, where $C$ is a maximal clique and $A$ is a subset of ascents of $C$.

To describe the characterization in terms of descents, let $D$ be a subset of descents of a maximal clique $C$. For each $d\in D$, let
\[
C-d := C\setminus{d}\cup{d'}
\]
denote the maximal clique covered by $C$, where $d'$ is the decreasing flip of $d$. We then define
\[
C-D := \bigwedge_{d\in D}(C-d),
\]
which gives rise to the interval $[C-D,C]$.

The next proposition relates both constructions to the characterization in terms of interior faces. The full characterization is then stated in~\Cref{cor_pureintervals_equivalence}.

\begin{proposition}\label{prop_fromPuretoCminCmax}
    Let $C\in \scrL_{G,F}$ be a maximal clique. The following hold:
    \begin{enumerate}
        \item If $D\subseteq \Desc(C)$, then $S=C\smallsetminus D$ is an interior face and
        \[
        [C-D,C]=[C_{\min}(S),C_{\max}(S)].
        \]
        \item If $A\subseteq \Asc(C)$, then $S=C\smallsetminus A$ is an interior face and
        \[
        [C,C+A]=[C_{\min}(S),C_{\max}(S)].
        \]
    \end{enumerate}  
\end{proposition}

\begin{proof}
We prove (2). By~\Cref{cor_interior_faces_as_clique_minus_ascents}(2), the set $S=C\smallsetminus A$ is an interior face.

Let $A={a_1,\ldots,a_n}$ and, for each $i$, let $C_i=C+a_i$ be the maximal clique covering $C$ obtained by the increasing flip of $a_i$. Then
\[
S=\bigcap_{i=1}^n C_i.
\]
Hence, by~\Cref{lem_join_ndowncovers},
\[
C=C_{\min}(S).
\]
Moreover,
\[
C+A=\bigvee_{i=1}^n C_i=C_{\max}(S),
\]
where the second equality follows from~\Cref{thm_join_max}. This proves (2).

Statement (1) is proved similarly using~\Cref{cor_interior_faces_as_clique_minus_ascents}(1), \Cref{lem_meet_nupcovers}, and~\Cref{thm_meet_min}.
\end{proof}

\begin{lemma}\label{lemma_interior_faces_equal_intersection}
    Let $I_S=[C_{\min}(S),C_{\max}(S)]$ be the interval associated to an interior face~$S$ of~$\Delta_{G,F}$. Then
    \[
    S=\bigcap_{C\in I_S} C.
    \]
\end{lemma}
\begin{proof}
The interval $I_S=[C_{\min}(S), C_{\max}(S)]$ consists precisely of the maximal cliques containing $S$. Hence
\[
S\subseteq \bigcap_{C\in I_S} C.
\]
We only need to show that $\bigcap_{C\in I_S} C \subseteq S$. Suppose this is not true; then there is a route $R\in (\bigcap_{C\in I_S} C) \setminus S$. Choose any $C_1\in I_S$. Since $R\in\bigcap_{C\in I_S} C$, we have $R\in C_1\setminus S$.

The route $R$ is not flippable in $C_1$; otherwise, there would be a $C_2$ such that $C_2=C_1 \setminus\set{R}\cup\set{R'}$ for a route $R'$, and $S\subseteq C_2$, which means that $C_2\in I_S$, but $R\notin C_2$, implying that $R\notin \bigcap_{C\in I_S} C$.

However, if $R$ is not flippable, then $C_1\setminus\set{R}$ is a boundary face, so $S\subseteq C_1\setminus\set{R}$ is also in the boundary, contradicting that $S$ is an interior face.
\end{proof}

\begin{remark}
    If $S$ is not interior,~\Cref{lemma_interior_faces_equal_intersection} does not necessarily hold. For instance, if $S=\emptyset$, then $I_S$ is equal to the entire lattice and $\bigcap_{C\in I_S} C$ is the set of exceptional routes. This intersection is not necessarily empty.
\end{remark}

The following result provides three equivalent characterizations of pure intervals.

\begin{corollary}[Pure intervals]\label{cor_pureintervals_equivalence}
The following collections of intervals coincide:
\begin{enumerate}
\item Intervals of the form $[C_{\min}(S),C_{\max}(S)]$, where $S$ is an interior face of $\Delta_{G,F}$.
\item Intervals of the form $[C,C+A]$, where $C$ is a maximal clique and $A\subseteq C$ is a subset of ascents of $C$.
\item Intervals of the form $[C-D,C]$, where $C$ is a maximal clique and $D\subseteq C$ is a subset of descents of $C$.
\end{enumerate}
\end{corollary}

\begin{proof}
We prove the equivalence of (1) and (2).

By~\Cref{prop_fromPuretoCminCmax}, every interval $[C,C+A]$ is of the form
\[
[C,C+A]=[C_{\min}(S),C_{\max}(S)],
\]
where $S=C\smallsetminus A$ is an interior face.

Conversely, let
\[
I_S=[C_{\min}(S),C_{\max}(S)]
\]
for an interior face $S$.

By~\Cref{lemma_interior_faces_equal_intersection}, the interior face $S$ is uniquely determined by $I_S$. Setting
\[
C=C_{\min}(S)
\qquad\text{and}\qquad
A=C\smallsetminus S,
\]
\Cref{prop_fromPuretoCminCmax} yields
\[
I_S=[C,C+A].
\]

The equivalence of (1) and (3) is proved analogously.
\end{proof}

We are finally ready to prove~\Cref{thm_pure_min_max}.

\begin{proof}[Proof of~\Cref{thm_pure_min_max}]
By~\Cref{cor_pureintervals_equivalence}, every pure interval is of the form
\[
I_S=[C_{\min}(S),C_{\max}(S)]
\]
for some interior face $S$ of $\Delta_{G,F}$. Thus the map
\[
S \longmapsto I_S
\]
is surjective onto the set of pure intervals.

To prove injectivity, let $S$ and $S'$ be interior faces such that $I_S=I_{S'}$. By~\Cref{lemma_interior_faces_equal_intersection},
\[
S=\bigcap_{C\in I_S} C
=\bigcap_{C\in I_{S'}} C
=S',
\]
and hence the map is injective. This proves~(1).

For~(2), recall that $I_S$ is precisely the set of maximal cliques containing $S$. Therefore,
\[
S\subseteq S'
\]
implies that every maximal clique containing $S'$ also contains $S$, and hence
\[
I_{S'}\subseteq I_S.
\]

Conversely, if
\[
I_{S'}\subseteq I_S,
\]
then
\[
S=\bigcap_{C\in I_S} C
\subseteq
\bigcap_{C\in I_{S'}} C
=S',
\]
where the first and last equalities follow from~\Cref{lemma_interior_faces_equal_intersection}. Thus
\[
S\subseteq S',
\]
which completes the proof.
\end{proof}

\subsection{Proof of Corollary \ref{cor_def_framingtopes_one_three}}
Now we are ready to show that the first and third definitions of framingtopes (\Cref{def_framingtope_asDual,def_framingtope_pureIntervals}) are equivalent:

\begin{itemize}
    \item The complex of interior faces $I$ of~$\Delta_{G,F}$, ordered by reverse inclusion.
    \item The collection of pure intervals of~$\mathscr{L}_{G,F}$, ordered by containment.
\end{itemize}

The bijection
\[
S \longmapsto I_S=[C_{\min}(S), C_{\max}(S)]
\]
between interior faces and pure intervals in~\Cref{cor_def_framingtopes_one_three} satisfies
\[
S\subseteq S' \Leftrightarrow I_{S'}\subseteq I_S.
\]
The result follows.
\qed
\section{Tropical framingtopes: a geometric realization}\label{sec_tropical_framingtope}
Having established three combinatorial descriptions of framingtopes, a natural question arises: \textit{can we obtain a geometric realization?} The answer is affirmative, and the tools to achieve it come from two sources, the Cayley trick and tropical geometry.

We now explain how the tropical realization arises. The argument has three steps. First, the framed triangulation is identified, via the Cayley trick, with a fine mixed subdivision of a Minkowski sum of lower-dimensional flow polytopes. Second, tropical duality identifies the cells of the corresponding tropical arrangement with the faces of this mixed subdivision. Finally, bounded tropical cells correspond precisely to interior faces of the original framed triangulation. Thus the bounded complex of the tropical arrangement is the desired framingtope.

This process is illustrated for two examples in~\Cref{fig_tropical}.

\begin{figure}[ht]
    \centering
    \includegraphics[]{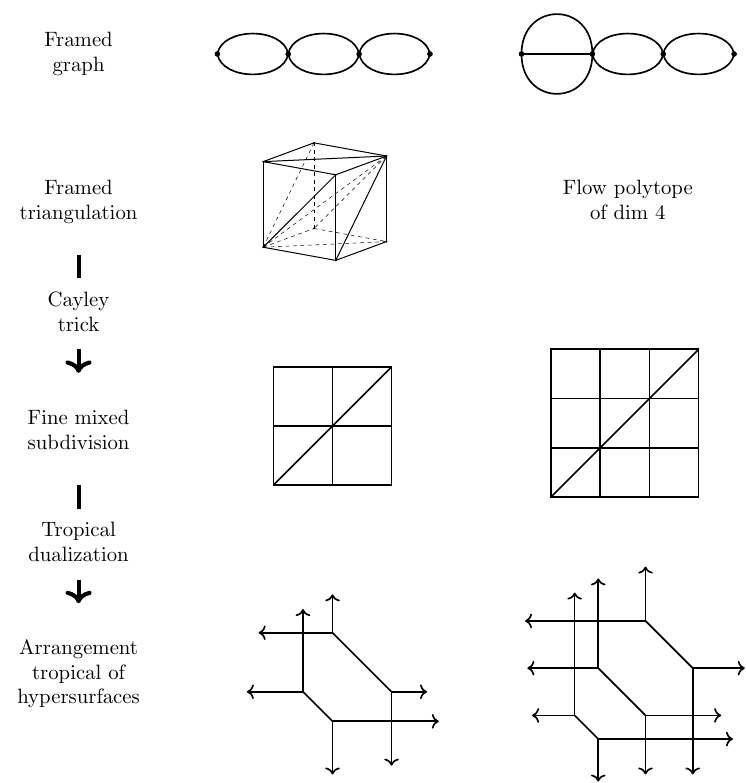}
    \caption{The tropical Cayley trick applied to framed triangulations of flow polytopes.}
    \label{fig_tropical}

\end{figure}

We begin by recalling the notion of inner normal cones.

\subsection{Normal inner cone and normal fan}

Let $P\subseteq\mathbb{R}^d$ be a polytope, not necessarily full-dimensional, and let $Q$ be a nonempty face of $P$ (possibly equal to $P$). The \defn{inner normal cone} of $Q$ in $P$ is:

\begin{equation}\label{eq_inner_normal_cone}
NC_P^{\mathrm{in}}(Q) := \set{\mathbf{x}\in\mathbb{R}^d : \langle\mathbf{x},q\rangle \leq \langle\mathbf{x},p\rangle \text{ for all } q\in Q \text{ and all } p\in P}.    
\end{equation}
In particular, $NC_P^{\mathrm{in}}(Q)$ is a cone containing $0$ for every face $Q$. Setting
\[
V_P = \mathrm{span}\set{p-q : p,q\in P}
\]
to be the space generated by $P$, the cone corresponding to $Q=P$ is its orthogonal complement:
\[
NC_P^{\mathrm{in}}(P) = V_P^\perp.
\]

If $P$ is full-dimensional, then $NC_P^{\mathrm{in}}(P)=\{0\}$. Otherwise, $NC_P^{\mathrm{in}}(P)$ is a linear subspace contained in the inner normal cone of every face $Q$ of $P$.

The \defn{inner normal fan} of $P$ is the collection of all inner normal cones $NC_P^{\mathrm{in}}(Q)$, where $Q$ ranges over the faces of $P$. It is dual to the face poset of $P$ in the sense that, for any two faces $Q_1$ and $Q_2$ of $P$,
\[
Q_1 \subseteq Q_2
\quad\Longleftrightarrow\quad
NC_P^{\mathrm{in}}(Q_2)\subseteq NC_P^{\mathrm{in}}(Q_1).
\]

The notion of inner normal cones extends naturally to faces of a subdivision $\mathcal{S}$ of $P$: for any face $Q\in\mathcal{S}$, the cone $NC_P^{\mathrm{in}}(Q)$ is defined exactly as in~\eqref{eq_inner_normal_cone}.

When $Q\in\mathcal{S}$ is an interior face of the subdivision, $NC_P^{\mathrm{in}}(Q) = V_P^\perp$. When $B\in\mathcal{S}$ is a boundary face, $NC_P^{\mathrm{in}}(B)$ is spanned by $V_P^\perp$ together with the inward-pointing normal vectors of $B$ in $P$.

\begin{example}\label{example_pentagon_cones}
Consider the pentagon $P\subset\mathbb{R}^2$ with vertices
\[
a=(0,0),\quad
b=(0,1),\quad
c=(1,2),\quad
d=(2,1),\quad
e=(1,0),
\]
subdivided as in~\Cref{fig_pentagon_normal_cones}, and take $Q_1=\operatorname{conv}\{b,c\}$, $Q_2=\operatorname{conv}\{c,e\}$, $Q_3=\operatorname{conv}\{d\}$, $Q_4=\operatorname{conv}\{a,c,e\}$. Here $NC_P^{\mathrm{in}}(P)=NC_P^{\mathrm{in}}(Q_2)=NC_P^{\mathrm{in}}(Q_4)=\set{(0,0)}$; $NC_P^{\mathrm{in}}(Q_1)=\set{\alpha(0,-1): \alpha \geq 0}$ and $NC_P^{\mathrm{in}}(Q_3)=\set{\alpha(-1,-1) + \beta(-1,1): \alpha,\beta \geq 0}$.

\begin{figure}[ht]
    \centering
    \includegraphics[]{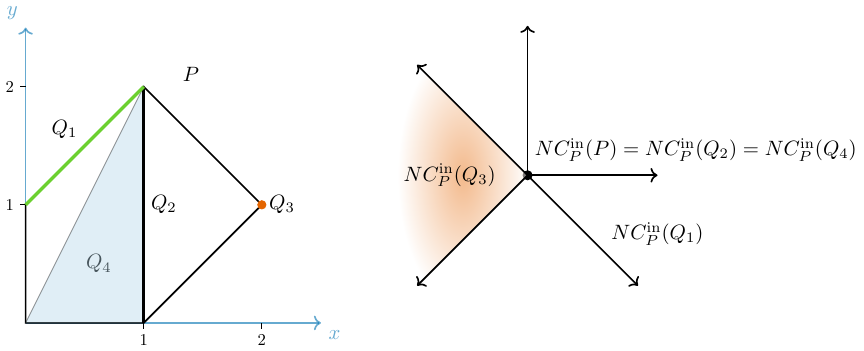}
    \caption{Inner normal cones of faces of a subdivided pentagon.}\label{fig_pentagon_normal_cones}

\end{figure}

\end{example}

\subsection{Tropical dualization}\label{section: troplicaldual}
Tropical geometry provides a natural framework for dualizing regular polyhedral subdivisions. The goal is to realize a regular subdivision as the dual of a polyhedral complex induced by a tropical hypersurface. We refer the reader o~\cite{joswig_essentials_tropical_2021,maclagan_strumfels_tropicalGeometry} for comprehensive introductions to tropical geometry and tropical combinatorics.

We work with the \defn{tropical semiring} $\mathbb{T}:=(\mathbb{R}\cup\{\infty\},\oplus,\otimes)$, where tropical addition and multiplication are defined by $x\oplus y = \min\{x,y\}$ (the minimum of $x$ and $y$) and $x\otimes y = x+y$. The \defn{tropical projective space} is
\[
\mathbb{TP}^m = \mathbb{T}^{m+1}/\mathbb{R}(1,\ldots,1),
\]
the tropical analogue of classical projective space.

Let $\mathcal{R}=\{R_1,\ldots,R_\ell\}$ be a point configuration in $\mathbb{R}^d$ (in our case, vertices of the flow polytope). A \defn{height function} $\h:\mathcal{R}\to\mathbb{R}$, with $\h_i = \h(R_i)$, induces a subdivision $\mathcal{S}^{\h}$ whose faces are the projections of the upper faces of the lifted configuration. Concretely, we lift each vertex $R_i$ to $(h_i,R_i)\in\mathbb{R}^{d+1}$, and the subdivision $\mathcal{S}^{\h}$ is the projection of the upper envelope of this lifted point set onto $\mathbb{R}^d$.

The height function $h$ is called \defn{admissible} for a subdivision $\mathcal{S}$ if the induced subdivision is $\mathcal{S}^h=\mathcal{S}$. In such case, the subdivision $\mathcal{S}$ is called \defn{regular}.

\begin{example}\label{exam_lift_cube}
    Let $P = \operatorname{conv}\{0,e_1,e_2,e_1+e_2\}$ be the unit square in $\mathbb{R}^2$. The height function $\h(0)=1$, $\h(e_1)=\h(e_2)=2$, $\h(e_1+e_2)=4$, induces the triangulation by the diagonal from $0$ to $e_1+e_2$. \Cref{fig_lift_square} shows the lift and its projection.



\begin{figure}[ht]
    \centering
    \includegraphics[]{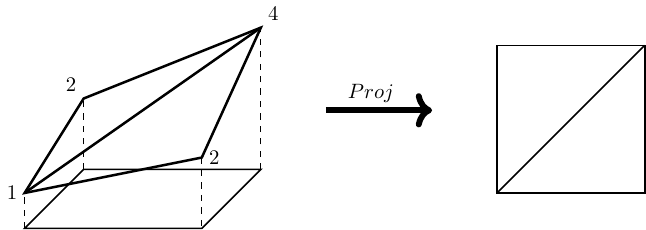}
    \caption{Projection of the lifted square.}
    \label{fig_lift_square}

\end{figure}

\end{example}

Given a height function $\h$, the subdivision $\mathcal{S}^{\h}$ is associated with the \defn{tropical hypersurface}
\[
T^{\h}(\mathbf{x}) = \bigoplus_{i=1}^{\ell} \h_i \otimes 
\mathbf{x}^{R_i} = \min\{-\h(R_i) + \langle R_i,\mathbf{x}\rangle 
: i\in[\ell]\},
\]
where $\mathbf{x}\in\mathbb{R}^d$ and $\langle\cdot,\cdot\rangle$ is the standard inner product. The hypersurface itself is the set of points where this minimum is attained at least twice.

This is the tropical analogue of taking a polynomial in $\mathbb{R}^d$ and setting it equal to~$0$. The lifted points $(\h(R_i), R_i)\in\mathbb{R}^{d+1}$ define a piecewise linear function, and we minimize the linear functional $(-1,\mathbf{x})$ over this lifted configuration (The first entry is~$-1$ because we minimize over vectors pointing downward in the lifted polytope.). The tropical hypersurface~$T^{\h}(\mathbf{x})$ decomposes~$\mathbb{R}^d$ into regions; we denote this \defn{tropical cell decomposition} by~$\mathcal{T}^{\h}$.

The following proposition summarizes the standard duality between regular subdivisions and the associated polyhedral complex of a tropical hypersurface (cf.\ \cite[§1.2, Theorem~1.13]{joswig_essentials_tropical_2021} and \cite[Chapter~3]{maclagan_strumfels_tropicalGeometry}).

\begin{proposition}[{cf.\ \cite{joswig_essentials_tropical_2021,maclagan_strumfels_tropicalGeometry}}]\label{pro: tropical_map}
For any polytope $P\subset\mathbb{R}^d$ and any height function~$\h$, the regular subdivision $\mathcal{S}^{\h}$ of $P$ is dual to the decomposition $\mathcal{T}^{\h}$.

More precisely, there is a containment order-reversing bijection
\[
\varphi:\mathcal{S}^{\h}\longrightarrow\mathcal{T}^{\h}
\]
that sends each $k$-dimensional face of $\mathcal{S}^{\h}$ to a $(d-k)$-dimensional cell of $\mathcal{T}^{\h}$.

For a face $Q\in\mathcal{S}^{\h}$, the corresponding cell is
\[
\varphi(Q)=
\left\{
\mathbf{x}\in\mathbb{TP}^d:
-\h(q)+\langle q,\mathbf{x}\rangle
\le
-\h(p)+\langle p,\mathbf{x}\rangle :
\forall q\in\operatorname{Ver}(Q)
\text{ and } p\in\operatorname{Ver}(P)
\right\},
\]
where $\operatorname{Ver}(\cdot)$ denotes the vertex set. We write $\tilde{Q}:=\varphi(Q)$.
\end{proposition}

The bijection tells us that the cells of $\mathcal{T}^{\h}$ are obtained by solving systems of linear inequalities (and equalities), which are minimized at the face $Q$. When $P$ is full-dimensional, a facet of $\mathcal{S}^{\h}$ (a codimension-zero face) maps to a unique point, and every vertex of $\mathcal{S}^{\h}$ maps to a full-dimensional region.


\begin{example}\label{example: trop_square}
Taking $P$ and $\h$ as in \Cref{exam_lift_cube}, the associated tropical hypersurface is
\[
T^{\h}(\mathbf{x})=\min\{-1,\,-2+x,\,-2+y,\,-4+x+y\}.
\]
The resulting tropical cell decomposition is shown in \Cref{fig_dual_square}.

Consider the shaded triangle $Q$, whose vertex set is
\[
\operatorname{Ver}(Q)=\{0,e_2,e_1+e_2\}.
\]
The corresponding cell $\widetilde{Q}=\varphi(Q)$ consists of all points $(x,y)$ satisfying
\[
-1=-2+y=-4+x+y\le -2+x.
\]
This system has the unique solution $(x,y)=(2,1)$, and therefore
\[
\widetilde{Q}=\{(2,1)\}.
\]

\begin{figure}[ht]
    \centering
    \includegraphics{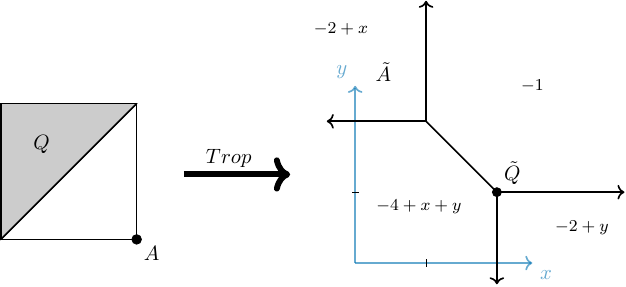}
    \caption{Dual tropicalization of a square.}
    \label{fig_dual_square}
    
\end{figure}

\end{example}

\subsection{Faces of the tropical cell decomposition}\label{sec_tropical_cells}

We now give an explicit description of the cell $\widetilde{Q}\in\mathcal{T}^{\h}$ corresponding to a face $Q\in\mathcal{S}^{\h}$.

Let $C\in\mathcal{S}^{\h}$ be a facet. The corresponding cell $\widetilde{C}=\varphi(C)$ is determined by the system of equations
\[
-\h_q+\langle q,\mathbf{x}\rangle
=
-\h_{q'}+\langle q',\mathbf{x}\rangle,
\qquad
q,q'\in\operatorname{Ver}(C).
\]
This system has a unique solution modulo the lineality space $V_P^\perp$; we denote a representative by $\widetilde{C}^\circ$. Thus
\[
\widetilde{C}=\widetilde{C}^\circ+V_P^\perp.
\]

More generally, every cell $\widetilde{Q}$ is defined by the inequalities
\[
-\h_q+\langle q,\mathbf{x}\rangle
\le
-\h_p+\langle p,\mathbf{x}\rangle,
\qquad
q\in\operatorname{Ver}(Q),\;
p\in\operatorname{Ver}(P).
\]
Ignoring the height values yields the system
\[
\langle q,\mathbf{x}\rangle
\le
\langle p,\mathbf{x}\rangle,
\]
which defines the inner normal cone $NC_P^{\mathrm{in}}(Q)$. Thus, the height function translates the bounded part of the cell while leaving its recession cone unchanged. The following classical result therefore gives the desired description.

\begin{lemma}[Minkowski--Weyl Theorem, {\cite[Theorem~1.3]{ziegler_lectures_1995}}]
A set $A\subseteq\mathbb{R}^d$ is a polyhedron if and only if there exists a polytope $B$ and a finitely generated cone $D$ such that
\[
A=B+D.
\]
\end{lemma}

The cone $D$ is called the \defn{recession cone} of $A$. The \defn{lineality space} of $A$ is the largest linear subspace contained in $A$ (or equivalently in $D$). In our setting, the recession cone of $\widetilde{Q}$ is $NC_P^{\mathrm{in}}(Q)$ and its lineality space is $V_P^\perp$, while its bounded part is the convex hull of the points $\widetilde{C}^\circ$, where $C$ ranges over the facets of $\mathcal{S}^{\h}$ containing $Q$. This yields the following theorem.

\begin{theorem}\label{thm_trop_cells}
The correspondence
\[
\varphi:\mathcal{S}^{\h}\longrightarrow\mathcal{T}^{\h}
\]
sends each face $Q$ to the polyhedron
\[
\widetilde{Q}
=
\operatorname{conv}\{\widetilde{C}^\circ:Q\subseteq C\}
+
NC_P^{\mathrm{in}}(Q).
\]
Moreover, its lineality space is
\[
V_P^\perp \subseteq NC_P^{\mathrm{in}}(Q).
\]
\end{theorem}

\begin{remark}
If $P$ is not full-dimensional, then $NC_P^{\mathrm{in}}(Q)$ contains the non-trivial lineality space $V_P^\perp$. In particular, every cell $\widetilde{Q}$ contains the translate $\widetilde{C}^\circ+V_P^\perp$ of each point $\widetilde{C}^\circ$ appearing in the convex hull above.
\end{remark}

\begin{example}\label{exam_pentagon}
Let $P$ be the pentagon as in~\Cref{example_pentagon_cones} with vertices
\[
a=(0,0),\quad
b=(0,1),\quad
c=(1,2),\quad
d=(2,1),\quad
e=(1,0),
\]
and let the height function $\h$ be given by
\[
\h(a)=0,\qquad
\h(b)=1,\qquad
\h(c)=5,\qquad
\h(d)=2,\qquad
\h(e)=1.
\]
The induced regular subdivision $\mathcal{S}^h$ is shown on the left of~\Cref{fig_pentagon_normal_cones}. Let
\[
Q_1=\operatorname{conv}\{b,c\},\qquad
Q_2=\operatorname{conv}\{c,e\},\qquad
Q_3=\{d\},\qquad
Q_4=\operatorname{conv}\{a,c,e\}.
\]
The tropical hypersurface is
\[
T^{\h}(\mathbf{x})=\min\set{-\h(i)+\langle i,\textbf{x}\rangle : i\in \set{a,b,c,d,e}},
\]
which amounts to finding $\mathbf{x}=(x,y)$ such that the minimum is attained at least twice in the set of values
\[
\set{0,
-1+y,
-5+x+2y,
-2+2x+y,
-1+x}.
\]

The solution is shown in~\Cref{fig_dual_pentagon}. To obtain the image $\tilde{Q}_4$ of $Q_4$, we can use the polynomials $0$, $-5+x+2y$, and $-1+x$ associated with $a,c,e$, to find that the only $\mathbf{x}$ minimizing these three polynomials is~$(1,2)$. The image $\tilde{Q}_2$ is the segment from $(1,2)$ to $(-1,2)$, and the image $\tilde{Q}_3$ is the point $(-1,2)$ plus the cone
\[
NC_P^{\text{in}}(Q_3)=\set{\alpha(-1,1)+\beta(-1,-1) : \alpha,\beta \geq 0}.
\]

\begin{figure}[ht]
    \centering
    \includegraphics[]{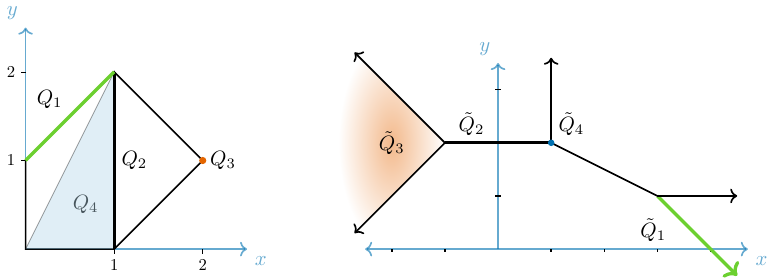}
    \caption{The tropical hypersurface $T^{\h}(\mathbf{x})$ corresponding to the subdivided pentagon.}
    \label{fig_dual_pentagon}
\end{figure}

Here, for example, $\tilde Q_3$ minimizes $-2+2x+y$.
\end{example}

\subsection{Cayley trick}

The Cayley trick relates triangulations of a Cayley embedding to mixed subdivisions of a Minkowski sum, thereby reducing the ambient dimension of the problem. In our setting, it allows us to study framed triangulations of flow polytopes via fine mixed subdivisions. The material in this section follows~\cite{HuberRambauSantos_CayleyTrick_2000}.

Given polytopes $P_1,\ldots,P_k\subseteq\mathbb{R}^d$, their \defn{Minkowski sum} is
\[
P_1+\cdots+P_k
:=
\{x_1+\cdots+x_k:x_i\in P_i\}.
\]

A \defn{Minkowski cell} of a mixed subdivision is a Minkowski sum $B=B_1+\cdots+B_k$ where each $B_i$ is the convex hull of a subset of the vertices of $P_i$; the maximal cells are full-dimensional.


A \defn{mixed subdivision} of $P_1+\cdots+P_k$ is a collection of Minkowski cells whose union is $P_1+\cdots+P_k$ and such that any two cells $B=B_1+\cdots+B_k$ and $B'=B'_1+\cdots+B'_k$ intersect properly; that is, $B_i\cap B'_i$ is a face of both $B_i$ and $B'_i$ for every $i$.

A \defn{fine mixed subdivision} is a mixed subdivision that is minimal with respect to containment of the summands.

Given the standard basis $e_1,\ldots,e_k$ of $\mathbb{R}^k$, the \defn{Cayley embedding} of $P_1,\ldots,P_k$ is the polytope
\[
C(P_1,\ldots,P_k)
:=
\operatorname{conv}\bigl((\{e_1\}\times P_1)\cup\cdots\cup(\{e_k\}\times P_k)\bigr)
\subseteq
\mathbb{R}^k\times\mathbb{R}^d.
\]

\begin{proposition}[{\cite[Section~5]{Sturmfels1994}; \cite{HuberRambauSantos_CayleyTrick_2000}}]\label{prop: subdivision_are_triangulations}
Let $P_1,\ldots,P_k$ be polytopes in $\mathbb{R}^d$. Then the polytopal subdivisions (respectively, triangulations) of the Cayley embedding $C(P_1,\ldots,P_k)$ are in bijection with the mixed subdivisions (respectively, fine mixed subdivisions) of the Minkowski sum $P_1+\cdots+P_k$.
\end{proposition}

The bijection in \Cref{prop: subdivision_are_triangulations} can be described explicitly as follows.

Let $\mathcal{S}$ be a subdivision of $C(P_1,\ldots,P_k)$. Every face of $\mathcal{S}$ containing at least one vertex from each copy $\{e_i\}\times P_i$ is of the form
\[
Q=
\operatorname{conv}\bigl((\{e_1\}\times Q_1)\cup\cdots\cup(\{e_k\}\times Q_k)\bigr),
\]
where each $Q_i$ is a non-empty face of $P_i$. To such a face $Q$ we associate the mixed cell
\[
M(Q):=Q_1+\cdots+Q_k.
\]
The corresponding mixed subdivision, denoted by $\mathcal{M}(\mathcal{S})$, consists of the cells $M(Q)$, where $Q$ ranges over all faces of $\mathcal{S}$ containing at least one vertex from each copy $\{e_i\}\times P_i$.

\begin{proposition}
\label{prop_Cayley_mixedsubdivision}
(\cite{HuberRambauSantos_CayleyTrick_2000}) Let $\mathcal{S}$ be a subdivision of $C(P_1,\ldots,P_k)$, and let $\mathcal{M}(\mathcal{S})$ be the corresponding mixed subdivision of $P_1+\cdots+P_k$. Then the map
\[
Q\longmapsto M(Q)
\]
is an order-preserving bijection between the faces of $\mathcal{S}$ containing at least one vertex from each copy $\{e_i\}\times P_i$ and the faces of $\mathcal{M}(\mathcal{S})$. Moreover, under this correspondence, interior faces of $\mathcal{S}$ correspond to interior faces of $\mathcal{M}(\mathcal{S})$.
\end{proposition}

A geometric way to see this is through the one-picture proof of \cite{HuberRambauSantos_CayleyTrick_2000}: the Cayley embedding contains the triangulated polytopes as parallel slices, and intersecting with a suitable hyperplane recovers the fine mixed subdivision (cf.~\Cref{fig_mix_sub_exmaple}).

\begin{example}\label{ex_Cayley_cube_oruga3}
Consider the unit squares
\[
P_1=P_2=\operatorname{conv}\{(0,0),(1,0),(0,1),(1,1)\},
\]
and let
\[
K=C(P_1,P_2)
\]
be their Cayley embedding. Then $K$ is a $3$-dimensional cube in $\mathbb{R}^4$ with vertices
\[
(1,0,0,0),(1,0,1,0),(1,0,0,1),(1,0,1,1),
\]
\[
(0,1,0,0),(0,1,1,0),(0,1,0,1),(0,1,1,1).
\]
The first four vertices belong to the copy $\{e_1\}\times P_1$, while the last four belong to $\{e_2\}\times P_2$. Geometrically, these are the left and right facets of the cube, respectively.

Assign the following heights to the vertices of $K$:
\[
h(1,0,0,0)=1,\quad
h(1,0,1,0)=2,\quad
h(1,0,0,1)=2,\quad
h(1,0,1,1)=4,
\]
\[
h(0,1,0,0)=2,\quad
h(0,1,1,0)=4,\quad
h(0,1,0,1)=4,\quad
h(0,1,1,1)=8.
\]
The induced regular subdivision of $K$ is shown in \Cref{fig_height_cube}.

\begin{figure}[ht]
    \centering
    \includegraphics{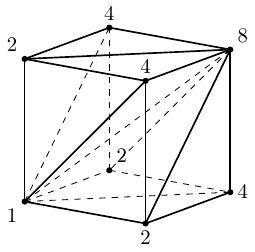}
    \caption{A regular subdivision of the cube $K=C(P_1,P_2)$ together with the height function~$\h$.}
    \label{fig_height_cube}
\end{figure}

Applying the Cayley trick transforms this regular subdivision into the fine mixed subdivision of the Minkowski sum $P_1+P_2$ shown in \Cref{fig_mix_sub_exmaple}.

\begin{figure}[ht]
    \centering
    \includegraphics[]{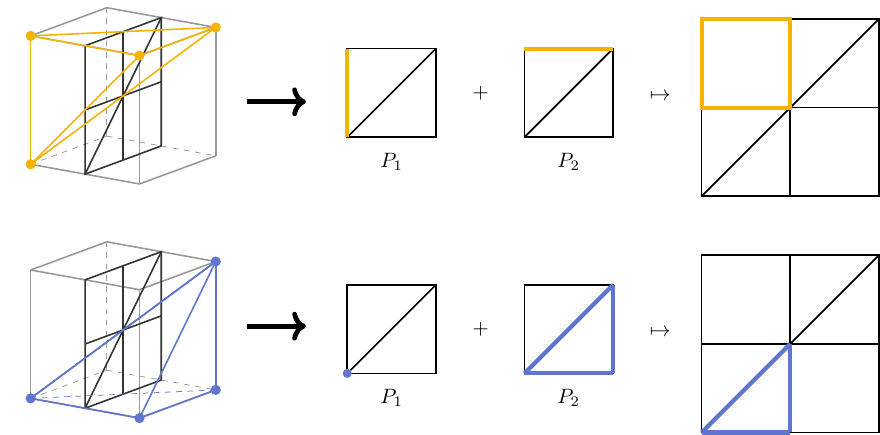}
    \caption{The Cayley trick applied to a regular subdivision of the cube.}
    \label{fig_mix_sub_exmaple}
\end{figure}

Geometrically, the Cayley trick replaces the subdivision of the Cayley embedding by its ``cross-section'' through the Minkowski sum. Each cell of the mixed subdivision records how the corresponding cell of the subdivision decomposes into summands from~$P_1$ and~$P_2$.
\end{example}


\subsection{Tropical Cayley trick}

The Cayley trick and tropical duality are closely related, as explained in \cite{santos2005cayley,Joswig2017,joswig_essentials_tropical_2021,d2023realizing}. We now describe this correspondence explicitly.

Let $P_1,\ldots,P_k$ be polytopes in $\mathbb{R}^d$, and let
\[
C(P_1,\ldots,P_k)\subseteq \mathbb{R}^k\times\mathbb{R}^d
\]
be their Cayley embedding. A height function
\[
h:\operatorname{Ver}(C(P_1,\ldots,P_k))\longrightarrow\mathbb{R}
\]
induces a regular subdivision $\mathcal{S}^h$ of the Cayley embedding. The same height function also determines an arrangement
\[
\mathsf{T}^h=\{T_i^h\}_{i=1}^k
\]
of tropical hypersurfaces, where
\[
T_i^h(\mathbf{x})
=
\min\bigl\{-h(\{e_i\}\times R)+\langle R,\mathbf{x}\rangle :
R\in\operatorname{Ver}(P_i)\bigr\},
\]
for $\mathbf{x}\in\mathbb{R}^d$ and $i\in[k]$.

Each tropical hypersurface $T_i^h$ induces a polyhedral decomposition $\mathcal{T}_i^h$ of $\mathbb{R}^d$. The common refinement of these decompositions is the tropical cell decomposition $\mathscr{T}^h$ of the arrangement $\mathsf{T}^h$. Its cells are precisely the intersections
\[
\tilde Q=\bigcap_{i=1}^k \tilde Q_i,
\]
where $\tilde Q_i$ is a cell of $\mathcal{T}_i^h$ for each $i\in[k]$.

Every face of $\mathcal{S}^h$ containing at least one vertex from each copy $\{e_i\}\times P_i$ is of the form
\[
Q=
\operatorname{conv}\bigl((\{e_1\}\times Q_1)\cup\cdots\cup(\{e_k\}\times Q_k)\bigr),
\]
where each $Q_i$ is a non-empty face of $P_i$. For each $i\in[k]$, the cell $\psi_i(Q_i)$ of $\mathcal{T}_i^h$ is defined as follows\footnote{This is the same cell as $\varphi_i(Q_i)$ from~\Cref{pro: tropical_map}. We use the notation $\psi_i$ here to emphasize that the height function is defined on the vertices of the face $\{e_i\}\times P_i$ of the Cayley embedding rather than on $P_i$ itself. The resulting cells, however, are identical.}
\[
\left\{
\mathbf{x}\in\mathbb{TP}^d:
\begin{aligned}
&-\h(\{e_i\}\times q)+\langle q,\mathbf{x}\rangle
\le
-\h(\{e_i\}\times p)+\langle p,\mathbf{x}\rangle \\
&\forall\, q\in\operatorname{Ver}(Q_i),\;
p\in\operatorname{Ver}(P_i)
\end{aligned}
\right\}.
\]
We then associate to $Q$ the cell
\[
\psi(Q)=
\bigcap_{i=1}^k \psi_i(Q_i).
\]
For convenience, we write
\[
\tilde Q_i:=\psi_i(Q_i)
\qquad\text{and}\qquad
\tilde Q:=\psi(Q).
\]

The following~\Cref{thm_Cayley_tropicalArrangement} is the corresponding statement to~\Cref{pro: tropical_map} after applying tropical duality simultaneously to the \(k\) Cayley summands.

\begin{theorem}\label{thm_Cayley_tropicalArrangement}
Let $\mathcal{S}^h$ be the regular subdivision of $C(P_1,\ldots,P_k)$ induced by a height function $h$, and let $\mathscr{T}^h$ be the tropical cell decomposition of the corresponding arrangement of tropical hypersurfaces. Then the map
\[
Q\longmapsto\psi(Q)
\]
is an order-reversing bijection between the faces of $\mathcal{S}^h$ containing at least one vertex from each copy $\{e_i\}\times P_i$ and the cells of $\mathscr{T}^h$.
\end{theorem}

\begin{proof}
The faces of $\mathcal{S}^h$ are the projections of the upper faces of the lifted Cayley embedding of $C(P_1,\dots,P_k)$. Equivalently, they are the faces minimizing linear functionals of the form
\[
(-1,\mathbf{y},\mathbf{x}),
\]
where $\mathbf{y}=(y_1,\dots,y_k)\in\mathbb{R}^k$ and $\mathbf{x}\in\mathbb{R}^d$. Thus, a face $Q$ consists of all vertices $\{e_i\}\times R$, with $R\in\operatorname{Ver}(P_i)$, for which
\begin{equation}\label{eq_minimizedfaces}
-h(\{e_i\}\times R)+y_i+\langle R,\mathbf{x}\rangle
\end{equation}
is minimized over all $R\in\operatorname{Ver}(P_i)$ and $i\in[k]$.

Suppose that $Q$ contains at least one vertex from each copy $\{e_i\}\times P_i$, and write
\[
Q=
\operatorname{conv}\bigl((\{e_1\}\times Q_1)\cup\cdots\cup(\{e_k\}\times Q_k)\bigr).
\]
Fix $i\in[k]$. Since the term $y_i$ in~\eqref{eq_minimizedfaces} is constant on the copy $\{e_i\}\times P_i$, the face $Q_i$ is precisely the face of $P_i$ minimizing the linear functional
\[
-h(\{e_i\}\times R)+\langle R,\mathbf{x}\rangle,
\qquad
R\in\operatorname{Ver}(P_i).
\]
Hence,
\[
\mathbf{x}\in\tilde Q_i=\psi_i(Q_i),
\]
so each $\tilde Q_i$ is non-empty. Therefore,
\[
\mathbf{x}\in\bigcap_{i=1}^k\psi_i(Q_i)
=\psi(Q),
\]
and thus $\psi(Q)$ is a non-empty cell of $\mathscr{T}^h$. This proves that the map
\[
Q\longmapsto\psi(Q)
\]
is well defined.

Conversely, let
\[
\tilde Q=\bigcap_{i=1}^k\tilde Q_i
\]
be a non-empty cell of $\mathscr{T}^h$, where each $\tilde Q_i$ is a non-empty cell of $\mathcal{T}_i^h$. Choose any $\mathbf{x}\in\tilde Q$ in the relative interior of each $\tilde Q_i$. For each $i\in[k]$, the cell $\tilde Q_i$ corresponds to the face~$Q_i$ of $P_i$ minimizing
\[
-h(\{e_i\}\times R)+\langle R,\mathbf{x}\rangle,
\qquad
R\in\operatorname{Ver}(P_i).
\]
The minimum values need not coincide for different $i$. However, choosing and adding constants $y_1,\dots,y_k$ so that these minima become equal, and setting $\mathbf{y}=(y_1,\dots,y_k)$, the upper face minimizing $(-1,\mathbf{y},\mathbf{x})$ projects precisely to
\[
Q=
\operatorname{conv}\bigl((\{e_1\}\times Q_1)\cup\cdots\cup(\{e_k\}\times Q_k)\bigr).
\]
Consequently,
\[
\psi(Q)
=\bigcap_{i=1}^k\psi_i(Q_i)
=\bigcap_{i=1}^k\tilde Q_i
=\tilde Q.
\]
Therefore every non-empty cell of $\mathscr{T}^h$ is the image of a face of $\mathcal{S}^h$ containing at least one vertex from each copy $\{e_i\}\times P_i$. Thus, the map $\psi$ is surjective.

To prove injectivity, observe that, for a fixed $\mathbf{x}$, the vector $\mathbf{y}$ is uniquely determined up to adding a multiple of $(1,\dots,1)$. Adding the same constant to all coordinates of $\mathbf{y}$ adds the same constant to every value of the linear functional \eqref{eq_minimizedfaces}, and therefore does not change its minimizing face. Hence the face $Q$ is uniquely determined by the cell $\tilde Q=\psi(Q)$, showing that $\psi$ is injective.

Finally, suppose
\[
Q=
\operatorname{conv}\bigl((\{e_1\}\times Q_1)\cup\cdots\cup(\{e_k\}\times Q_k)\bigr)
\subseteq
Q'=
\operatorname{conv}\bigl((\{e_1\}\times Q_1')\cup\cdots\cup(\{e_k\}\times Q_k')\bigr).
\]
Then $Q_i\subseteq Q_i'$ for every $i\in[k]$. Hence every point of $\psi_i(Q_i')$ satisfies the defining inequalities for $\psi_i(Q_i)$, and therefore
\[
\psi_i(Q_i')\subseteq\psi_i(Q_i).
\]
Intersecting over all $i$ gives
\[
\psi(Q')
=
\bigcap_{i=1}^k\psi_i(Q_i')
\subseteq
\bigcap_{i=1}^k\psi_i(Q_i)
=
\psi(Q),
\]
showing that $\psi$ is order reversing.
\end{proof}

The following corollary appears in~\cite[Corollary~4.9]{joswig_essentials_tropical_2021}; see also~\cite[Theorem~5.3]{d2023realizing}. When $P_1,\dots,P_k$ are all equal to the simplex $\Delta_{d-1}$, the Cayley embedding is the product of simplices \(\Delta_{k-1}\times\Delta_{d-1}\), the corresponding tropical hypersurfaces are tropical hyperplanes, and the associated mixed subdivisions are mixed subdivisions of the dilated simplex~\(k\Delta_{d-1}\). Thus, the corollary recovers the classical duality between mixed subdivisions of \(k\Delta_{d-1}\) and arrangements of tropical hyperplanes~\cite{santos2005cayley,develin_tropicalConvexity_2004,fink_stiefel_2015}.

\begin{corollary}[{cf.\ \cite[Corollary 4.9]{joswig_essentials_tropical_2021}, \cite[Theorem 5.3]{d2023realizing}}]
\label{cor_tropicalCayleyMixedSubdivision}
Let $\mathcal{S}^h$ be the regular subdivision of $C(P_1,\ldots,P_k)$ induced by a height function $h$. Then the tropical cell decomposition~$\mathscr{T}^h$ of the corresponding arrangement of tropical hypersurfaces is dual to the mixed subdivision~$\mathcal{M}(\mathcal{S}^h)$.
\end{corollary}

\begin{proof}
By Proposition~\ref{prop_Cayley_mixedsubdivision}, the map
\[
Q\longmapsto M(Q)
\]
is an order-preserving bijection between the faces of $\mathcal{S}^h$ containing at least one vertex from each copy $\{e_i\}\times P_i$ and the faces of the mixed subdivision $\mathcal{M}(\mathcal{S}^h)$.

By Theorem~\ref{thm_Cayley_tropicalArrangement}, the map
\[
Q\longmapsto\psi(Q)
\]
is an order-reversing bijection between the same collection of faces and the cells of $\mathscr{T}^h$.

Composing these two bijections yields an order-reversing bijection between the faces of $\mathcal{M}(\mathcal{S}^h)$ and the cells of $\mathscr{T}^h$. Hence $\mathscr{T}^h$ is dual to the mixed subdivision $\mathcal{M}(\mathcal{S}^h)$.
\end{proof}

As in~\Cref{thm_trop_cells}, the cells of the tropical cell decomposition $\mathscr{T}^h$ admit the following explicit description.

\begin{theorem}\label{thm_trop_cells_arrangement}
The correspondence
\[
\psi:\mathcal{S}^h\longrightarrow\mathscr{T}^h
\]
sends each face
\[
Q=
\operatorname{conv}\bigl((\{e_1\}\times Q_1)\cup\cdots\cup(\{e_k\}\times Q_k)\bigr)
\]
to the polyhedron
\[
\widetilde{Q}
=
\operatorname{conv}\{\widetilde{C}^\circ:Q\subseteq C\}
+
\bigcap_{i=1}^k NC_{P_i}^{\mathrm{in}}(Q_i),
\]
where $C$ ranges over the facets of $\mathcal{S}^h$ containing $Q$, and $\widetilde{C}^\circ$ is any representative point of $\psi(C)$. Moreover, its lineality space is
\[
\bigcap_{i=1}^k V_{P_i}^{\perp}= V_{P_1+\cdots+P_k}^\perp
\subseteq
\bigcap_{i=1}^k NC_{P_i}^{\mathrm{in}}(Q_i).
\]
\end{theorem}

\begin{proof}
By definition,
\[
\widetilde{Q}=\psi(Q)=\bigcap_{i=1}^k\psi_i(Q_i).
\]
By~\Cref{thm_trop_cells}, the recession cone of $\psi_i(Q_i)$ is $NC_{P_i}^{\mathrm{in}}(Q_i)$. Since the recession cone of an intersection of polyhedra is the intersection of their recession cones, it follows that the recession cone of $\widetilde{Q}$ is
\[
\bigcap_{i=1}^k NC_{P_i}^{\mathrm{in}}(Q_i).
\]

The lineality space of $\widetilde{Q}$ is the largest linear subspace of its recession cone. Since the lineality space of an intersection of polyhedra is the intersection of their lineality spaces, and the lineality space of $NC_{P_i}^{\mathrm{in}}(Q_i)$ is $V_{P_i}^{\perp}$, we obtain
\[
\bigcap_{i=1}^k V_{P_i}^{\perp} = V_{P_1+\cdots+P_k}^{\perp}.
\]

Finally, after quotienting by the lineality space, the bounded part of $\widetilde{Q}$ has vertices~$\widetilde{C}^{\circ}$, where $C$ ranges over the facets of $\mathcal{S}^h$ containing $Q$, and $\widetilde{C}^{\circ}$ is any representative point of $\psi(C)$. Therefore,
\[
\widetilde{Q}
=
\operatorname{conv}\{\widetilde{C}^{\circ}:Q\subseteq C\}
+
\bigcap_{i=1}^k NC_{P_i}^{\mathrm{in}}(Q_i).
\qedhere
\]
\end{proof}

\begin{corollary}
The lineality space of the tropical cell decomposition $\mathscr{T}^h$ is the orthogonal complement
\[
V_{P_1+\cdots+P_k}^{\perp}.
\]
\end{corollary}

\begin{proof}
The claim follows immediately from \Cref{thm_trop_cells_arrangement}, since every cell of $\mathscr{T}^h$ has lineality space $V_{P_1+\cdots+P_k}^{\perp}$.
\end{proof}

\begin{corollary}\label{cor_inteorior_vs_tropBounded}
If $P_1+\cdots+P_k$ is full-dimensional, then:
\begin{enumerate}
    \item the bounded cells of $\mathscr{T}^h$ are dual to the interior faces of the mixed subdivision $\mathcal{M}(\mathcal{S}^h)$; \item the complex of interior faces of $\mathcal{S}^h$ is dual to the complex of bounded cells of $\mathscr{T}^h$.
\end{enumerate}
\end{corollary}

\begin{proof}
By \Cref{cor_tropicalCayleyMixedSubdivision}, the cells of $\mathscr{T}^h$ are dual to the faces of the mixed subdivision~$\mathcal{M}(\mathcal{S}^h)$.

Since $P_1+\cdots+P_k$ is full-dimensional, the lineality space of $\mathscr{T}^h$ is
\[
V_{P_1+\cdots+P_k}^{\perp}=\{0\}.
\]
Hence the bounded cells of $\mathscr{T}^h$ are precisely the cells dual to the interior faces of $\mathcal{M}(\mathcal{S}^h)$. This proves~(1).

Finally, by \Cref{prop_Cayley_mixedsubdivision}, the complexes of interior faces of $\mathcal{S}^h$ and $\mathcal{M}(\mathcal{S}^h)$ are isomorphic. Therefore the complex of interior faces of $\mathcal{S}^h$ is dual to the complex of bounded cells of $\mathscr{T}^h$, proving~(2).
\end{proof}

\begin{remark}
If $P_1+\cdots+P_k$ is not full-dimensional, then every cell of the tropical cell decomposition $\mathscr{T}^h$ is unbounded, since each contains the lineality space $V_{P_1+\cdots+P_k}^{\perp}$.
\end{remark}

\begin{example}\label{ex_tropical_cube}
Consider the unit squares
\[
P_1=P_2=\operatorname{conv}\{(0,0),(1,0),(0,1),(1,1)\},
\]
together with the height function on the vertices of the Cayley embedding $K=C(P_1,P_2)$ from \Cref{ex_Cayley_cube_oruga3}. Applying the tropical Cayley trick to the resulting regular subdivision yields an arrangement consisting of the two tropical hypersurfaces
\[
T_1^h = \min\{-1,-2+x,-2+y,-4+x+y\},
\]
and
\[
T_2^h = \min\{-2,-4+x,-4+y,-8+x+y\}.
\]

The corresponding tropical cell decomposition is depicted in \Cref{fig_tropical_minkowski_sum}. By \Cref{cor_tropicalCayleyMixedSubdivision}, it is dual to the fine mixed subdivision of $P_1+P_2$ shown in \Cref{fig_mix_sub_exmaple}. In particular, vertices, edges, and two-dimensional cells of the tropical decomposition correspond to cells of complementary dimension in the mixed subdivision.

\begin{figure}[ht]
    \centering
    
    \includegraphics[]{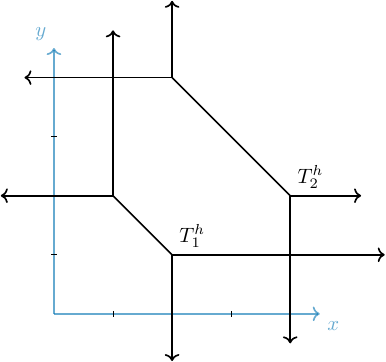}
    \caption{The arrangement of tropical hypersurfaces associated with the regular subdivision of the cube from \Cref{ex_Cayley_cube_oruga3}.}
    \label{fig_tropical_minkowski_sum}
\end{figure}
\end{example}

\subsection{The tropical framingtope}

We now apply the tools developed above to construct a geometric realization of the framingtope. The main intuitive idea is to express the framed triangulation of a flow polytope $\calF_G$ as a fine mixed subdivision and then pass to its dual tropical cell decomposition. This construction relies on a decomposition of every flow polytope as a Cayley embedding of smaller flow polytopes, which we now describe.

Let $G$ be a flow graph, and denote by $e_{i,k}$ the $k$-th outgoing edge of vertex $i+1$, ordered from bottom to top, with $k=0$ for the bottommost edge. Let be the source edges of $G$:
\[
\mathcal{O}_0=\{e_{0,0},\ldots,e_{0,j_0}\}.
\]
For $i\in\{0,\dots,j_0\}$, let $\hat G_i$ denote the smallest subgraph of $G$ containing every route that starts with the source edge $e_{0,i}$. Then the flow polytope $\calF_{\hat G_i}$ can be written as
\[
\calF_{\hat G_i} = \{e_{0,i}\}\times\calF_{G_i},
\]
where $\calF_{G_i}$ is the projection of $\calF_{\hat G_i}$ obtained by deleting the coordinate corresponding to the source edge $e_{0,i}$. Equivalently, $\calF_{G_i}$ is the flow polytope of the graph $G_i$ obtained from~$\hat G_i$ by deleting the edge $e_{0,i}$.

Moreover, the flow polytope of $G$ is the Cayley embedding
\begin{equation}\label{eq_Cayley_flowPolytope}
\calF_G = C(\calF_{G_0},\ldots,\calF_{G_{j_0}}).
\end{equation}
Each $\calF_{\hat G_i}$ is a face of $\calF_G$. A framing $F$ of $G$ naturally induces a framing on both $\hat G_i$ and $G_i$, which we also denote by $F$. Likewise, the framed triangulation $\Delta_{\hat G_i,F}$ is the restriction of $\Delta_{G,F}$ to the face $\calF_{\hat G_i}$, and may be naturally identified with $\Delta_{G_i,F}$.

By~\Cref{prop: subdivision_are_triangulations}, the framed triangulation $\Delta_{G,F}$ can therefore be viewed, via the Cayley trick, as the corresponding fine mixed subdivision of
\begin{equation}
\calF_{G_0}+\cdots+\calF_{G_{j_0}}.    
\end{equation}
This viewpoint is the key to our construction of the framingtope. It allows us to apply the tropical Cayley trick to the framed triangulation while working in a lower-dimensional ambient space.

\begin{example}\label{example: multioruga_prisma}
Let $(G,F)$ be the framed graph in~\Cref{fig_multioruga_prisma} with the framing induced by the drawing. The flow polytopes $\mathcal{F}_{G_0}$, $\mathcal{F}_{G_1}$, and $\mathcal{F}_{G_2}$ are all squares. The fine mixed subdivision associated with the framed triangulation of $\mathcal{F}_G$ is shown in~\Cref{fig_mixed_prisma}.


\begin{figure}[ht]
    \centering
    \includegraphics[]{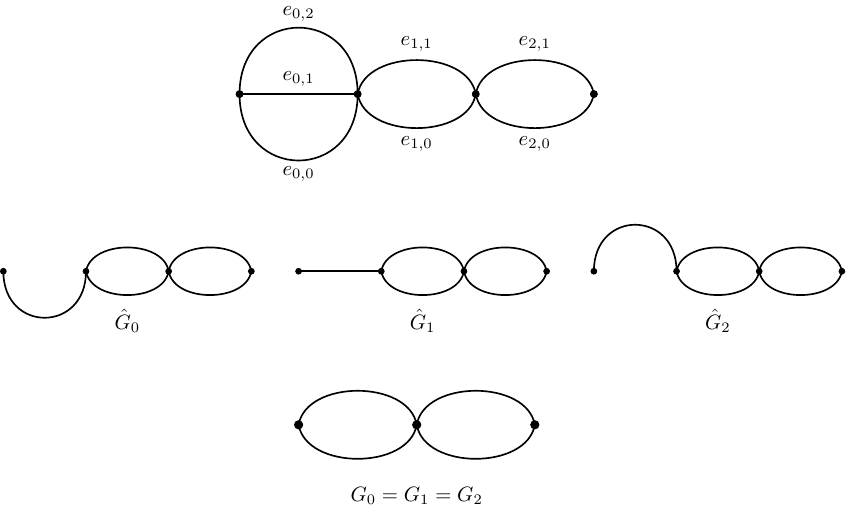}
    \caption{A graph $G$ and its induced subgraphs $\hat G_0$, $\hat G_1$, and $\hat G_2$. In this example, deleting the initial edge of each $\hat G_i$ yields the same graph, namely $G_0=G_1=G_2$.}
    \label{fig_multioruga_prisma}
\end{figure}

\begin{figure}[ht]
    \centering
    \includegraphics[]{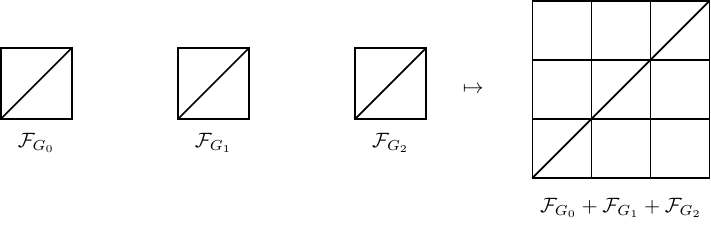}
    \caption{Fine mixed subdivision associated with the framed triangulation $\Delta_{G,F}$ of $\calF_{G}$.}
    \label{fig_mixed_prisma}
\end{figure}

Note that $\mathcal{F}_G$ is a 4-dimensional polytope living in a 7-dimensional space. However, using the Cayley trick, we can study its triangulation in a 2-dimensional space.
\end{example}

To tropicalize, we would in principle need to work in $\mathbb{R}^7$ since $\mathcal{F}_G\subset\mathbb{R}^7$. However, the Cayley trick allows us to work in the dimension of the individual pieces $\mathcal{F}_{G_i}$, and we can further reduce to the intrinsic dimension via the following projection.

For each $i=0,\dots,n-1$, label the outgoing edges of vertex $i+1$ in $G$ by
\[
e_{i,0},e_{i,1},\dots,e_{i,j_i},
\]
from bottom to top, as illustrated in~\Cref{fig_multioruga_prisma}.

Recall that two polytopes $P$ and $Q$ are \defn{integrally equivalent} if there exists an affine lattice isomorphism between them.

\begin{proposition}[Projected realization of a flow polytope]
\label{prop: projection_flowpolytopes}
Let $\mathcal{F}_G$ be the flow polytope associated with the graph $G$ and denote by $f(e_{i,j})$ the flow associated to the edge $e_{i,j}$. Let
\[
\phi:\mathbb{R}^{E(G)}
\longrightarrow
\mathbb{R}^{|E(G)|-|V(G)|+1}
\]
be the projection obtained by deleting the coordinates
\[
f(e_{0,0}),\,f(e_{1,0}),\,\ldots,\,f(e_{n-2,0}).
\]
Then $\mathcal{F}_G$ is integrally equivalent to $\phi(\mathcal{F}_G)$.
\end{proposition}

\begin{proof}
Let $G$ have $n$ vertices. By definition, the map
\[
\phi:\mathbb{R}^{E(G)}
\longrightarrow
\mathbb{R}^{|E(G)|-n+1}
\]
forgets the coordinates corresponding to the edges $e_{0,0},e_{1,0},\ldots,e_{n-2,0}$.

To prove that $\phi$ restricts to a bijection on $\mathcal{F}_G$, it suffices to show that the omitted coordinates are uniquely determined by the remaining ones. Let $v_i$ be a vertex of $G$, and let
\[
\mathcal{O}_i'=\mathcal{O}_i\setminus\{e_{i,0}\}
\]
be the set of outgoing edges of $v_i$ other than $e_{i,0}$. By flow conservation at $v_i$, we have
\[
f(e_{i,0}) = \sum_{e\in\mathcal{I}_i}f(e) - \sum_{e\in\mathcal{O}_i'}f(e),
\]
for $i\neq 0$. At the source vertex,
\[
f(e_{0,0}) = 1-\sum_{e\in\mathcal{O}_0'}f(e),
\]
since the total outgoing flow is equal to $1$.

Hence every omitted coordinate is uniquely determined by the remaining coordinates. Therefore $\phi$ induces an affine bijection between $\mathcal{F}_G$ and $\phi(\mathcal{F}_G)$. Since $\phi$ is a coordinate projection, this bijection preserves the integer lattice, and thus $\mathcal{F}_G$ and $\phi(\mathcal{F}_G)$ are integrally equivalent.
\end{proof}

\begin{remark}
By \Cref{prop: projection_flowpolytopes}, we may identify the flow polytope $\mathcal{F}_G$ with its projected realization $\phi(\mathcal{F}_G)$ and therefore omit the coordinates corresponding to the edges $e_{0,0},e_{1,0},\ldots,e_{n-2,0}$. Consequently, the tropical dualization can be performed on~$\phi(\mathcal{F}_G)$, whose ambient space has the same dimension as the flow polytope.
\end{remark}

We now turn to admissible height functions for framed triangulations. Recall that a height function $\h$ assigns a real number to each route of $G$, and that $\h$ is admissible for the framed triangulation $\Delta_{G,F}$ if the regular triangulation induced by $\h$ coincides with~$\Delta_{G,F}$. The following \Cref{lemma_admisible_heigth_function} provides a simple criterion for admissibility in terms of the framing.

\begin{definition}\label{def_top_bot}
Let $(G,F)$ be a framed graph. We say that two routes $R_1$ and $R_2$ are \defn{incoherent along exactly one path} $P\subseteq R_1\cap R_2$ if $R_1$ and $R_2$ are incoherent at every vertex $v\in P$ and nowhere else.

If $R_1<^{\mathrm{cw}}_v R_2$, we define their \defn{top part} and \defn{bottom part}, respectively, by
\[
\Top(R_1,R_2):=R_1\,v\,R_2,
\qquad
\Bot(R_1,R_2):=R_2\,v\,R_1.
\]
Thus, $\Top(R_1,R_2)$ follows $R_1$ up to $v$ and then continues along $R_2$, whereas $\Bot(R_1,R_2)$ follows $R_2$ up to $v$ and then continues along $R_1$. If $R_2<^{\mathrm{cw}}_v R_1$, we instead define
\[
\Top(R_1,R_2):=R_2\,v\,R_1,
\qquad
\Bot(R_1,R_2):=R_1\,v\,R_2.
\]
\end{definition}

\begin{remark}
The previous definition is natural in view of \cite[Lemma~1.1.10]{vonBell_framing_2024}, which shows that if $C_1$ and $C_2$ are adjacent maximal cliques satisfying $C_2=(C_1\setminus R_1)\cup R_2$, then $R_1$ and $R_2$ are incoherent along exactly one path.
\end{remark}

\begin{lemma}[cf.~{\cite[Lemmas~5.4 and~5.5]{d2023realizing}}]
\label{lemma_admisible_heigth_function}
Let $(G,F)$ be a framed graph. A height function $\h:R(G)\to\mathbb{R}$ is admissible for the framed triangulation $\Delta_{G,F}$ if and only if, for every pair of routes $R_1$ and $R_2$ that are incoherent along exactly one path we have
\[
\h(\Top(R_1,R_2)) +\h(\Bot(R_1,R_2)) > \h(R_1)+\h(R_2).
\]
\end{lemma}

\begin{definition}[Arrangement of tropical hypersurfaces for framingtopes]
\label{def_framingtopes_arrangement} Given an admissible height function $\h$ for the framed triangulation $\Delta_{G,F}$, the tropical Cayley trick determines an arrangement
\[
\mathsf{T}^h=\{T_i^h\}_{i=0}^{j_0}
\]
of tropical hypersurfaces associated with the framed triangulations $\Delta_{G_i,F}$. They are given explicitly by
\begin{equation}\label{eq_framingtope_trop_hypersurfaces}
T_i^h(\mathbf{x}) = \min\bigl\{ -h(e_{0,i}R)+\langle\vec{R},\mathbf{x}\rangle : R\in R(G_i) \bigr\}.
\end{equation}
For each $i=0,\dots,j_0$, the minimum ranges over all routes of $G_i$, and $e_{0,i}R$ denotes the route of $G$ obtained by prepending the edge $e_{0,i}$ to $R$.

The variable
\[
\mathbf{x}=(x_{i,j})_{(i,j)\in I}
\]
is indexed by
\[
I=\{(i,j):1\le i\le n-2,\;1\le j\le j_i\},
\]
so there is one coordinate for each edge $e_{i,j}$ with $1\le i\le n-2$ and $1\le j\le j_i$. Thus, $\mathbf{x}$ records precisely the coordinates that remain after omitting the source edges ($i\neq 0$) and projecting each flow polytope $\mathcal{F}_{G_i}$ by forgetting its distinguished edge ($j\neq 0$).

Finally, $\vec{R}\in\{0,1\}^{I}$ denotes the incidence vector of the route $R$: its $(i,j)$-coordinate is equal to $1$ if $R$ uses the edge $e_{i,j}$, and $0$ otherwise. Since $(i,j)\in I$ requires $i,j\ge1$, neither the source edges nor the distinguished edges $e_{i,0}$ contribute to $\vec{R}$.

We denote by $\mathscr{T}^h$ the tropical cell decomposition of the arrangement $\mathsf{T}^h$.
\end{definition}

\begin{example}\label{ex_tropicalization_prisma}

Let $(G,F)$ be the framed graph from \Cref{example: multioruga_prisma}. Consider the height function $\h(R)$ defined as the number of routes weakly below $R$, including $R$ itself. For example, the bottommost route has height $1$, since no other route lies weakly below it, while the topmost route has height $12$, since all $12$ possible routes lie weakly below it. As shown below (see \Cref{prop_ABBE_height_admissible,prop_plane_framings_are_ABBE}), this height function is admissible for the framed triangulation $\Delta_{G,F}$.

In this case, the index set is
\[
I=\{(1,1),(2,1)\},
\]
so there are only two variables:
\[
x:=x_{1,1}, \qquad y:=x_{2,1},
\]
corresponding to the edges $e_{1,1}$ and $e_{2,1}$, respectively. The tropical hypersurfaces associated with the restricted triangulations of $\mathcal{F}_{G_0}$, $\mathcal{F}_{G_1}$, and $\mathcal{F}_{G_2}$ are
\[
T_{0}^h=\min\{-1,-2+x,-2+y,-4+x+y\},
\]
\[
T_{1}^h=\min\{-2,-4+x,-4+y,-8+x+y\},
\]
\[
T_{2}^h=\min\{-3,-6+x,-6+y,-12+x+y\}.
\]

The corresponding arrangement of tropical hypersurfaces is shown in \Cref{fig_Trop_3_squares}.



\begin{figure}[ht]
\centering
    \includegraphics[]{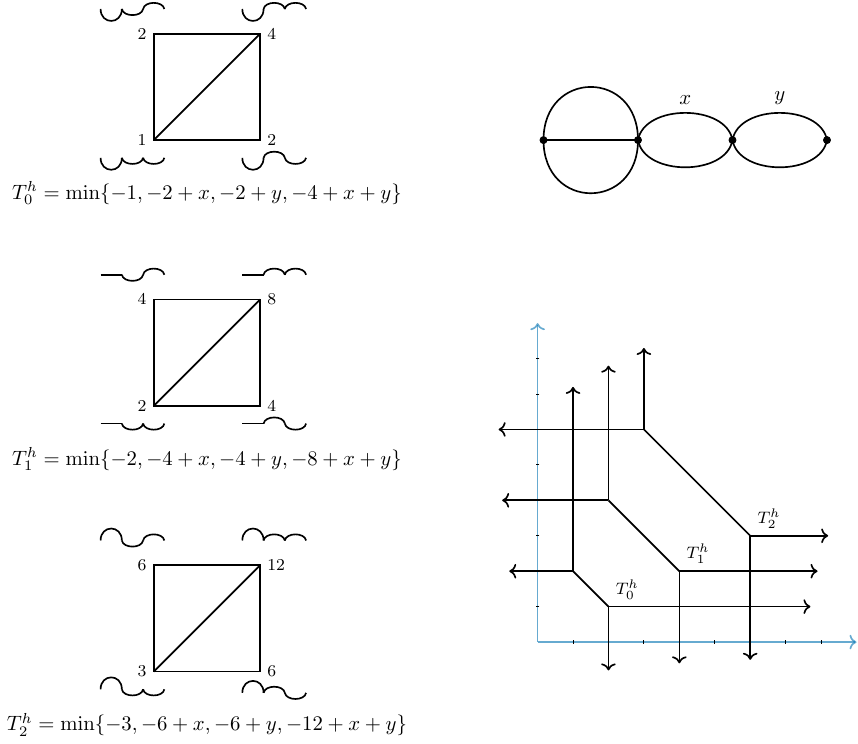}
    \caption{The arrangement of tropical hypersurfaces associated with the framed graph $(G,F)$ from \Cref{example: multioruga_prisma}.}
    \label{fig_Trop_3_squares}
\end{figure}

\end{example}

\begin{theorem}\label{thm_trop_realization_cells}
Let $\h$ be an admissible height function for the framed triangulation~$\Delta_{G,F}$, and let $\mathscr{T}^h$ be the tropical cell decomposition of the corresponding arrangement of tropical hypersurfaces~\eqref{eq_framingtope_trop_hypersurfaces} in~\Cref{def_framingtopes_arrangement}. Then:
\begin{enumerate}
    \item the cells of $\mathscr{T}^h$ are in bijection with the cliques whose routes contain a route starting with each of the source edges
    \[
    e_{0,0},e_{0,1},\ldots,e_{0,j_0};
    \]
    \item the complex of bounded cells of $\mathscr{T}^h$ is dual to the complex of interior faces of~$\Delta_{G,F}$.
\end{enumerate}
\end{theorem}

\begin{proof}
By \Cref{thm_Cayley_tropicalArrangement}, the cells of $\mathscr{T}^h$ are in bijection with the faces of $\Delta_{G,F}$ containing at least one vertex from each copy in the Cayley embedding
\[
C(\mathcal{F}_{G_0},\ldots,\mathcal{F}_{G_{j_0}}).
\]
Equivalently, they correspond to the faces of the framed triangulation that contain at least one vertex from each of the flow polytopes $\mathcal{F}_{\hat G_i}$. In terms of the graph $G$, this condition means that the corresponding cliques contain routes starting with each of the source edges
\[
e_{0,0},e_{0,1},\ldots,e_{0,j_0}.
\]
This proves the first statement.

For the second statement, we use the projected realizations of the flow polytopes $\mathcal{F}_{G_i}$ introduced above. With these realizations, the Minkowski sum
\[
\mathcal{F}_{G_0}+\cdots+\mathcal{F}_{G_{j_0}}
\]
is full-dimensional. Indeed, the dimension of the underlying space is the number of edges of the form $e_{i,j}$ with $i\neq 0$ (non-source edges) and $j\neq 0$ (non-bottom-most outgoing edges). Every such an edge belongs to at least one of the graphs $G_i$, and the projected flow polytope $\mathcal{F}_{G_i}$ is full dimensional on its underlying space.

Hence, we can apply \Cref{cor_inteorior_vs_tropBounded}, which implies that the bounded cells of $\mathscr{T}^h$ are dual to the interior faces of the mixed subdivision associated with the Cayley embedding
\[
C(\mathcal{F}_{G_0},\ldots,\mathcal{F}_{G_{j_0}}).
\]
By the Cayley trick, the interior faces of this mixed subdivision are in bijection with the interior faces of $\Delta_{G,F}$. Therefore, the complex of bounded cells of $\mathscr{T}^h$ is dual to the complex of interior faces of $\Delta_{G,F}$.
\end{proof}

This result motivates the following definition.

\begin{definition}[Tropical framingtopes: geometric realization]
\label{def_framingtope_asTropical}
Let $(G,F)$ be a framed graph, and let $\h$ be an admissible height function for the framed triangulation $\Delta_{G,F}$. The \defn{tropical framingtope} $\mathscr{P}_{G,F}^h$ is the polytopal complex formed by the bounded cells of the tropical cell decomposition $\mathscr{T}^h$ of the corresponding arrangement of tropical hypersurfaces~\eqref{eq_framingtope_trop_hypersurfaces} in~\Cref{def_framingtopes_arrangement}.
\end{definition}

\begin{corollary}
\label{cor.geometric_realization}
Let $(G,F)$ be a framed graph, and let $\h$ be an admissible height function for the framed triangulation $\Delta_{G,F}$. Then the framingtope $\mathscr{P}_{G,F}$ is geometrically realized by the polytopal complex $\mathscr{P}_{G,F}^h$.
\end{corollary}

\begin{example}\label{ex_trop_framingtope_multi211}
In \Cref{ex_tropicalization_prisma}, the tropical framingtope is obtained by restricting the tropical cell decomposition to its bounded cells; see \Cref{fig_trop_max_clique}. The vertices of this realization correspond to the maximal cliques of the framed graph. Their coordinates can be computed directly from the tropical hypersurfaces in \eqref{eq_framingtope_trop_hypersurfaces}.

For the maximal clique shown in \Cref{fig_trop_max_clique}, the corresponding vertex is determined by the minimizing equalities
\[
-1=-2+y,\qquad
-4+y,\qquad
-6+y=-12+x+y,
\]
coming from the tropical hypersurfaces $T_0^h$, $T_1^h$, and $T_2^h$, respectively. Solving these equations yields the vertex
\[
(x,y)=(6,1).
\]


\begin{figure}[ht]
    \includegraphics[]{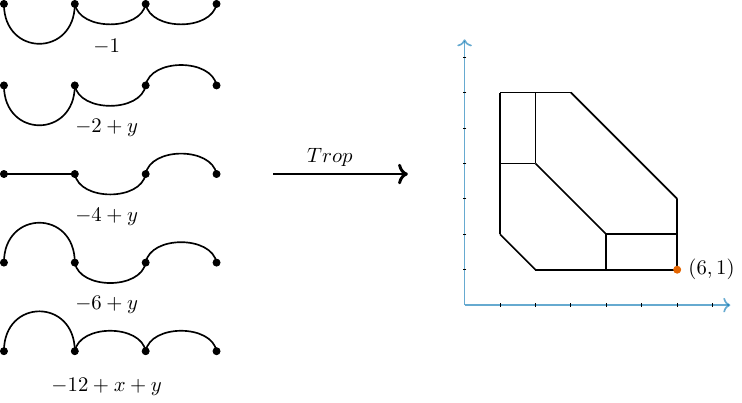}
    \caption{A maximal clique and its corresponding vertex in the tropical framingtope.}
    \label{fig_trop_max_clique}
\end{figure}
\end{example}

As a final consequence, we obtain the following.

\begin{corollary}
Let $(G,F)$ be a framed graph, and let $\h$ be an admissible height function for the framed triangulation $\Delta_{G,F}$. Then the Hasse diagram of the framing lattice $\mathscr{L}_{G,F}$ is the $1$-skeleton of the tropical framingtope $\mathscr{P}_{G,F}^h$.
\end{corollary}

\section{Framingtopes for plane framed graphs}

The previous section provides tropical geometric realizations $\mathscr{P}_{G,F}^h$ of framingtopes that depend on an admissible height function $h$ for the framed triangulation $\Delta_{G,F}$. While admissible height functions for arbitrary framed graphs $(G,F)$ were constructed in~\cite{dkktriangulation}, simpler choices are possible for special families of examples. In this section, we illustrate our constructions for the family of plane framed graphs. The resulting realizations are more convenient and, in particular, can be chosen to have vertices with integer coordinates.

\begin{definition}\label{def_plane_framed_graph}
A \defn{plane framed graph} is a framed graph $(G,F)$ such that $G$ is a flow graph with a crossing-free embedding in the plane, whose vertices are located from left to right in increasing order, whose edges are oriented monotonically from left to right, and whose framing $F$ is induced by the embedding. See~\Cref{fig_plane_framed_graph}.
\end{definition}

\begin{figure}[ht]
\centering
\includegraphics[]{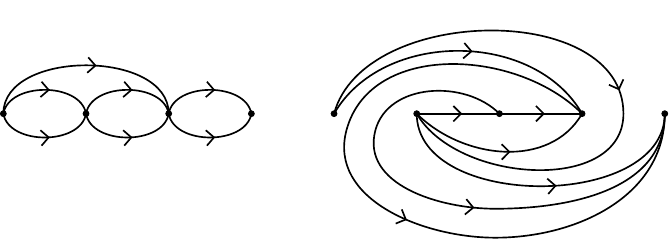}
    \caption{Example of a plane framed graph (left) and a plane graph that is not a plane framed graph (right), since its edges are not oriented monotonically from left to right.}
    \label{fig_plane_framed_graph}
\end{figure}

Plane framed graphs form a particularly important family in the study of flow polytopes. M\'esz\'aros, Morales, and Striker showed that flow polytopes of plane graphs\footnote{We include the condition that the edges are oriented monotonically from left to right, which removes an ambiguity in the construction, cf.~\Cref{fig_plane_framed_graph}.} are integrally equivalent to certain order polytopes, and established a correspondence between their framed triangulation and Stanley's triangulation of order polytopes in terms of linear extensions \cite{mezarosmoralesstriker}.

Let $(G,F)$ be a plane framed graph on the vertex set $[n]$. For each $i=0,\dots,n-1$, label the outgoing edges of vertex $i+1$ in $G$ as above by
\[
e_{i,0},e_{i,1},\dots,e_{i,j_i},
\]
ordered from bottom to top, with $e_{i,0}$ denoting the bottommost edge. The graph $G$ decomposes the plane into regions. Every bounded region is bounded by two paths that start at the same vertex $i$. We call such a bounded region a \defn{bubble}, and label it by~$p_{i,j}$, where \(e_{i,j}\) is the upper boundary path's first edge (so, $j\neq0$). See \Cref{fig_graph_poset} for an example.

\begin{figure}[ht]
\centering
\includegraphics[]{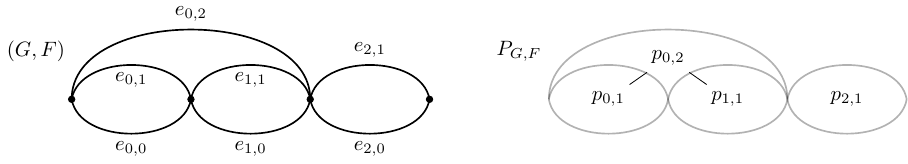}
\caption{A plane framed graph and its associated poset.}
\label{fig_graph_poset}
\end{figure}

Let $P_{G,F}$ denote the set of bubbles. For two bubbles $p,q\in P_{G,F}$, we write
\[
p\lessdot q
\]
if $p$ and $q$ share an edge and $p$ lies below $q$. The \defn{bubble poset} $(P_{G,F},\leq)$ is the poset on~$P_{G,F}$ obtained by taking the transitive closure of this relation.\footnote{We remark that $\lessdot$ is not necessarily the cover relation of the poset, as there may be bubbles $p,q,r\in P_{G,F}$ such that $ p<q<r, $ while both $p\lessdot q$ and $p\lessdot r$.}

Under these conventions, each route $R$ of $G$ determines the order ideal $I_R$ consisting of the bubbles lying below it. This gives a bijective correspondence between the routes of $G$ and the order ideals of the poset $P_{G,F}$. Moreover, maximal cliques correspond to maximal chains of order ideals
\[
I_0\subset I_1\subset \dots \subset I_d,
\]
or, equivalently, to linear extensions $q_1<q_2<\dots<q_d$ of the poset $P_{G,F}$, where $q_j$ is the unique element of $I_j\setminus I_{j-1}$. In other words, maximal cliques can be constructed by starting with the bottommost route and recursively constructing the next route by enlarging the corresponding order ideal by one bubble at a time.

\begin{figure}[ht]
    \centering
    \includegraphics[]{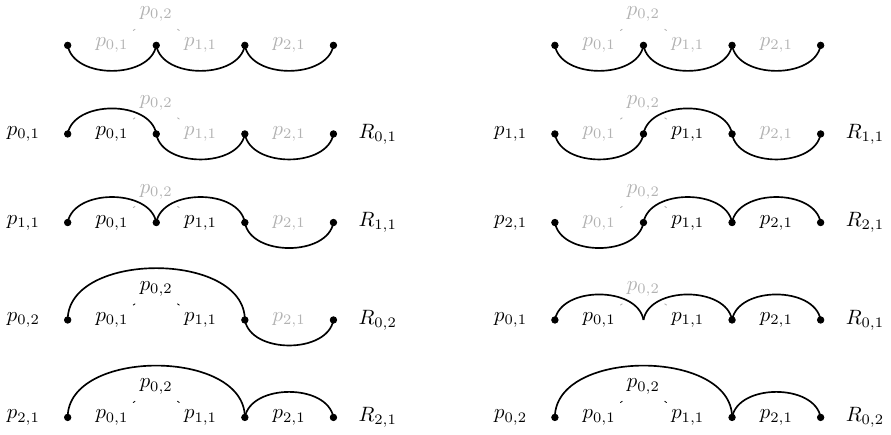}
    \caption{Two maximal cliques and their corresponding linear extensions.}
    \label{fig_two_linear_intervals}
\end{figure}
\Cref{fig_two_linear_intervals} illustrates two examples of maximal cliques, corresponding to the linear extensions
\[
p_{0,1}<p_{1,1}<p_{0,2}<p_{2,1},
\]
and
\[
p_{1,1}<p_{2,1}<p_{0,1}<p_{0,2},
\]
respectively.

For a maximal clique $C$ and an edge $e_{i,j}$ with $j\neq 0$, we denote by $R_{i,j}(C)$ the route of $C$ whose corresponding order ideal contains $p_{i,j}$ and all elements preceding it in the associated linear extension, see~\Cref{fig_two_linear_intervals} for several examples. Equivalently, when constructing $C$ by adding bubbles from bottom to top, $R_{i,j}(C)$ is the first route containing the edge $e_{i,j}$. We denote by $R_{i,j}^{\downarrow}(C)$ the route immediately preceding it. When the maximal clique is clear from the context, we often write $R_{i,j}=R_{i,j}(C)$ and $R_{i,j}^{\downarrow}=R_{i,j}^{\downarrow}(C)$, omitting $C$ for simplicity.

\subsection{Vertex coordinates}
\label{sec_vertex_coordinates}

Our next goal is to give a simple description of the vertex coordinates of the tropical framingtope $\mathscr{P}_{G,F}^h$. Let $\mathbf{x}=\mathbf{x}(C)$ be the vertex corresponding to a maximal clique $C$.

Recall that the coordinates
\[
\mathbf{x}=(x_{i,j})_{(i,j)\in I}
\]
are indexed by
\[
I=\{(i,j):1\leq i\leq n-2,\;1\leq j\leq j_i\},
\]
so that there is one coordinate for each edge $e_{i,j}$ with $1\leq i\leq n-2$ and $1\leq j\leq j_i$.

For a route $R$, let $I_R$ denote the set of indices $(k,l)\in I$ such that $e_{k,l}\in R$. For $i\neq 0$, the routes $R_{i,j}=R_{i,j}(C)$ and $R_{i,j}^{\downarrow}=R_{i,j}^{\downarrow}(C)$ share the same source edge. Hence, the corresponding linear forms in \eqref{eq_framingtope_trop_hypersurfaces} attain the minimum simultaneously, and therefore satisfy
\begin{equation}
-h(R_{i,j})+
\sum_{(k,l)\in I_{R_{i,j}}}x_{k,l}
=
-h(R_{i,j}^{\downarrow})+
\sum_{(k,l)\in I_{R_{i,j}^{\downarrow}}}x_{k,l}.
\label{eq_bubble}
\end{equation}

In this equation, the variables $x_{k,l}$ corresponding to edges in the intersection $R_{i,j}\cap R_{i,j}^{\downarrow}$ cancel. The remaining variables correspond to the two bounding paths of the bubble~$p_{i,j}$: on the left-hand side, only the variables corresponding to the top bounding path remain, while on the right-hand side, those corresponding to the bottom bounding path remain. Since all but the first edge of the top bounding path are of the form $e_{\star,0}$, the left-hand side reduces to $x_{i,j}$; see \Cref{fig_general_ordered}.

\begin{figure}[ht]
\centering
\includegraphics[scale=1.15]{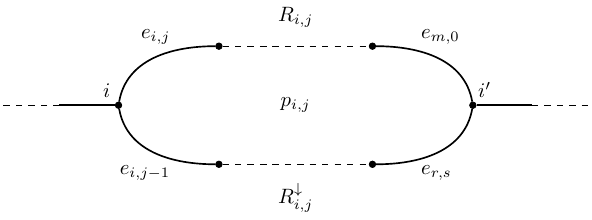}
\caption{The local structure of the routes $R_{i,j}$ and $R_{i,j}^{\downarrow}$ within the bubble $p_{i,j}$.}
\label{fig_general_ordered}
\end{figure}

Let $D(p_{i,j})$ denote the set of indices $(k,l)$ with $l\neq 0$ corresponding to edges of the bottom bounding path of the bubble $p_{i,j}$. Then we obtain the following.

\begin{lemma}\label{lem_coordinates_recursion}
Let $(G,F)$ be a plane framed graph and let $C$ be a maximal clique. Then
\[
x_{i,j}
=
h(R_{i,j})-h(R_{i,j}^{\downarrow})
+\sum_{(k,l)\in D(p_{i,j})}x_{k,l}.
\]
\end{lemma}

\begin{proof}
This follows from~\eqref{eq_bubble} and the discussion above.
\end{proof}

\begin{example}
Let $(G,F)$ be the plane framed graph from~\Cref{fig_graph_poset}. In this case, the index set is
\[
I=\{(1,1),(2,1)\},
\]
so there are only two variables,
\[
x_{1,1}, \qquad x_{2,1},
\]
corresponding to the edges $e_{1,1}$ and $e_{2,1}$, respectively.

Let $C$ be the maximal clique on the left of \Cref{fig_two_linear_intervals}. In this case, we have
\[
R_{1,1}^{\downarrow}=R_{0,1},
\qquad
D(p_{1,1})=\emptyset,
\]
and
\[
R_{2,1}^{\downarrow}=R_{0,2},
\qquad
D(p_{2,1})=\emptyset.
\]
Therefore,
\[
x_{1,1}=h(R_{1,1})-h(R_{0,1}),
\]
and
\[
x_{2,1}=h(R_{2,1})-h(R_{0,2}).
\]
\end{example}

The previous example was relatively simple because the sets $D(p_{i,j})$ were empty. When this is not the case, we apply \Cref{lem_coordinates_recursion} recursively, removing bubbles below $p_{i,j}$ one at a time and thereby eliminating the new variables $x_{k,l}$ that appear at each step. This process terminates when no further variables appear. We illustrate this procedure in the next example.

\begin{figure}[ht]
    \centering
    \includegraphics[scale=1.2]{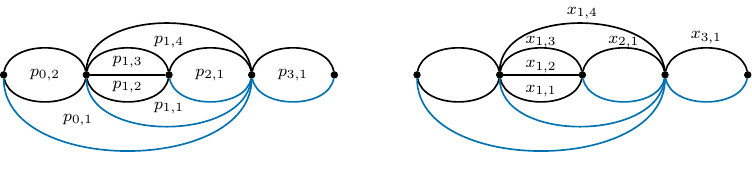}
    \caption{A plane framed graph $(G,F)$ with edges of the form $e_{\star,0}$ highlighted in blue.}
    \label{fig_example_coordinates_ordered}
\end{figure}

\begin{example}\label{ex_coordinates_recursion}
Let $(G,F)$ be the plane framed graph in \Cref{fig_example_coordinates_ordered}, and let $\mathbf{x}=\mathbf{x}(C)$ be the vertex corresponding to a maximal clique $C$. As above, write $R_{i,j}=R_{i,j}(C)$ and $R_{i,j}^{\downarrow}=R_{i,j}^{\downarrow}(C)$. Independently of the choice of $C$, we can compute the coordinate $x_{1,4}$ recursively as follows:
\[
x_{1,4}
= h(R_{1,4})-h(R_{1,4}^{\downarrow}) + x_{1,3}+x_{2,1},
\]
where
\[
x_{1,3}
= h(R_{1,3})-h(R_{1,3}^{\downarrow}) + x_{1,2},
\qquad
x_{2,1}
= h(R_{2,1})-h(R_{2,1}^{\downarrow}),
\]
and
\[
x_{1,2}
= h(R_{1,2})-h(R_{1,2}^{\downarrow}) + x_{1,1},
\qquad
x_{1,1}
= h(R_{1,1})-h(R_{1,1}^{\downarrow}).
\]
Combining these equations, we obtain
\[
\begin{aligned}
x_{1,4}
={}&\bigl(h(R_{1,4})-h(R_{1,4}^{\downarrow})\bigr)
+\bigl(h(R_{1,3})-h(R_{1,3}^{\downarrow})\bigr)\\
&+\bigl(h(R_{1,2})-h(R_{1,2}^{\downarrow})\bigr)
+\bigl(h(R_{1,1})-h(R_{1,1}^{\downarrow})\bigr)\\
&+\bigl(h(R_{2,1})-h(R_{2,1}^{\downarrow})\bigr).
\end{aligned}
\]
More concisely, $x_{1,4}$ is the sum of $h(R_{k,l})-h(R_{k,l}^{\downarrow})$ over all indices $(k,l)$ such that $p_{k,l}\leq p_{1,4}$ and $p_{k,l}$ can be reached from $p_{1,4}$ by moving downward without crossing an edge of the form $e_{\star,0}$. Thus, $p_{0,1}$ is disregarded, even though $p_{0,1}\leq p_{1,4}$. The edges of the form $e_{\star,0}$ are highlighted in blue in \Cref{fig_example_coordinates_ordered}.
\end{example}

The preceding example suggests the following general description of the coordinates.

\begin{theorem}\label{thm_vertex_coordinates}
Let $(G,F)$ be a plane framed graph and let $\mathbf{x}=\mathbf{x}(C)$ be the vertex corresponding to a maximal clique $C$. Then
\[
x_{i,j}
=
\sum_{(k,l)\in\overline{D}(p_{i,j})}
\bigl(h(R_{k,l})-h(R_{k,l}^{\downarrow})\bigr),
\]
where $\overline{D}(p_{i,j})$ is the set of indices $(k,l)$ such that $p_{k,l}\leq p_{i,j}$ and $p_{k,l}$ can be reached from~$p_{i,j}$ by moving downward without crossing an edge of the form $e_{\star,0}$.
\end{theorem}

\begin{proof}
The proof follows the recursive procedure described in \Cref{ex_coordinates_recursion}. Starting with $x_{i,j}$, apply \Cref{lem_coordinates_recursion} recursively to each variable that appears on the right-hand side. Every variable $x_{k,l}$ with $(k,l)\in\overline{D}(p_{i,j})$ appears at some stage of this procedure. The process terminates when the bottom bounding paths of the remaining bubbles consist entirely of edges of the form $e_{\star,0}$. For these bubbles, $D(p_{k,l})=\emptyset$, so no further variables appear. Consequently, each $(k,l)\in\overline{D}(p_{i,j})$ contributes exactly
\[
h(R_{k,l})-h(R_{k,l}^{\downarrow}),
\]
which gives the stated formula.
\end{proof}

\subsection{A nice height function and explicit realizations}
\label{sec_nice_height}

\Cref{thm_vertex_coordinates} expresses the vertex coordinates in terms of an admissible height function $\h$, but we have not yet specified a particular choice of $\h$. In this subsection, we identify a particularly simple and natural one: the function that counts the routes lying weakly below a given route. To ensure that this counting function is well-defined and admissible, we introduce a class of framed graphs for which these properties hold.

Recall from~\Cref{def_top_bot} that two routes $R_1$ and $R_2$ are \emph{incoherent along exactly one path} $P\subseteq R_1\cap R_2$ if $R_1$ and $R_2$ are incoherent at every vertex $v\in P$ and nowhere else, and that for such routes their top $\Top(R_1,R_2)$ and bottom $\Bot(R_1,R_2)$ parts are well defined.

\begin{definition}[ABBE graph]
\label{def_ABBE}
Let $(G,F)$ be a framed graph. Assume that $G$ is embedded in the plane with vertices located from left to right in increasing order, with all edges oriented monotonically from left to right, and with the framing $F$ induced by the drawing. We say that $(G,F)$ is an \defn{ABBE graph} (\textbf{AB}ove-\textbf{BE}low) if, for any two routes $R_1$ and $R_2$ that are incoherent along exactly one path, the route $\Bot(R_1,R_2)$ lies \defn{geometrically below} $\Top(R_1,R_2)$. That is, on every vertical line intersecting both routes, the intersection point with $\Bot(R_1,R_2)$ lies below the intersection point with $\Top(R_1,R_2)$. See~\Cref{fig_ABBE} for an illustration.
\end{definition}

\begin{figure}[ht]
    \centering
    \includegraphics[]{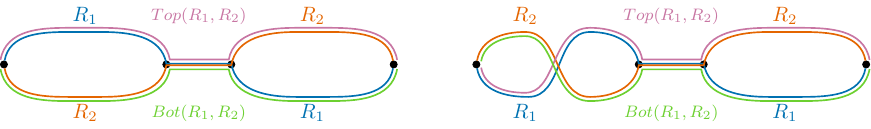}
    \caption{Illustration of the ABBE graph property. In the first example (left one), $\Bot(R_1,R_2)$ is geometrically below $\Top(R_1,R_2)$, whereas in the second example (right one) it is not.}
    \label{fig_ABBE}
\end{figure}

\begin{proposition}[Route-counting height function]
\label{prop_ABBE_height_admissible}
Let $(G,F)$ be an ABBE graph, and let $\h(R)$ denote the number of routes $R'$ of $G$ lying geometrically below $R$, including $R$ itself. Then $\h$ is an admissible height function for the framed triangulation $\Delta_{G,F}$.
\end{proposition}

\begin{proof}
By \Cref{lemma_admisible_heigth_function}, it suffices to show that, for any two routes $R_1$ and $R_2$ that are incoherent along exactly one path,
\[
\h(\Top(R_1,R_2))+\h(\Bot(R_1,R_2))
>
\h(R_1)+\h(R_2).
\]
Both sides count routes of $G$ according to the regions in which they lie relative to the four routes involved. By the ABBE condition, every route lying geometrically below $\Bot(R_1,R_2)$ is counted twice on both sides. Thus, it suffices to compare the routes in the region between $\Bot(R_1,R_2)$ and $\Top(R_1,R_2)$.

In this region, every route below $R_1$ or $R_2$ is counted exactly once on the right-hand side. On the other hand, every such route is counted exactly once by $\h(\Top(R_1,R_2))$. In addition, $\Top(R_1,R_2)$ itself is counted by $\h(\Top(R_1,R_2))$, but by neither $\h(R_1)$ nor $\h(R_2)$. Hence the left-hand side counts at least one more route than the right-hand side, proving the desired inequality.
\end{proof}

The following proposition shows that the ABBE condition is automatically satisfied for plane framed graphs, so the preceding construction applies in this setting.

\begin{proposition}\label{prop_plane_framings_are_ABBE}
    Plane framed graphs are ABBE graphs.
\end{proposition}

\begin{proof}
It suffices to show that if $R_1$ and $R_2$ are two routes that are incoherent along exactly one path $P$, then $\Bot(R_1,R_2)$ lies geometrically below $\Top(R_1,R_2)$. Otherwise, the two routes would have to cross along an edge outside $P$, contradicting the planarity condition.
\end{proof}

\begin{example}
Let us consider an example of a graph with a natural framing that is not an ABBE graph.

Let $G$ be the graph shown in \Cref{fig_amply_framed}, with the natural framing induced by the drawing, and let $R_1$ and $R_2$ be the two routes shown in the figure. The routes $R_1$ and $R_2$ are incoherent along exactly one path, consisting in this case of a single vertex marked in red in the figure. However, $\Bot(R_1,R_2)$ is not geometrically below $\Top(R_1,R_2)$.

For the route-counting height function $\h$, we have
\[
\h(R_1)=6,\qquad
\h(R_2)=2,\qquad
\h(\Top(R_1,R_2))=4,\qquad
\h(\Bot(R_1,R_2))=3.
\]
Thus,
\[
\h(\Top(R_1,R_2))+\h(\Bot(R_1,R_2))
=7
<8
=\h(R_1)+\h(R_2),
\]
so the admissibility criterion in \Cref{lemma_admisible_heigth_function} is violated. In particular, the route-counting height function is not admissible for this framed triangulation.

\begin{figure}[ht]
    \centering
    \includegraphics[scale=0.8]{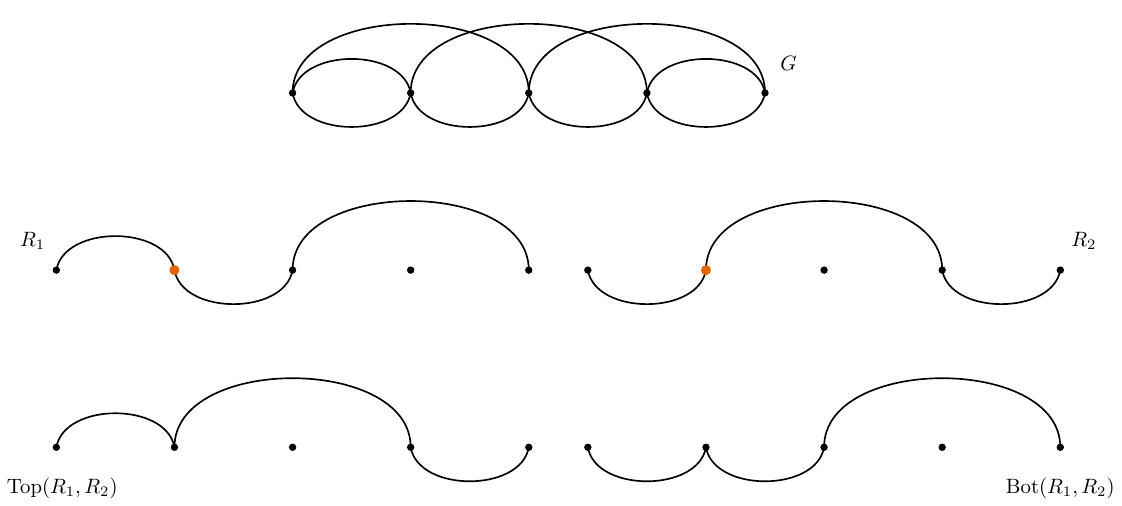}
    \caption{A non-ABBE graph with routes $R_1$, $R_2$, $\Top(R_1,R_2)$, and~$\Bot(R_1,R_2)$.}
    \label{fig_amply_framed}
\end{figure}
\end{example}

\begin{remark}[Cross Tamari lattices]
In~\cite[Section~2.4.3]{vonBell_framing_2024}, a generalization of the caracol graph in connection with cross-Tamari lattices was considered. The resulting framed cross-Tamari caracol graphs are not plane framed graphs in general. However, it is easy to see that they are all ABBE graphs. Hence, the tropical geometric realizations of framingtopes obtained from the simple route-counting function $\h$ in~\Cref{prop_ABBE_height_admissible} apply to this subfamily of examples. In particular, they provide the first known geometric realizations of the Hasse diagrams of cross-Tamari lattices as edge graphs of polytopal complexes. Moreover, these realizations have particularly simple vertex coordinates, which are integers and are obtained directly from the route-counting height function.

\begin{figure}[ht]
    \centering
    \includegraphics[scale=1.2]{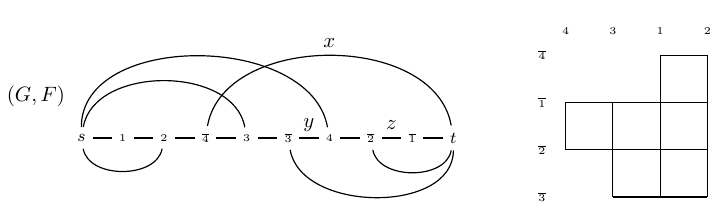}
    \caption{Example of a framed cross-Tamari caracol graph and its corresponding cross Tamari grid.}
    \label{fig_example_croostamari}
\end{figure}

\Cref{fig_example_croostamari} shows an example of a framed cross-Tamari caracol graph, reproduced from~\cite[Figure~44]{vonBell_framing_2024}. The corresponding framingtope is shown in~\Cref{fig_cross_framingtope}.

\begin{figure}[ht]
    \centering
    \includegraphics[scale=0.98]{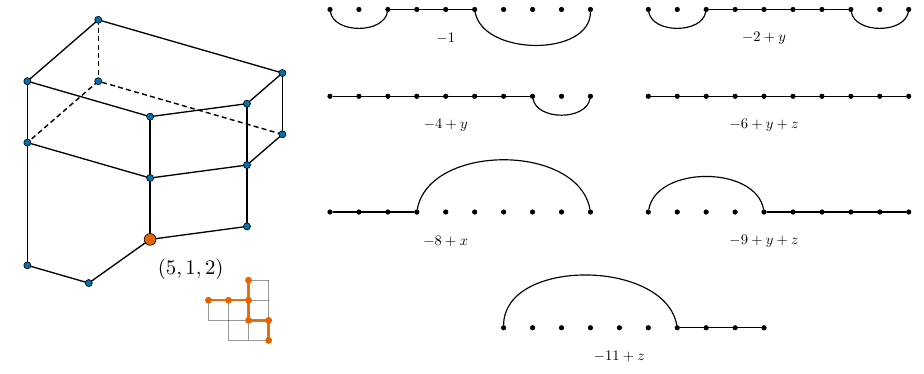}
    \caption{The framingtope of the framed cross-Tamari caracol graph in~\Cref{fig_example_croostamari}, associated to the route-counting function in~\Cref{prop_ABBE_height_admissible}, with one of its maximal cliques.}
    \label{fig_cross_framingtope}
\end{figure}

It is determined by the arrangement of tropical hypersurfaces
\[
\begin{aligned}
T_{0}^h &= \min\{-1,-2+y,-3+y+z,-4+x\},\\
T_{1}^h &= \min\{-2,-4+y,-6+y+z,-8+x\},\\
T_{2}^h &= \min\{-3,-6+y,-9+y+z\},\\
T_{3}^h &= \min\{-7,-11+z\}.
\end{aligned}
\]
It has 14 vertices with coordinates:
\[
(6,3,4),
(3,3,4),
(6,3,3),
(3,3,3),
(6,2,4),
(5,1,4),
(3,1,4),\]
\[
(6,2,3),
(5,1,3),
(3,1,3),
(6,2,2),
(5,1,2),
(4,1,1),
(3,1,1).
\]

We computed these coordinates using the grid model of cross Tamari lattices from~\cite[Section~2.4.3]{vonBell_framing_2024}. In this model, routes correspond to lattice points, and the height function of a lattice point is given by the number of lattice points in the grid lying weakly south-east of it. Each pair of lattice points in the same column determines an equality that the coordinates must satisfy. Solving these equalities for each cross Tamari tree yields the corresponding vertex coordinates.
\end{remark}

\subsection{The multioruga case}\label{sec_multipermutations}

In this section, we illustrate our construction for a special class of plane framed graphs, called \defn{multioruga graphs}, which were studied in~\cite{vonBell_framing_2024}. Their framing lattices provide a natural generalization of the classical weak order on permutations to multipermutations.

Fix $\mathbf{K}:=(k_0,\ldots,k_n)\in\mathbb{N}^{n+1}$. The \defn{multioruga graph} $\oru{\mathbf{K}}$ is the graph with $n+2$ vertices $v_0,\ldots,v_{n+1}$ and $k_0+\cdots+k_n+n$ edges, with $k_i+1$ edges from $v_i$ to $v_{i+1}$ for each $i\in\{0,\ldots,n\}$. This graph admits a natural crossing-free embedding in the plane, as illustrated in \Cref{fig_multioruga_11223}, and we consider the framing~$F$ induced by this embedding. For each $i\in\{0,\ldots,n\}$, we label the outgoing edges of $v_i$ by
\[
e_{i,0},e_{i,1},\dots,e_{i,k_i},
\]
ordered from bottom to top, as shown in~\Cref{fig_multioruga_11223}.

\begin{figure}[ht]
    \centering
    \includegraphics[]{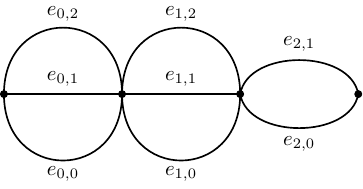}
    \caption{The multioruga graph $\oru{\mathbf{K}}$ for $\mathbf{K}=(2,2,1)$.}
    \label{fig_multioruga_11223}
\end{figure}

Our goal is to give an explicit description of the coordinates of the corresponding framingtope using the route-counting height function from \Cref{prop_ABBE_height_admissible}. For this purpose, we use the combinatorial description of the maximal cliques in terms of multipermutations given in~\cite{vonBell_framing_2024}, which we now recall.

Let $S(\mathbf{K})$ be the multiset containing the number $i$ exactly $k_i$ times, for $i\in\{0,\ldots,n\}$. The \defn{multipermutations of type $\mathbf{K}$} are the permutations of the multiset $S(\mathbf{K})$. We denote the set of all such multipermutations by $\mathfrak{S}_{\mathbf{K}}$ and identify its elements with words. For $\mathbf{K}=(2,1,1)$, we have
\[
\mathfrak{S}_{\mathbf{K}}
=
\{0012,0021,0102,0120,0201,0210,
1002,1020,1200,2001,2010,2100\}.
\]

Multipermutations of type $\mathbf{K}$ are in bijection with the maximal cliques of the multioruga graph $\oru{\mathbf{K}}$ with its natural framing. Moreover, the corresponding framing lattice is the \defn{weak order} on $\mathfrak{S}_{\mathbf{K}}$, defined as the transitive closure of the relation generated by adjacent transpositions

$$
\cdots ij\cdots \;\longrightarrow\; \cdots ji\cdots
\qquad\text{for } i<j.
$$

\begin{figure}[ht]
    \centering
    \includegraphics[]{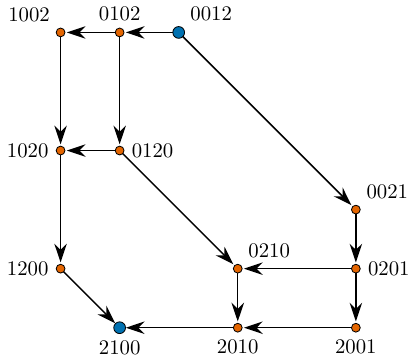}
    \caption{The Hasse diagram of the weak order on$\mathfrak{S}_{(2,1,1)}$. It is obtained as an oriented edge graph of the framingtope of the multioruga graph $\oru{(2,1,1)}$ in \Cref{fig_Trop_3_squares,fig_trop_max_clique}.}
    \label{fig_weak_order_211}
\end{figure}

\Cref{fig_weak_order_211} shows an example of the weak order on $\mathfrak{S}_{(2,1,1)}$, which is the framing lattice of $\oru{(2,1,1)}$ with its natural framing. Its Hasse diagram can be obtained by orienting the edges of the corresponding framingtope in \Cref{fig_Trop_3_squares,fig_trop_max_clique}.

The bijection between multipermutations and maximal cliques can be described as follows. For $\delta\in\mathfrak{S}_{\mathbf{K}}$, the associated maximal clique $C_\delta$ is constructed recursively by reading the word $\delta$ from left to right. We start with the bottom route
\[
[e_{0,0},e_{1,0},\ldots,e_{n,0}]
\]
and add one route for each letter of $\delta$. Suppose that the $j$-th letter of $\delta$ is $l$. Then the associated route is obtained from the previous one by replacing the edge $e_{l,\star}$ with $e_{l,\star+1}$, where $\star$ is the number of occurrences of $l$ among the first $j-1$ letters of $\delta$. The set of all resulting routes is the maximal clique $C_\delta$ associated with $\delta$. \Cref{fig_max_clique_12001} shows an example for $\mathbf{K}=(2,2,1)$ and $\delta=12001$.

\begin{figure}[ht]
   \centering
    \includegraphics[scale=0.7]{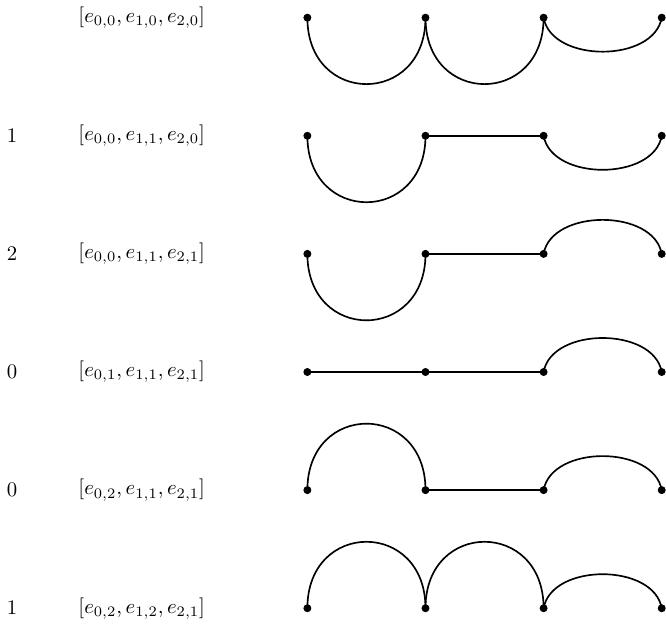}
    \caption{The maximal clique $C_\delta$ associated with $\delta=12001$. Reading~$\delta$ from left to right, each letter determines the edge that is replaced to obtain the next route.} \label{fig_max_clique_12001}
\end{figure}

For $i\in\{0,\dots,n\}$ and $j\in\{1,\dots,k_i\}$, the route $R_{i,j}=R_{i,j}(C_\delta)$, defined as the first route in $C_\delta$ containing the edge $e_{i,j}$, is precisely the route corresponding to the $j$th occurrence of the letter $i$ in $\delta$. For $k\in \{0,\dots,n\}$, define
\begin{equation}\label{eq_counting_letters}
a_{i,j}^k=a_{i,j}^k(\delta)
=
\#\{\text{letters $k$ occurring up to the $j$th letter $i$ in $\delta$}\}.
\end{equation}
Then
\[
R_{i,j}
=
[e_{0,a_{i,j}^0},e_{1,a_{i,j}^1},\dots,e_{n,a_{i,j}^n}].
\]
This allows us to explicitly evaluate the route-counting function from \Cref{prop_ABBE_height_admissible} for $R_{i,j}$ and its predecessor.

\begin{lemma}\label{lem_h_multioruga}
Let $C_\delta$ be the maximal clique associated to a multipermutation $\delta$. Let $R_{i,j}=R_{i,j}(C_\delta)$ and $R_{i,j}^{\downarrow}=R_{i,j}^{\downarrow}(C_\delta)$, and let $a_{i,j}^k=a_{i,j}^k(\delta)$ be as in \eqref{eq_counting_letters}. Then, for $i\in\{0,\dots,n\}$ and $j\in\{1,\dots,k_i\}$, the route-counting function $\h$ from \Cref{prop_ABBE_height_admissible} satisfies
\[
\h(R_{i,j})=\prod_{k=0}^n (a_{i,j}^k+1),
\]
and
\[
\h(R_{i,j}^{\downarrow})
=a_{i,j}^i\prod_{k\neq i}(a_{i,j}^k+1).
\]
\end{lemma}

\begin{proof}
The routes that lie geometrically below $R_{i,j}$ are precisely those of the form
\[
[e_{0,b_0},e_{1,b_1},\dots,e_{n,b_n}]
\]
where
\[
0\leq b_k\leq a_{i,j}^k
\qquad\text{for all }k\in\{0,\dots,n\}.
\]
Thus, for each $k$, there are $a_{i,j}^k+1$ possible choices for $b_k$, and the first formula follows.

Furthermore, $R_{i,j}^{\downarrow}$ is obtained from $R_{i,j}$ by replacing the edge $e_{i,a_{i,j}^i}$ by $e_{i,a_{i,j}^i-1}$. Hence, the routes lying geometrically below $R_{i,j}^{\downarrow}$ are obtained in the same way as those below~$R_{i,j}$, except that the index $b_i$ has only $a_{i,j}^i$ possible choices. Therefore,
\[
\h(R_{i,j}^{\downarrow})
=a_{i,j}^i\prod_{k\neq i}(a_{i,j}^k+1),
\]
as desired.
\end{proof}

\begin{corollary}\label{cor_h_difference_multioruga}
Under the same assumptions as in~\Cref{lem_h_multioruga}, we have
\[
\h(R_{i,j})-\h(R_{i,j}^{\downarrow})
=
\prod_{k\neq i}(a_{i,j}^k+1).
\]
\end{corollary}

\begin{proof}
The result follows immediately by subtracting the two equations in \Cref{lem_h_multioruga}.
\end{proof}

The next corollary gives a simple description of the vertex coordinates of the framingtope for the multioruga graph.

\begin{corollary}\label{cor_multioruga_coordinates}
Let $\h$ be the route-counting function from~\Cref{prop_ABBE_height_admissible} for the multioruga graph $\oru{(k_0,k_1,\dots,k_n)}$ with its natural framing. The corresponding framingtope has one vertex for each multipermutation $\delta$, with coordinates $\mathbf{x}(\delta)=(x_{i,j})$, where $i\in\{1,\dots,n\}$ and $j\in\{1,\dots,k_i\}$, given by
\[
x_{i,j}
=
\sum_{p=1}^{j}
\prod_{k\neq i}(a_{i,p}^k+1),
\]
where $a_{i,j}^k=a_{i,j}^k(\delta)$ is defined as in \eqref{eq_counting_letters}.
\end{corollary}

\begin{proof}
This follows from the vertex coordinate description in~\Cref{thm_vertex_coordinates} together with \Cref{cor_h_difference_multioruga}.
\end{proof}

\begin{example}
    Let $\delta=12001$ be a multipermutation in $\mathfrak{S}_{\mathbf{K}}$ for $\mathbf{K}=(2,2,1)$. Applying \Cref{cor_multioruga_coordinates}, we compute the corresponding coordinate
    \[
        \mathbf{x}(\delta)=(x_{1,1},x_{1,2},x_{2,1}).
    \]
    The first entry is
    \[
        x_{1,1}=(0+1)(0+1)=1,
    \]
    since there are zero $0$'s and zero $2$'s to the left of the first $1$ in $\delta$.

    For the second entry, we have
    \[
        x_{1,2}=(0+1)(0+1)+(2+1)(1+1)=7.
    \]
    The first term is precisely $x_{1,1}$, while the second term accounts for the two $0$'s and one $1$ to the left of the second $1$.

    Finally,
    \[
        x_{2,1}=(0+1)(1+1)=2,
    \]
    since there are zero $0$'s and one $1$ to the left of the first $2$.

    Thus, the coordinate of $\delta$ in the tropical realization of the framingtope is
    \[
        \mathbf{x}(\delta)=(1,7,2).
    \]
    Computing the coordinates for all multipermutations and drawing a segment between the coordinates corresponding to multipermutations that are adjacent in the weak order on $\mathfrak{S}_{\mathbf{K}}$, we obtain \Cref{fig_framingtope_11223}.

    \begin{figure}[ht]
        \centering
        
        \includegraphics[scale=0.6]{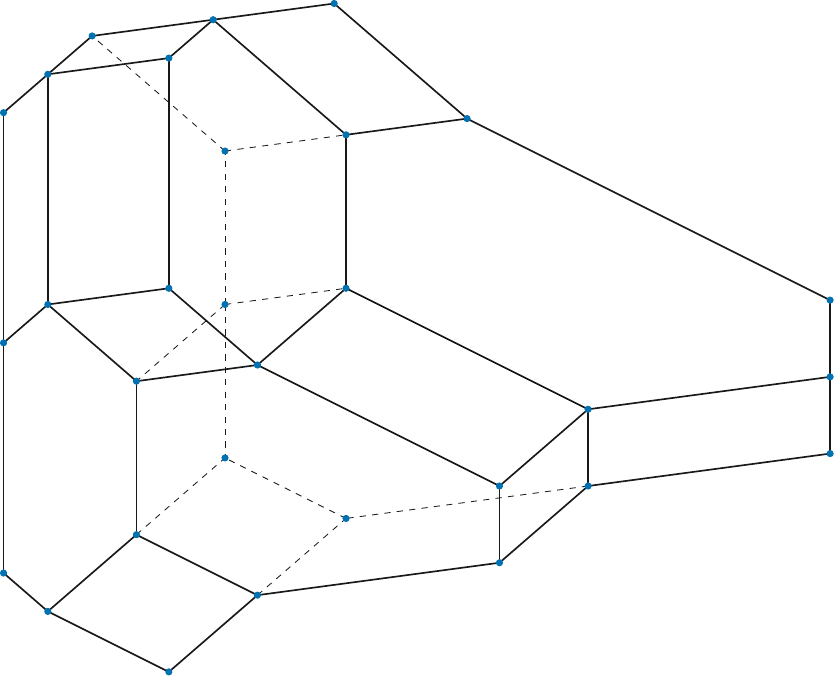}
        \caption{Bounded $1$-skeleton of the tropical framingtope of $\oru{\mathbf{K}}$ with the natural framing, for $\mathbf{K}=(2,2,1)$.}
        \label{fig_framingtope_11223}
    \end{figure}
\end{example}

\begin{remark}
    The oruga case corresponds to the multioruga graph $\oru{(1,\dots,1)}$ where each letter from $0$ to $n$ appears exactly once. In this scenario, its framingtope is a tropical realization of the permutahedron. For a permutation $\pi$ of the set $\set{0, \ldots, n}$, the corresponding coordinate $\mathbf{x}=(x_1,\ldots,x_n)$ is given by $x_i=2^{\pi^{-1}(i)}$, where $\pi^{-1}(i)$ denotes the number of letters preceding $i$ in the permutation $\pi$. Hence, the framingtope is the image of the polytope whose vertices are all permutations of $(1,2,4,\dots,2^n)$ under the projection $\mathbb{R}^{n+1} \rightarrow \mathbb{R}^n$ given by $(x_0,x_1\dots,x_n) \rightarrow (x_1,\dots,x_n)$.
\end{remark}

The multioruga examples illustrate how the tropical construction of framingtopes can lead to explicit and particularly simple geometric realizations. In the case of the ordinary oruga, the resulting coordinates recover a tropical realization of the permutahedron, while for general multioruga graphs they provide analogous realizations for the weak orders on multipermutations. Thus, the route-counting height function not only gives a uniform tropical construction of framingtopes for ABBE graphs, but also produces concrete integer coordinates in relevant families of examples.

\printbibliography 

\end{document}